\documentclass[10pt]{amsart}
\usepackage[foot]{amsaddr}
\usepackage{cancel}
\usepackage[british]{babel}
\usepackage{lmodern}
\usepackage{pifont,subfigure,graphicx}
\usepackage{caption}
\usepackage[colorlinks,citecolor=blue,urlcolor=blue]{hyperref}
\usepackage{url}
\usepackage{amsthm,amsmath,amssymb,amsfonts,mathrsfs,amsfonts,amstext,amsopn}
\usepackage{mathtools}
\usepackage{tikz}
\usepackage[version=3]{mhchem}
\usepackage{texdraw}
\usepackage{extarrows}
\usepackage{bbm}
\usepackage{enumitem}
\usepackage[all]{xy}
\usepackage{tikz}
\usepackage{tikz-cd}
\usepackage{pgfplots}
\usepackage{tikz-3dplot}
\usetikzlibrary{shapes.geometric, arrows.meta, decorations.markings}

\usepackage{pgfplots}
\usetikzlibrary{topaths,calc}
\pgfplotsset{compat = newest}
\usetikzlibrary{decorations.fractals}
\usetikzlibrary{positioning}
\usetikzlibrary{arrows}
\numberwithin{equation}{section}
\theoremstyle{plain}
\newtheorem{theorem}{Theorem}[section]

\newtheorem{proposition}[theorem]{Proposition}

\theoremstyle{definition}
\newtheorem{definition}[theorem]{Definition}
\newtheorem{example}[theorem]{Example}
\theoremstyle{remark}
\newtheorem{remark}[theorem]{Remark}

\newcommand\sbullet[1][.5]{\mathbin{\vcenter{\hbox{\scalebox{#1}{$\bullet$}}}}}

\newcommand{\lt}{\left}
\newcommand{\rt}{\right}

\newcommand{\e}{\text{e}}

\newcommand{\R}{\mathbb{R}}

\newcommand{\N}{\mathbb{N}}
\newcommand{\T}{\mathbb{T}}
\newcommand{\cA}{\mathcal{A}}
\newcommand{\cB}{\mathcal{B}}
\newcommand{\cC}{\mathcal{C}}
\newcommand{\cI}{\mathcal{I}}

\newcommand{\cH}{\mathcal{H}}
\newcommand{\cK}{\mathcal{K}}
\newcommand{\cM}{\mathcal{M}}
\newcommand{\cP}{\mathcal{P}}
\newcommand{\cS}{\Phi}

\newcommand{\cV}{\mathcal{V}}
\newcommand{\rd}{\mathrm{d}}

\newcommand{\supp}{{\rm supp\,}}
\newcommand{\Diam}{\textrm{Diam\,}}

\newcommand{\sT}{\mathbb{T}}
\newcommand{\abs}{{\sf abs}}
\newcommand{\Lip}{\mathcal{L}}
\newcommand{\BL}{\mathcal{BL}}

\newcommand{\TV}{{\sf TV}}

\allowdisplaybreaks
\begin{document}
\title[VE on DHGM]
  {Vlasov Equations on Directed Hypergraph Measures}
\author[Christian Kuehn and Chuang Xu]{Christian Kuehn $^{1,2}$}
\email{ckuehn@ma.tum.de (Christian Kuehn)}
\author{Chuang Xu $^{3}$}
\email{chuangxu@hawaii.edu (Chuang Xu)}
\address{$^1$
Department of Mathematics\\
Technical University of Munich, Munich\\
Garching bei M\"{u}nchen\\
85748, Germany.}
\address{$^2$
Munich Data Science Institute (MDSI)\\
Technical University of Munich, Munich\\
Garching bei M\"{u}nchen\\
85748, Germany.}
\address{$3$
Department of Mathematics\\
University of Hawai'i at M\={a}noa, Honolulu, Hawai'i\\
96822, USA
}

\date{\today}

\noindent

\begin{abstract}
In this paper we propose a framework to investigate the mean field limit (MFL) of interacting particle systems on directed hypergraphs. We provide a non-trivial measure-theoretic viewpoint and make extensions of directed hypergraphs as directed hypergraph measures (DHGMs), which are measure-valued functions on a compact metric space. These DHGMs can be regarded as hypergraph limits which include limits of a sequence of hypergraphs that are sparse, dense, or of intermediate densities. Our main results show that the Vlasov equation on DHGMs are well-posed and its solution can be approximated by empirical distributions of large networks of higher-order interactions. The results are applied to a Kuramoto network in physics, an epidemic network, and an ecological network, all of which include higher-order interactions. To prove the main results on the approximation and well-posedness of the Vlasov equation on DHGMs, we robustly generalize the method of [Kuehn, Xu. Vlasov equations on digraph measures, JDE, 339 (2022), 261--349] to higher-dimensions. 
 \end{abstract}


\maketitle
\section{Introduction}

Many science phenomena in diverse areas including epidemiology \cite{KMS19}, ecology \cite{P69}, physics \cite{N18}, social science \cite{HK02}, communication systems \cite{CS06}, etc., can be described as an interacting particle system (IPS), or equivalently, a dynamical system on networks \cite{N84,N18}. Classical works mainly focused on the dynamics of the IPS with interactions of rather simple types coupled on complete graphs \cite{N84}. More recently, studies have concentrated more on the complex dynamics caused by (i) the types of interaction,  and (ii) the heterogeneity of the underlying network \cite{B20}. In particular, certain IPS incorporating indirect interactions beyond pairwise interactions between two particles have been considered \cite{BBK20,BBK22}. For instance, in a deterministic chemical reaction network, the concentration of a species A depends on all reactions where A appears either in the reactant or in the product \cite{F19}. Such interactions are called \emph{higher order (or polyadic) interactions}~\cite{B20,T21,BGHS22}.

It is in general difficult to analyze an IPS analytically or numerically, due to the large size (e.g., the total species counts in a \emph{chemical reaction network} \cite{F19}, or the total population size of a population process \cite{MCO05}) of the IPS. One popular way to study these IPS by certain approximation is to use the so-called \emph{mean field analysis} \cite{S91}. For example, given a large system of \emph{indistinguishable} oscillators coupled on a given graph, one asks how to characterize the dynamics of a typical oscillator? It turns out that one can first sample the individual behavior of the oscillators on each node independently to derive a empirical distribution of the first $N$ oscillators, and then consider a \emph{weak limit} of the sequence of empirical distributions as $N\to\infty$. Such a limit, if it exists, is the so-called \emph{mean field limit} (MFL) \cite{S91,G13}. MFLs generally relate to a transport type PDE, the so-called Vlasov equation (VE), whose weak solution is the density of the MFL. Hence to explore the dynamics of the IPS of a large size, one can turn to study the dynamics of the VE.

Let us review some recent works, where the emphasis is the influence of the heterogeneity of the underlying graph on the dynamics of the network, in which case particles are distinguishable. Among all types of networks in epidemiology, ecology, social science, one archetypal example is the \emph{Kuramoto network}, a network of oscillators used to model diverse phenomena in e.g., physics, neuroscience, etc. \cite{N18}. For deterministic oscillator dynamics (particularly for deterministic Kuramoto network coupled on deterministic or random (e.g., Erd\H{o}s-R\'{e}nyi) graphs), the well-posedness and approximation of the MFL have been investigated \cite{KM18,CM19} for a Kuramoto network coupled on a sequence of heterogeneous graphs whose limit is a \emph{graphon}  (an integrable function on the square $[0,1]^2$) \cite{L12}. Based on the VE, one can then further study the bifurcations of the Kuramoto network \cite{CM19}. The results were further generalized to the situation, where the graph limit may not necessarily be a graphon, but e.g., a \emph{graphop} \cite{GK20}, a \emph{digraph measure} \cite{KX21}, or a limit of a sequence of sparse graphs\---a generalized graphon \cite{JPS21}. We point out that the technical approach in \cite{JPS21} is different from that in \cite{KX21}. In stochastic settings, MFL results were established for IPS based on systems of SDEs coupled also on heterogeneous graphs (particularly sparse graphs due to practical considerations) \cite{LRW19a,LRW19,ORS20}.

There are many network models with higher-order interactions, e.g., the Kuramoto-Sakaguchi network \cite{BAR16}, mathematical models based on random replicator dynamics \cite{B20}, three-player games \cite{B20}, Chua oscillator models \cite{G20}, among many others.  Effectively almost all of these models have been only been studied by direct numerical simulation or formal analytical techniques \cite{B20}. In contrast, rigorous theoretical results for MFLs of IPS with higher-order interactions are rare. The underlying topological structure for such systems are hypergraphs (see Figure~\ref{figure-hyper-graph} for the illustration of a hypergraph) \cite{MKJ20}. In \cite{BBK20}, the MFL of a Kuramoto network coupled on \emph{complete $k$-uniform} hypergraphs for finite $k\in\mathbb{N}$ was analyzed. In \cite{KH22}, \emph{quenched mean field approximation} (QMFA) of continuous time Markov chains on directed hypergraphs of uniformly bounded  ranks (for the definition of \emph{rank} of a hypergraph, we refer the reader to Definition~\ref{hyperdigraph}) was analyzed. It is noteworthy that the QMFA is also called \emph{$N$-intertwined mean field approximation (NIMFA)} \cite{VV15}. Such mean field approximation neglects \emph{dynamical correlation between nodes of the hypergraph}, which is different from the mean field limit. Our results on MFLs serve as a bridge leading to rigorous analysis of dynamics of relevant models, and provide a deeper insight into how  dynamics of a network of higher-order interactions depend on the heterogeneity of the underlying coupling graph or hypergraph.

Indeed, even from the viewpoint of graph theory, there exist limited references on hypergraph limits (a special class of hypergraph limits, is called \emph{ultralimit hypergraph} in \cite{ES12}, corresponding to the limit of dense hypergraphs) \cite{T06,G07,J08,ES12,L12,B13,Z15}, in comparison to the graph limits \cite{BS01,LS06,L12,KLS19,BS20}. It is noteworthy that hypergraph limits are defined as tensors in a recent work \cite{Z23}.

\begin{figure}[b]
\begin{center}
\begin{tikzcd}[column sep=5em,row sep=5em]
\mathbf{4} 
\arrow[dr,yellow,thick] \arrow[r,bend right=20,blue,thick] & \mathbf{3} \arrow[l,bend right=20,red,thick] \arrow[ld,bend right=20,green,thick] \arrow[d,bend right=20,blue,thick]\\
\mathbf{1}  
\arrow[ur,bend right=20,green,thick] \arrow[u,blue,thick] & \mathbf{2} \arrow[u,bend right=20,red,thick] \arrow[l,shift left=0ex,yellow] 
\end{tikzcd}
\end{center}
\caption{A directed hypergraph of $4$ vertices $\{1,2,3,4\}$ and rank $4$, two directed hyper-edges of cardinality 2 ($(1,3)$ and $(3,1)$), two directed hyper-edges of cardinality 3 ($(4,2,1)$ and $(2,3,4)$), and one hyper-edge of cardinality 4 ($(1,4,3,2)$).} 
\label{figure-hyper-graph}
\end{figure}
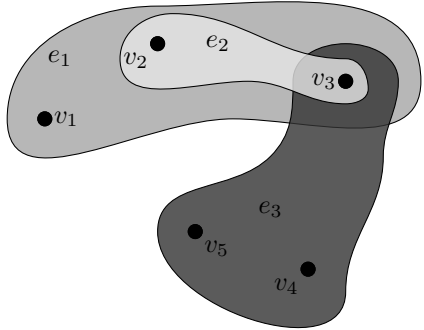

In this paper, we aim to propose a general framework to study the MFL of a class of IPS coupled on heterogeneous hypergraphs. For the ease of exposition, we hereby only consider dynamical systems on one hypergraph. Straightforward extension to a finitely many hypergraphs can be done, e.g., in a similar manner as showcased in \cite{KX21}.

Consider the following network dynamical system on a directed hypergraph of rank $k$: 
\begin{equation}\label{multi-hypergraph}\begin{split}
  \dot{\phi}_i^N=&h^N_i(t,\phi_i^N)+\sum_{\ell\in J}\frac{1}{N^{\ell-1}}\sum_{j_1=1}^N
  \cdots\sum_{j_{\ell-1}=1}^NW^{(\ell,N)}_{i,j_1,\ldots,j_{\ell-1}}
  g_{\ell}(t,\phi_i^N,\phi_{j_1}^N,\ldots,\phi_{j_{\ell-1}}^N),
\end{split}\end{equation}
where $J$, with $\max J=k$ independent of $N$, is a finite set of cardinalities of directed hyper-edges of cardinality $\ell$ of  the hypergraph $\mathcal{H}^{N}=(\cV^N,\mathcal{E}^{N})$ with $\cV^N$ identified with $[N]$, and for every $\ell\in J$,   
$g_{\ell}$ is the coupling function of an $\ell$-uniform sub hypergraph $\mathcal{H}^{\ell,N}=(\mathcal{V}^N,\mathcal{E}^{\ell,N})$  of $\cH$, and $W^{(\ell,N)}=(W^{(\ell,N)}_{i,j_1,\ldots,j_{\ell-1}})$ is the weight matrix associated with the $\ell$-uniform hypergraph $\mathcal{H}^{\ell,N}$;  $\phi_i^N\in\R^d$ stands for the evolutionary dynamics on the $i$-th node, and $h_i^N$ is the vector field for node $i$, for $i=1,\ldots,N$. 
Although such framework may not be suitable for modelling networks on a sequence of simplicial complexes with unbounded ranks \cite{ND22,RS23}, it already covers several important network dynamical models from different disciplines of sciences \cite{B20}. Furthermore, one may extend the results of this paper to investigate the MFL of IPS on Riemannian manifolds with higher-order interactions (see Section~\ref{sect-applications}).

To better reveal the essence of our main result (Theorem~\ref{th-approx}) on the MFL of \eqref{multi-hypergraph}, we consider a special model\---a one-dimensional Kuramoto network coupled on a \emph{higher-dimensional ring} which is a sum of two uniform hypergraphs of rank $2$ and $3$, respectively, where the associated weight matrices are given by 
$$W^{(1,N)}_{ij}=\begin{cases}
  1,\quad \text{if}\quad |i-j|=1\!\!\!\!\mod N,\\
  0,\quad \text{else},
\end{cases}\ W^{(2,N)}_{ijk}=\begin{cases}
  1,\quad \text{if}\quad |i-j|=|j-k|=1\!\!\!\!\mod N,\ k\neq i,\\
  0,\quad \text{else}.
\end{cases}$$ This is a \emph{sparse} hypergraph of rank $3$ (sparsity is specified in Definition~\ref{Def-dense-sparse}). For ease of exposition, let $h_i^N(t,u)\equiv h(t,u)$ for all $i=1,\ldots,N$, where some $h\in\mathcal{C}^1([0,T]\times\mathbb{T})$, $T>0$, and $N\in\mathbb{N}$. Let $X=\mathbb{T}$. It is readily verified that the weak limit of $W^{(\ell,N)}$ is a measure-valued function $\eta_{\ell}\colon x\mapsto\eta^x_{\ell}$ defined as follows for $\ell=1,2$: $$\eta_1^x=2\delta_x,\quad \eta_2^x=2\delta_{{(x,x)}},\quad \text{for}\quad x\in X.$$
Then the mean field limit of solutions of \eqref{multi-hypergraph} as $N\to\infty$ is an absolutely continuous w.r.t. the Lebesgue measure on the circle $\mathbb{T}$, and its density is the weak solution to the following so-called \emph{Vlasov equation} (VE):
\begin{align*}
&\frac{\partial\rho}{\partial t}(t,x,\phi)+\textrm{div}_{\phi}\left(\rho(t,x,\phi)\widehat{V}[\eta,\rho(\cdot),h](t,x,\phi)\right)=0, t\in(0,T],\ x\in X,\ \lambda\text{-a.e.}\ \phi\in \mathbb{T},\\
&\rho(0,\cdot)=\rho_0(\cdot),
  \end{align*}where
  \begin{equation*}
\begin{split}
\widehat{V}[\eta,\rho(\cdot),h](t,x,\phi)=&
\int_{X}\int_{\mathbb{T}}g_2(t,\phi,\psi_1)
\rho(t,y_{1},\psi_{1})\rd\psi_{1}\rd\eta^{x}_{1}(y_1)+\int_{X^2}\int_{\mathbb{T}}\int_{\mathbb{T}} g_3(t,\phi,\psi_1,\psi_2)\\
&\rho(t,y_{2},\psi_{2})
\rho(t,y_{1},\psi_{1})\rd\psi_2\rd\psi_{1}\rd\eta^{x}_{2}(y_1,y_2)+h(t,x,\phi).
    \end{split}
\end{equation*}
As we will see in Section~\ref{sect-applications}, one may obtain well-posedness as well as approximation of the MFL as an application of a  result for the general model \eqref{multi-hypergraph}  (Theorem~\ref{th-approx}).

In the following, we give a brief outline of the techniques to obtain the well-posedness as well as approximation for  model \eqref{multi-hypergraph}. First, we introduce the concept of \emph{directed hypergraph measures} (abbreviated as DHGM) as a generalization of directed hypergraphs. DHGM are measure-valued functions defined on a compact metric space. The definition of DHGM seems a natural extension of our definition \emph{digraph measure} introduced in \cite{KX21},  characterizing the limit of a sequence of directed graphs. We point out that such a generalization to hypergraphs is different from the hypergraph limits defined in the literature of combinatorics and graph theory \cite{ES12,L12}. In particular, we compare the difference between two DHGMs using uniform bounded Lipschitz metric, in contrast to the classical \emph{homeomorphism density} \cite{ES12,JPS21} used in the contexts of both graph limits and hypergraph limits. 
Then we introduce \emph{fiberized equation of characteristics} associated with \eqref{multi-hypergraph}. The idea of introducing such fiberized  equation of characteristics was motivated by \cite{KM18}. Using the classical theory of ODE, we obtain the well-posedness of the fiberized equation of characteristics. We define the generalized Vlasov equation as a \emph{fixed point equation} induced by the flow generated by the family of fiberized equations of characteristics.  This equation is a measure-valued differential equation (i.e., the measure at time $t$ is the push-forward of the initial measure under the flow of the equation of characteristics). Then in the light of the uniform bounded rank of DHGM, we obtain the well-posedness of the VE. To obtain the approximation of the MFL by empirical measures (atomic measures with equal mass on each point in the finite support), we rely on the recently developed result of uniform approximation of probability measures on Euclidean spaces \cite{XB19,C18,BJ22} (see also \cite{KX21}). 

Next, we outline the structure of the paper. In Section~\ref{sect-preliminaries}, we introduce the notation, provide some basics in measure theory, and define \emph{directed hypergraph measures} (DHGM) as a new generalization of directed hypergraphs with ample examples of independent interest. In Section~\ref{sect-Vlasov},  we define the \emph{fiberized equation of characteristics} associated with \eqref{multi-hypergraph}. In
 Section~\ref{sect-approximation}, we prove approximation of the MFL by a sequence of ODEs coupled on hypergraphs. 
Finally in Section~\ref{sect-applications}, we apply our results to study mean field limit of three network models of higher-order interactions in physics, epidemiology, and ecology.

\section{Preliminaries}\label{sect-preliminaries}
\subsection*{Notation}
Let $\R$ the set of real numbers and $\R_{+}$ the set of nonnegative real numbers. For $a\in\R$, let $\langle a\rangle\in[0,1)$ be its fractional part. The unit circle $\T$ is identified with $[0,1)$ via the natural projection $x\mapsto\langle x\rangle$ for $x\in\R$. For $i=1,2$, let $X_i$ be closed subsets of a finite dimensional Euclidean space.  We equip $X_i$  with the metric $d_i$ to make it a complete metric space. For instance $X_i$ can be a sphere, a torus, or any other compact subset of a Euclidean space. For any $k\in\N$ and $x\in\R^k$, let $[k]\coloneqq\{j\}_{j=1}^k$. For any measurable subset $A$ in a Euclidean space endowed with norm $|\cdot-\cdot|$, let $\text{dim}_H(A)$ denote the Hausdorff dimension of $A$, $\textrm{Diam}\ A\coloneqq\sup_{x,y\in A}|x-y|$ be its diameter; by convention, $\Diam A=0$ if the cardinality $\# A\le1$. We use $\lambda|_A$ to denote the uniform (probability) measure over $A$ whenever appropriate (e.g., when $A$ is either finite or Lebesgue measurable with a finite Lebesgue measure). We slightly abuse $\lambda$ (not restricted to any given set) for the Lebesgue measure on $\R^k$, where the dimension $k$ is deemphasized in the notation, without causing any ambiguity in the context. Let $\mathbbm{1}_A$ be the indicator function on $A$. Let $X$ be a Borel measurable space. For $k\in\mathbb{N}$, and for any Borel set $A\subseteq X^k$, let $\pi_1 A=\{z_1\in X\colon z\in A\}$ denote the projection of the collection of elements in $A$ onto their first coordinates in $X^k$. For any $a\in A$, let $\delta_x$ be the Dirac measure at the point $a\in A$. For two real-valued functions $f$ and $g$ on $X$, we denote $f\precsim g$ if there exists $C>0$ such that $ \frac{f(x)}{g(x)}\le C$, and $f\thicksim g$ if both $f\precsim g$ and $g\precsim f$.

\subsection*{Measure-valued functions}

Let $\cM_+(X_2)$ be the set of all finite Borel positive measures on $X_2$ and $\cP(X_2)\subseteq\cM_+(X_2)$ the space of all probabilities on $X_2$.  Given a reference measure $\mu_{X_2}\in\cP(X_2)$, let $\cM_{+,\abs}(X_2)\subseteq\cM_+(X_2)$ be the space of all absolutely continuous finite positive measures w.r.t. $\mu_{X_2}$. 

Let $\mathcal{B}(X_1,X_2)$ ($\mathcal{C}(X_1,X_2)$, $\mathcal{C}_{\sf b}(X_1,X_2)$, respectively) be the space of \emph{bounded} measurable functions (\emph{continuous} functions, \emph{bounded continuous} functions, respectively) from $X_1$ to $X_2$ equipped with the same uniform metric:
 \[d(f,g)=\sup_{x\in X_1}d_2(f(x),g(x)).\]Let $\mathcal{L}(X_1,X_2)\coloneqq\{g\in\mathcal{C}(X_1,X_2)\colon \mathcal{L}(g)\coloneqq\sup_{x\neq y}\frac{d_2(g(x),g(y))}{d_1(x,y)}<\infty\}$ be the space of Lipschitz continuous functions from $X_1$ to $X_2$. Hence $\mathcal{BL}(X_1,X_2)=\mathcal{B}(X_1,X_2)\cap\mathcal{L}(X_1,X_2)$ denotes the space of bounded Lipschitz continuous functions. In particular, when $X_2=\R$, we suppress $X_2$ in $\mathcal{B}(X_1,X_2)$ and simply write $\mathcal{B}(X_1)$. Similarly, we write $\mathcal{C}(X_1)$ for $\mathcal{C}(X_1,\R)$, etc. In this case, $\mathcal{B}(X_1)$, $\mathcal{C}(X_1)$ and $\mathcal{C}_{\sf b}(X_1)$ are all Banach spaces with the supremum norm $$\|f\|_{\infty}\coloneqq\sup_{x\in X_1}|f(x)|.$$
 Let $\mu_{X_1}\in\cP(X_1)$ be the reference measure on $X_1$. Let $\mathcal{BL}_1(X_2)=\{g\in \mathcal{BL}(X_2)\colon \mathcal{BL}(g)\coloneqq\|g\|_{\infty}+\mathcal{L}(g)\le1\}$ denote the set of bounded Lipschitz functions of a normalized constant. 

Define the bounded Lipschitz norm (on the space of all finite signed Borel measures):
\[\|\nu\|_{\sf BL}\coloneqq\sup_{f\in\mathcal{BL}_1(X_2)}\int_{X_1}f\rd\nu,\quad \nu\in\cM_+(X_2),\]
which induces the bounded Lipschitz distance: For $\nu^1,\ \nu^2\in\cM_+(X_2)$,
\[d_{\sf BL}(\nu^1,\nu^2)=\sup_{f\in\mathcal{BL}_1(X_2)}\int f(x)\mathrm{d}(\nu^1(x)-\nu^2(x)).\]
Recall that $(\cM_+(X_2),d_{\sf BL})$ is a complete metric space \cite{B07}.  Define the \emph{uniform bounded Lipschitz metric}:
\begin{equation}\label{fibermetricBL}
  d_{\infty}(\eta_1,\eta_2)=\sup_{x\in X}d_{\sf BL}(\eta_1^x,\eta_2^x),\quad \text{for}\quad \eta_1,\eta_2\in \mathcal{B}(X_1,\cM_+(X_2)).
\end{equation}
Hence $\mathcal{B}(X_1,\cM_+(X_2))$ and $\mathcal{C}_{{\sf b}}(X_1,\cM_+(X_2))$ equipped with the uniform bounded Lipschitz metric are complete. 

For every $\eta\in\cM_+(X_2)$, let $$\|\eta\|_{\TV}=\sup_{A\in\mathfrak{B}(X_2)}\eta(A)=\eta(X_2)$$ be the total variation \emph{norm} of $\eta$, where $\mathfrak{B}(X_2)$ is the Borel sigma algebra of $X_2$. For $\mathcal{B}(X_1,\cM_+(X_2))\ni\eta\colon x\mapsto\eta^x$, let
$$\|\eta\|\coloneqq\sup_{x\in X_1}\|\eta^x\|_{\TV}$$

Let $$\mathcal{B}_{*}(X_1,\cM_+(X_2))\coloneqq\left\{\eta\in\cB(X_1,\cM_+(X_2))\colon \int_{X_1}\eta^{x}(X_2)\rd \mu_{X_1}(x)=1\right\},$$ $$\mathcal{C}_{*}(X_1,\cM_+(X_2))\coloneqq\mathcal{B}_{*}(X_1,\cM_+(X_2))\cap \mathcal{C}(X_1,\cM_+(X_2)),$$

Let $\cI\subseteq\mathbb{R}$ be a compact interval and $k\in\N$. For $h\in\cC(\cI\times X_1\times X_2,\mathbb{R}^{k})$, let $$\|h\|_{\infty,\cI}\coloneqq\sup_{\tau\in\cI}\sup_{u\in X_1}\sup_{v\in X_2}|h(\tau,u,v)|.$$
For $\eta_{\cdot}\in\mathcal{C}(\mathcal{I},\mathcal{B}(X_1,\cM_+(X_2)))$, let $$\|\eta_{\cdot}\|=\sup_{t\in\cI}\sup_{x\in X_1}\|\eta^x_t\|_{\sf TV}$$ be the norm of the function $\eta_{\cdot}$

To construct solutions to Vlasov equations on DHGMs, we need to first define ``weak continuity'' of measure-valued functions \cite{KX21}.

\begin{definition}
For $i=1,2$, let $X_i$ be {closed subsets} of a finite dimensional Euclidean space. Assume $(X_1,\mathfrak{B}(X_1),\mu_{X_1})$ is a compact probability space.
For $$\mathcal{B}_{*}(X_1,\cM_+(X_2))\ni\eta\colon\begin{cases}
  X_1\to \cM_+(X_2),\\ x\mapsto\eta^x,
\end{cases}$$ $\eta$ is  {\emph{weakly $x$-continuous} (or \emph{weakly continuous} whenever the variable is clear from the context)} if for every $f\in\mathcal{C}_{\sf b}(X_2)$,
\[\mathcal{C}(X_1)\ni\eta(f)\colon\begin{cases}
  X_1\to\R,\\ x\mapsto\eta^x(f)\coloneqq\int_{X_2}f\rd\eta^x.
\end{cases}\]
\end{definition}

\begin{definition} For $i=1,2$, let $X_i$ be {closed subsets} of a finite dimensional Euclidean space. Assume $(X_1,\mathfrak{B}(X_1),\mu_{X_1})$ is a compact probability space.
Let $\mathcal{I}\subseteq\R$ be a compact interval. For $$\eta_{\cdot}\colon\begin{cases}
  \mathcal{I}\to \mathcal{B}(X_1,\cM_+(X_2)),\\ t\mapsto\eta_t,
\end{cases}$$ $\eta_{\cdot}$ is  {\emph{$x$-uniformly weakly $t$-continuous} (or \emph{uniformly weakly $t$-continuous} whenever the \emph{spatial} variable $x$ is clear from the context)} if for every $f\in \mathcal{C}_{\sf b}(X_2)$, $t\mapsto\eta_t^x(f)$ is continuous in $t$ uniformly in $x\in X_1$ {in the following sense:
\[\forall \epsilon>0,\quad \exists \sigma>0\quad \text{such that}\quad |t_1-t_2|<\sigma\implies \bigl|\eta_{t_1}^x(f)-\eta_{t_2}^x(f)\bigr|<\epsilon,\quad \forall x\in X_1\]}
\end{definition}
\begin{proposition}\label{prop-nu}
For $i=1,2$, let $X_i$ be {closed subsets} of a finite dimensional Euclidean space. {Assume $X_1$ and $X_2$ are both compact, and $(X_1,\mathfrak{B}(X_1),\mu_{X_1})$ is a probability space.} Let $\mathcal{I}\subseteq\R$ be a compact interval.
\begin{enumerate}
\item[\textnormal{(i)}] Let $\eta_{\cdot}\colon \mathcal{I}\to \mathcal{B}_{*}(X_1,\cM_+(X_2))$. Then $\eta_{\cdot}$ is {$x$-uniformly weakly $t$-continuous} if and only if $\eta_{\cdot}\in \mathcal{C}(\mathcal{I},\mathcal{B}_{*}(X_1,\cM_+(X_2)))$.
\item[\textnormal{(ii)}] Assume $\eta_{\cdot},\ \xi_{\cdot}\in \mathcal{C}(\mathcal{I},\mathcal{B}_{*}(X_1,\cM_+(X_2)))$, then
    $\|\eta_{\cdot}\|<\infty$ and $t\mapsto d_{\infty}(\eta_t,\xi_t)$ is continuous.
\item[\textnormal{(iii)}] Assume $\eta\in \mathcal{C}(X_1,\cM_+(X_2))$. Then $\eta$ is weakly $x$-continuous.
\item[\textnormal{(iv)}] Assume $\eta\in \mathcal{C}(X_1,\cM_+(X_2))$. Let $m\in\mathbb{N}$,
$$\eta(m) \colon X_1^m\ni(x_1,\ldots,x_{m})\mapsto\otimes_{j=1}^m\eta^{x_j}\in\cM_+(X_2^m).$$ Then $\eta(m)$ is weakly $x$-continuous.
\end{enumerate}
\end{proposition}
\begin{proof}
Items (ii)-(iii) follow directly from \cite[{Proposition~2.9}]{KX21}. The last statement (iv) can be proved similarly to (iii). We only prove (i). 

{\begin{enumerate}
\item[Step I.] $x$-uniform weak $t$-continuity implies continuity. Assume $\eta_{\cdot}$ is $x$-uniformly weakly $t$-continuous. Fix $t\in\mathcal{I}$ and $\mathcal{I}\ni t_j\to t$. Since $\eta_{\cdot}$ is $x$-uniformly weakly $t$-continuous, for every $f\in \mathcal{C}_{\sf b}(X_2)$, $$(\eta_{t_j})_{x}(f)\to(\eta_t)_{x}(f)\quad \text{uniformly in}\ x.$$
Since $X_2$ is a complete, separable metric space,  $d_{\sf BL}$ metrizes the weak-$*$ topology of $\cM_+(X_2)$ \cite[Thm.~8.3.2]{B07}. Using the supremum representation of $d_{\sf BL}$ and note that $\mathcal{BL}_1(X_2)\subseteq\mathcal{C}_{\sf b}(X_2)$, we have
\[\lim_{j\to\infty}d_{\sf BL}((\eta_{t_j})_x,(\eta_t)_x)=0,\]
and the convergence is uniform in $x\in X_1$. This means that
\[\lim_{j\to\infty}\sup_{x\in X_1}d_{\sf BL}((\eta_{t_j})_x,(\eta_t)_x)=0,\]
i.e., $$\lim_{j\to\infty}d_{\infty}(\eta_{t_j},\eta_t)=0,$$
which shows $\eta_{\cdot}\in\mathcal{C}(\mathcal{I},\mathcal{B}_{\mu_{X_1},1}(X_1,\cM_+(X_2)))$.
\item[Step II.] Continuity implies $x$-uniform weak $t$-continuity. Assume $\eta_{\cdot}\in\mathcal{C}(\mathcal{I},\mathcal{B}_{\mu_{X_1},1}(X_1,\cM_+(X_2)))$. For every fixed $t\in\mathcal{I}$ and $\mathcal{I}\ni t_j\to t$, we have
\begin{equation}\label{Eq-infinity-convergence}
\lim_{j\to\infty}d_{\infty}(\eta_{t_j},\eta_t)=0.
\end{equation} 
Let $f\in\mathcal{C}_{\sf b}(X_2)$. Since $X_2$ is compact, by \cite[Corollary~6.2.2]{CMN19} for every $\varepsilon>0$, there exists $M\ge1$ and $\tilde{f}\in\mathcal{BL}(X_2)$ such that $\mathcal{BL}(\tilde{f})\le M$ and 
\begin{equation}\label{Eq-Lipschit-approximation}
\|f-\tilde{f}\|_{\infty}<\frac{\varepsilon}{3(1+\|\eta_{\cdot}\|)}
\end{equation}
It follows from \eqref{Eq-infinity-convergence} that there exists $J\in\mathbb{N}$ such that for all $j\ge J$,
\begin{equation}\label{Eq-uniform-d-infinity-convergence}
\sup_{g\in\mathcal{BL}_1(X_2)}|\eta^x_{t_j}(g)-\eta^x_t(g)|<\frac{\varepsilon}{3M}
\end{equation}
Note that $\frac{\tilde{f}}{M}\in\mathcal{BL}_1(X_2)$. Hence 
\[|\eta_{t_j}^x(\tilde{f})-\eta_t^x(\tilde{f})|=M|\eta_{t_j}^x(\tilde{f}/M)-\eta_t^x(\tilde{f}/M)|<M\cdot\frac{\varepsilon}{3M}=\frac{\varepsilon}{3},\quad \text{for all}\quad x\in X\] 
Using triangle inequality, \eqref{Eq-Lipschit-approximation} and \eqref{Eq-uniform-d-infinity-convergence},
\begin{align*}
|\eta^x_{t_j}(f)-\eta^x(f)|\le&|\eta^x_{t_j}(f)-\eta^x_{t_j}(\tilde{f})|+|\eta^x_{t_j}(\tilde{f})-\eta^x(\tilde{f})|+|\eta^x(\tilde{f})-\eta^x(f)|\\
\le&\|f-\tilde{f}\|_{\infty}\|\eta_{\cdot}\|+M|\eta^x_{t_j}(\tilde{f}/M)-\eta^x(\tilde{f}/M)|+\|f-\tilde{f}\|_{\infty}\|\eta_{\cdot}\|\\
\le&\frac{\varepsilon}{3(1+\|\eta_{\cdot}\|)}+M\frac{\varepsilon}{3M}+\frac{\varepsilon}{3(1+\|\eta_{\cdot}\|)}<\varepsilon
\end{align*}
This shows that $\eta_{\cdot}$ is $x$-uniformly weakly $t$-continuous.
\end{enumerate}
}\end{proof}

\begin{definition}
For $i=1,2$, let $X_i$ be {closed subsets} of a finite dimensional Euclidean space. Assume that $(X_1,\mathfrak{B}(X_1),\mu_{X_1})$ is a compact probability space.  Let $\mathcal{I}\subseteq\R$ be a compact interval and $\alpha>0$.  For $\nu^1_{\cdot},\nu^2_{\cdot}\in \mathcal{C}(\mathcal{I},\mathcal{B}_{*}(X_1,\cM_+(X_2)))$, let
  \[d_{\alpha}(\nu^1_{\cdot},\nu^2_{\cdot})=\sup_{t\in\mathcal{I}}\textnormal{e}^{-\alpha t}d_{\infty}(\nu^1_t,\nu^2_t)\]be a \emph{weighted} uniform metric. In particular, $$d_0(\nu^1_{\cdot},\nu^2_{\cdot})=\sup_{t\in\mathcal{I}}d_{\infty}(\nu^1_t,\nu^2_t).$$
\end{definition}

\begin{proposition}\cite[{Proposition~2.11}]{KX21}
For $i=1,2$, let $X_i$ be closed subsets of a finite-dimensional Euclidean space. Assume that $(X_1,\mathfrak{B}(X_1),\mu_{X_1})$ is a compact probability space.  Let $\mathcal{I}\subseteq\R$ be a compact interval and $\alpha>0$. Then $(\mathcal{C}(\mathcal{I},\mathcal{B}_{*}(X_1,\cM_+(X_2))),d_{\alpha})$ and $(\mathcal{C}(\mathcal{I},\mathcal{C}_{*}(X_1,$ $\cM_+(X_2))),d_{\alpha})$ are both complete.
\end{proposition}
\subsection{Hyper-digraph measures}

\begin{definition}
  \label{hyperdigraph}
 {Let $J\subsetneq\mathbb{N}\setminus\{1\}$ be a finite sequence and an integer $N\ge\max J$.   A \emph{directed hypergraph} $\mathcal{H}$ is a pair $\mathcal{H}=(\mathcal{V},\mathcal{E})$, where $\mathcal{V}$ is the set of $N$ vertices, and $\mathcal{E}\subseteq \cup_{\ell\in J}\mathcal{V}_{\ell}$ with $\mathcal{V}_{\ell}=\{(v_1,\ldots,v_{\ell})\colon v_i\in\mathcal{V},\  i\in[\ell]\}$,  consists of directed hyper-edges (or edges for short, so long as no confusion will arise in the context) each of  which is identified as a vector with distinct entries. 
Let $\ell\in J$. For every directed hyper-edge $e=(v_{i_1},v_{i_2},\ldots,v_{i_{\ell}})\in\mathcal{E}$, $v_{i_1}$ and $v_{i_{\ell}}$ are  the \emph{head} and the \emph{tail} of the edge $e$, respectively, and $\ell$ the \emph{cardinality} of $e$. The maximal cardinality $k$ of all edges is called the \emph{rank} of $\mathcal{H}$. To exclude any unused elements in $J$, we assume w.l.o.g. that  $J$ is the set of all possible cardinalities of edges and hence $k=\max J$.   If every directed hyper-edge $e\in \mathcal{E}$ is assigned a positive weight $a_e$, then the  directed hypergraph $\mathcal{H}$ is called a \emph{weighted directed hypergraph}.\footnote{Throughout we will omit ``weighted'' since we will always refer to weighted hypergraphs.}  
A hypergraph simply refers to a directed hypergraph if no confusion arises from the context. A directed hypergraph $\mathcal{H}'=(\mathcal{V}',\mathcal{E}')$ is a \emph{sub-hypergraph}  of $\mathcal{H}$ if $\mathcal{V}'\subseteq\mathcal{V}$ and $\mathcal{E}'\subseteq\mathcal{E}$. In particular, if $J=\{k\}$ is a singleton, i.e., all directed hyper-edges share the same cardinality, then $\mathcal{H}$ is a \emph{$k$-uniform directed hypergraph}.}  

{The union of all directed hyperedges of cardinality $k$ of $\mathcal{H}$ together with their vertices is the \emph{maximal} $k$-uniform directed sub-hypergraph of $\mathcal{H}$.  Any directed hypergraph $\mathcal{H}$ of a finite rank $k$ can be decomposed into a union of finitely many maximal $\ell$-uniform directed sub-hypergraphs $\mathcal{H}_{\ell}$ for $\ell\in[k]\setminus\{1\}$;  in this case, each maximal $\ell$-uniform sub-hypergraph $\mathcal{H}_{\ell}$ is called a \emph{layer} of $\mathcal{H}$, and $\mathcal{H}$ is a \emph{multi-layer directed hypergraph} if it is not a uniform directed hypergraph.}

{For a directed hypergraph $\mathcal{H}$, the \emph{density} of $\mathcal{H}$ is defined as follows:
\begin{equation}\label{degree-directed-hypergraph}
D(\mathcal{H})=\frac{\#\mathcal{E}}{\sum_{\ell\in J}\ell!\small\begin{pmatrix}\#\mathcal{V}\\ \ell\end{pmatrix}}
\end{equation} 
  A $k$-uniform directed hypergraph of $N$ vertices consisting of all possible hyper-edges of cardinality $k$ is called a \emph{complete $k$-uniform hypergraph}, and denoted $\mathcal{K}_N^k$. A   
 hypergraph is \emph{complete} if each layer is a  complete $\ell$-uniform hypergraph for all $\ell\in J$. A $k$-uniform directed hypergraph $(\mathcal{V},\mathcal{E})$ is \emph{$[k]$-permutation invariant} if $\mathcal{E}=\{((j_{\sigma(1)},j_{\sigma(2)},\ldots,j_{\sigma(k)}))\colon (j_1,j_2,\ldots,j_k)\in\mathcal{E}\}$ for any permutation $\sigma\colon [k]\to[k]$.} 

{A directed hypergraph is \emph{permutation invariant} if all of its layers are permutaton invariant. A permutation invariant directed hypergraph is an \emph{undirected hypergraph}, or simply a hypergraph; in this case, all directed hyper-edges constituting of the same set of vertices are identified  with their set of vertices and counted \emph{without multiplicity}. The density of an undirected hypergraph $\mathcal{H}$ is defined as \begin{equation}\label{degree-undirected-hypergraph}
D(\mathcal{H})=\frac{\#\mathcal{E}}{\sum_{\ell\in J}\small\begin{pmatrix}\#\mathcal{V}\\ \ell\end{pmatrix}}
\end{equation} 
The density of a sequence of $k$-uniform directed hypergraphs $\{\mathcal{H}_j=(\mathcal{V}_j,\mathcal{E}_j)\}_{j\in\mathbb{N}}$ is characterized by the asymptotics of the sequence of densities of  the hypergraphs: The sequence $\{\mathcal{H}_j\}_{j\in\mathbb{N}}$  is \emph{dense} if $D(\mathcal{H}_j)\thicksim1$ for large $j$ while is \emph{sparse} if $\#\mathcal{E}_j\precsim\#\mathcal{V}_j$ for large $j$;  otherwise, the sequence is neither dense nor sparse.} 
\end{definition}
\begin{remark}\label{re-1}
\begin{enumerate}
\item[$\bullet$] The concept of layer of a hypergraph of a finite rank is motivated from \cite{CKK22}.  
\item[$\bullet$]
Hypergraph defined in the literature (e.g.,\cite{ES12}) is \emph{undirected} in Definition~\ref{hyperdigraph}. 
\item[$\bullet$] {Indeed, every directed hypergraph $\mathcal{H}$ of $N$ vertices and rank $k$ can be \emph{associated with} $\mathcal{K}_N^2$, the complete directed graph, where the set $\mathcal{E}$ of directed hyper-edges can be identified with a \emph{subset} of directed paths of $\mathcal{K}_N^2$, and the cardinality of each directed hyper-edge equals the length of the path. Nevertheless, it is worth pointing out that a directed hypergraph cannot be \emph{identified with} a directed subgraph of $\mathcal{K}_N^2$, since not all subsequences of a path are a directed hyper-edge.}
\end{enumerate}
\end{remark}
\begin{example}
{Consider the hypergraph given in Figure~\ref{figure-hyper-graph}. It is a 3-layer directed hypergraph $\mathcal{H}=(\mathcal{E},\mathcal{V})$ consisting of in total of 3 uniform directed hypergraphs. We can assign a unit weight to each hyper-edge to make $\mathcal{H}$ a weighted directed hypergraph. Note that despite $(1,4,3,2)$ is a directed hyperedge of cardinality 4, none of any proper subsequence (e.g., $1\to4$) of the path $1\to4\to3\to2$ corresponds to another directed hyperedge of $\mathcal{H}$ with a lower cardinality. Hence the hypergraph cannot be identified with a  directed subgraph of $\cK_4^2$.}
\end{example}

Now we propose one candidate as a generalization of directed  hypergraph.
\begin{definition}\label{hyper-DGM}
{For $k\in\N\setminus\{1\}$, a measure-valued function in $\mathcal{B}(X,\cM_+(X^{k-1}))$ is called a \emph{$k$-uniform directed hypergraph measure} (abbreviated as ``\emph{$k$-uniform DHGM}''). A $k$-uniform DHGM $\eta\in \mathcal{B}(X,\cM_+(X^{k-1}))$ is a $k$-HGM if $\xi\coloneqq\mu_X\otimes\eta^x$ as a measure in $\cM_+(X^{k})$\footnote{$\xi$ is understood as: $$\xi(A)\coloneqq\int_{\pi_1A}\eta^x(A_x)\mathrm{d}\mu_X(x),\quad \text{for all measurable set}\quad A\subseteq X^k$$}  is \emph{$[k]$-permutation invariant} for any permutation  $\sigma\colon[k]\to[k]$: $$\rd\xi(x_1,\ldots,x_k)=\rd\xi(\sigma(x_1),\ldots,\sigma(x_k)),\quad (x_1,\ldots,x_k)\in X^k.$$ A  measure-valued function $\eta$ is called a \emph{directed hypergraph measure} (DHGM) of rank $k$ if it is a finite sum of $\ell$-uniform DHGM $\eta_{\ell}$ for $\ell\in J\subseteq[k]$ such that $\eta=\sum_{\ell\in J}\eta_{\ell}$ with $k=\max J$.  In this case, $\eta$ is a \emph{multi-layer DHGM}, and each $\ell$-uniform directed hypergraph measure $\eta_{\ell}$ is called a \emph{layer} of $\eta$.}
\end{definition}
\begin{definition}\label{hypergraphon}
{Let $\mu_X\in\cP(X)$ be a reference measure. A $k$-uniform DHGM $\eta\in \mathcal{B}(X,\cM_+(X^{k-1}))$ is called a \emph{$k$-uniform directed hyper-graphon} w.r.t. $\otimes_{j=1}^k\mu_X$ if $\mu_X\otimes\eta^x\in\cM_{+}(X^{{k}})$ is absolutely continuous w.r.t. $\otimes_{j=1}^k\mu_X$. A $[k]$-permutation invariant directed hyper-graphon is a \emph{$k$-uniform hyper-graphon}. A DHGM is called a \emph{directed hyper-graphon} of rank $k$ if it is a finite sum of $\ell$-uniform directed hyper-graphons $\eta_{\ell}\in \mathcal{B}(X,\cM_+(X^{\ell-1}))$ for $\ell\in J\subseteq[k]$. A DHGM is called a \emph{hyper-graphon} of rank $k$ if it is a finite sum of $\ell$-uniform hyper-graphons.}
\end{definition}

Based on the density characterization of a sequence of $k$-uniform hypergraphs in Remark~\ref{re-1}, we introduce dense and sparse  $k$-uniform DHGM. 
\begin{definition}\label{Def-dense-sparse}
  A $k$-uniform DHGM $\eta$ is \emph{dense} if $$\text{essinf}_{x\in X}\text{dim}_H(\supp\ \eta^x)=\text{dim}_H(X)>0,$$ where the essential infimum is w.r.t. the reference measure $\mu_X$. Hence a hypergraphon is dense. A $k$-uniform DHGM $\eta$ is \emph{sparse} if $$\sup_{x\in X}\#\supp\eta^x<\infty.$$ {A DHGM is \emph{dense} if it is a sum of finitely many dense uniform DHGMs; a DHGM is \emph{dense} if it is a sum of finitely many sparse uniform DHGMs. A sparse DHGM is called a \emph{directed hypergraphing}.} 
\end{definition}

\begin{remark}
{A directed DHGM of a finite rank $k$ may be used to characterize the limit of a sequence of multi-layer hypergraphs $\{\mathcal{H}_j\}_{j\in\mathbb{N}}$ of rank $k$, where each sequence of $\ell$-uniform sub hypergraphs $\{\mathcal{H}_{j,\ell}\}_{j\in\mathbb{N}}$ has a limit, and $\mathcal{H}_{j,\ell}$ is a layer of $\mathcal{H}_{j}$, for every $\ell\in J$ for some $J$ with $\max J=k$, $j\in\mathbb{N}$. One can always embed a DHGM $\eta$ of rank $k$ into $\mathcal{B}(X,\cM_+(X^{m-1}))$ for all integers $m\ge k$. Nevertheless, the layers of $\eta$ as $\ell$-uniform DHGMs $\eta_{\ell}$ for $\ell<m$ may become \emph{singular} in the sense that their \emph{fibers} $\eta_{\ell}^x$ are singular w.r.t. $\otimes_{j=1}^m\mu_X$ when regarded as in $\cM_+(X^{m-1})$.}
\end{remark}

\begin{remark}
{In contrast to a directed multi-layer hypergraph of a finite rank, a DHGM of an unbounded rank may not always be decomposed into finitely many $k$-uniform DHGMs.  
Instead it may be decomposed into countably infinite uniform DHGMs of \emph{unbounded} ranks (e.g., simplicial complexes of unbounded ranks \cite{RS23}). In this case, this directed hypergraph limit may be represented as a series $\sum_{\ell\in\mathbb{N}}\eta_{\ell}$, where each $\eta_{\ell}$ is an $\ell$-uniform DHGM.  
Such a DHGM could be viewed as an elements in $\mathcal{B}(X,\cM_+(\ell_1(X)))$, where $\ell_1(X)$ consists of sequences $\{x_n\}_{n\in\mathbb{N}}$ in $X$ with $\sum_{i=1}^{\infty}d(x_i,0)<\infty$, where $X$ is a subset of $\mathbb{R}^m$ and $0\in\mathbb{R}^m$ is the origin, and the metric $d$ can be chosen to be either a geodesic metric for e.g., $X=\mathbb{S}^{d-1}$ or just the metric induced by the Euclidean norm. One might also replace $\ell_1(X)$ by a space of functions defined on $X$, e.g., the space of all continuous real-valued functions on $X$.} 
\end{remark}
\subsection*{Representation of finite hypergraphs}
There are many ways of representing a finite (hyper)graph: by extended real-valued functions or by measurable sets (\emph{graphons} \cite{LS06}, \emph{extended graphons} \cite{JPS21}, or \emph{hypergrahons} in the sense of \cite{ES12}), by positive linear operators (\emph{graphop} \cite{BS20}), or by measure-valued functions (so-called \emph{digraph measures} \cite{KX21}). We now provide two ways of representing a {finite} hypergraph. For ease of exposition, we assume $X=[0,1]$ with $\mu_X$ being the Lebesgue measure on $X$, and we only represent a $k$-uniform hypergraph. For a general hypergraph $\cH$ of rank $k$, one can always first represent each layer $\cH_{\ell}$ of the hypergraph and then take the sum as the representation of $\cH$.

Let $\mathcal{H}^N=([N],\mathcal{E}^N)$ be a weighted $k$-uniform directed  hypergraph. Let $\{I_i^N\}$ be an equipartition of $X$ with 
\begin{equation}\label{Eq-equipartition}
I_{i}^N=\bigl[\tfrac{i-1}{N},\tfrac{i}{N}\bigr[,\quad 1\le i\le N-1,\quad \text{and}\quad I_N^N=\bigl[\tfrac{N-1}{N},1\bigr]
\end{equation} Then $\{I_{i_1,\ldots,i_k}^N\}$ is a uniform partition of $X^{k}$ with cubes $I_{i_1,\ldots,i_k}^N\coloneqq\otimes_{j=1}^kI^N_{i_j}$ for $\{i_1,\ldots,i_k\}\subseteq[N]^k$. We  associate with $\mathcal{E}^N$ an  $N^k$-dimensional non-negative adjacency matrix $(a^N_{i_1,\ldots,i_k})$ such that $a^N_{i_1,\ldots,i_k}>0$ if and only if $(i_1,\ldots,i_k)\in\mathcal{E}^N$. 

\begin{enumerate}
\item[(i)] \emph{Function representation}. Define a simple function $W^N\in \mathsf{L}^p(X^k)$:
\begin{equation}
W^N=\sum_{i_1=1}^N\cdots\sum_{i_{k}=1}^Na^N_{i_1,\ldots,i_{k}}\mathbbm{1}_{I_{i_1,\ldots,i_{k}}^N}
\end{equation}
Define a \emph{piecewise uniform measure valued function} $\eta_N\in\cB(X,\cM_+(X^{k-1}))$ by the Radon-Nikodym derivative:
\[\frac{\rd\eta_N^x(y_1,\ldots,y_{k-1})}{\rd(\otimes_{\ell=1}^{k-1}\mu_X)(y_1,\ldots,y_{k-1})}=W^N(x,y_1,\ldots,y_{k-1}),\quad x\in X\]where $\otimes_{\ell=1}^{k-1}\mu_X$ is the product measure on $X^{k-1}$ satisfying 
\[(\otimes_{\ell=1}^{k-1}\mu_X)(\prod_{\ell=1}^{k-1}A_{\ell})=\otimes_{\ell=1}^{k-1}\mu_X(A_{\ell}),\]
and $A_{\ell}\subseteq X$ for $\ell=1,\ldots,k-1$ are Borel measurable sets. {Note that the absolutely continuous measure $\eta_N^x$ has a piecewise constant density $W^N(x,\cdot)$ owing to the finiteness of the hypergraph $\mathcal{H}^N$, and hence as a measure, $\eta_N^x$ is ``piecewise uniform''.} 
\item[(ii)] \emph{Atomic {measure} representation}. Define $\tilde{\eta}_N\in\cB(X,\cM_+(X^{k-1}))$:
\[\tilde{\eta}^x_N=N^{-(k-1)}\sum_{i_2=1}^N\cdots\sum_{i_k=1}^Na^N_{i_1,\ldots,i_{k}}\delta_{\left(\frac{2i_2-1}{2N},\ldots,\frac{2i_{k}-1}{2N}\right)},\quad x\in I^N_{i_1}\]
\end{enumerate}
Indeed, assume that:\\
\smallskip

\noindent (\textbf{H}) $\quad \max_{1\le i\le N}\sum_{j_1=1}^{N}\cdots\sum_{j_{k-1}=1}^Na^N_{i,j_1,\ldots,j_{k-1}}=o(N^k)$\\

Then these two representations are \emph{asymptotically the same}.

\begin{proposition}\label{prop-no-difference-representation}
Assume (\textbf{H}). Then $\lim_{N\to\infty}d_{\infty}(\eta_N,\tilde{\eta}_N)=0$.
\end{proposition}
\begin{proof}\begin{align*}
d_{\infty}(\eta_N,\tilde{\eta}_N)=&\sup_{x\in X}d_{\sf BL}(\eta^x_N,\tilde{\eta}^x_N)\\
=&\max_{1\le i\le N}\sup_{x\in I^N_i}\sup_{f\in\BL_1(X^{k-1})}\int_{X^{k-1}}f(y_1,\ldots,y_{k-1})\rd(\eta^x_N-\tilde{\eta}^x_N)\\
=&\max_{1\le i\le N}\sup_{x\in I^N_i}\sup_{f\in\BL_1(X^{k-1})}\sum_{j_1=1}^{N}\cdots\sum_{j_{k-1}=1}^N\int_{\prod_{\ell=1}^{k-1}I^N_{j_{\ell}}}f(y_1,\ldots,y_{k-1})\rd(\eta^x_N-\tilde{\eta}^x_N)\\
=&\max_{1\le i\le N}\sup_{x\in I^N_i}\sup_{f\in\BL_1(X^{k-1})}\sum_{j_1=1}^{N}\cdots\sum_{j_{k-1}=1}^Na^N_{i,j_1,\ldots,j_{k-1}}\\
&\cdot\int_{\prod_{\ell=1}^{k-1}I^N_{j_{\ell}}}\bigl(f(y_1,\ldots,y_{k-1})-f(\tfrac{2i_2-1}{2N},\ldots,\tfrac{2i_{k}-1}{2N})\bigr)\rd(\otimes_{\ell=1}^{k-1}\mu_X)(y_1,\ldots,y_{k-1})\\
\le&\max_{1\le i\le N}\sum_{j_1=1}^{N}\cdots\sum_{j_{k-1}=1}^Na^N_{i,j_1,\ldots,j_{k-1}}\\
&\cdot\int_{\prod_{\ell=1}^{k-1}I^N_{j_{\ell}}}|(y_1,\ldots,y_{k-1})-(\tfrac{2i_2-1}{2N},\ldots,\tfrac{2i_{k}-1}{2N})|\rd(\otimes_{\ell=1}^{k-1}\mu_X)(y_1,\ldots,y_{k-1})\\
\le&\max_{1\le i\le N}\sum_{j_1=1}^{N}\cdots\sum_{j_{k-1}=1}^Na^N_{i,j_1,\ldots,j_{k-1}}\\
&\cdot\int_{\prod_{\ell=1}^{k-1}I^N_{j_{\ell}}}\bigl(|y_1-\tfrac{2i_2-1}{2N}|+\ldots+|y_{k-1}-\tfrac{2i_{k}-1}{2N}|\bigr)\rd(\otimes_{\ell=1}^{k-1}\mu_X)(y_1,\ldots,y_{k-1})\\
\le&\frac{k-1}{4}\max_{1\le i\le N}N^{-k}\sum_{j_1=1}^{N}\cdots\sum_{j_{k-1}=1}^Na^N_{i,j_1,\ldots,j_{k-1}}\to0,\quad \text{as}\quad N\to\infty
\end{align*}
\end{proof}
\begin{remark}
\begin{enumerate}
\item[(i)] Proposition~\ref{prop-no-difference-representation} implies that with weights of appropriate scales, the two sequences of  measure-valued functions representing the same sequence of finite graphs, diverge simultaneously, or converge to the same limit. In other words, whether the sequence of hypergraphs viewed as DHGMs converges or not is \emph{independent} of the choice of the representation under the assumption (\textbf{H}).
\item[(ii)]  Note that $$\|\eta_N\|=\|\tilde{\eta}_N\|=\max_{1\le i\le N}N^{-k+1}\sum_{j_1=1}^{N}\cdots\sum_{j_{k-1}=1}^Na^N_{i,j_1,\ldots,j_{k-1}}$$ and (\textbf{H}) is equivalent to $\|\tilde{\eta}_N\|=o(N)$, which is a necessary condition for $\{\eta_N\}_{N\in\mathbb{N}}$ as well as $\{\tilde{\eta}_N\}_{N\in\mathbb{N}}$  to be convergent. Moreover, if $\{\eta_N\}_{N\in\mathbb{N}}$ and $\{\tilde{\eta}_N\}_{N\in\mathbb{N}}$ converge, then from the proof we know that $d_{\infty}(\eta_N,\tilde{\eta}_N)\to0$ as $N\to\infty$ with a rate no slower than $N^{-1}$.
\item[(iii)] This proposition also rules out the illusion that the atomic measure representation of finite (hyper)graphs implies the sequence of the (hyper)graphs are sparse. See Example~\ref{ex-complete-k-uniform-hypergraphon} below.
\item[(iv)] As will be seen below, $X$ is not necessarily $[0,1]$. Nevertheless, one can always represent a finite hypergraph in two ways analogously defined above.
\end{enumerate}
\end{remark}

\begin{example}\label{ex-complete-k-uniform-hypergraphon}
  Let $k\in\mathbb{N}\setminus\{1\}$ and $\{\mathcal{K}_N^k\}_{N\ge k}$ be the sequence of complete $k$-uniform hypergraphs \cite{L12}. Then $\{\mathcal{K}_N^k\}_{N\ge k}$, represented as DHGMs (in either of the two ways aforementioned), converges to the \emph{complete $k$-uniform hypergraphon} $\eta\in\cB(X,\cM_+(X^{k-1}))$: $$\eta^x\equiv\lambda|_{X^{k-1}}\quad \text{for all}\quad x\in X.$$
\end{example}
{\begin{example}\label{non-unique-representation}
Consider the graph of $N$ vertices with the adjacency matrix $$a_{ij}^N=\begin{cases}
    N^{\alpha},&\text{if}\ j=i+1,\ i<N,\\
    0,&\text{else},
  \end{cases}$$ for some constant $\alpha>0$. Hence \[\max_{1\le i\le N}\sum_{j=1}^Na_{ij}^N=N^{\alpha}.\]
For $0\le\alpha<2$, (\textbf{H}) is satisfied, and there is no difference in the two representations in the convergence of the graph sequence. With the atomic {measure} representation $\widetilde{\eta}_N\in\cB(X,\cM_+(X))$:
  \begin{equation*}
    \tilde{\eta}_N^x=N^{\alpha-1}\delta_{(2i+1)/(2N)},\quad x\in I^N_i,\quad i=1,\ldots, N-1,
  \end{equation*}
where $\{I_i^N\}_{i=1}^N$ is given in \eqref{Eq-equipartition}. 
The sequence converges to $\eta (\alpha)\in\cB(X,\cM_+(X))$ for $0<\alpha\le1$,  given by \begin{equation}\label{Eq-xi-limit}
\eta^x(\alpha)=\begin{cases}\delta_x,& \text{for}\ x\in X,\ \alpha=1,\\
0,& \text{for}\ x\in X,\ \alpha<1,
\end{cases}
\end{equation}
which is a non-trivial graph limit if and only if $\alpha=1$; moreover, the sequence diverges in $\cB(X,\cM_+(X))$ for $1<\alpha<2$. In contrast, using the function representation, one can construct a 2-uniform DHGM 
$\eta_N\in \cB(X,\cM_+(X))$ with the density at each $x\in X$ given by $W^N(x,\cdot)$:
  \begin{equation*}
   W^N(x,y) = \frac{\rd\eta_N^x(y)}{\rd \mu_X(y)}=\begin{cases}
      N^{\alpha},&\text{if}\quad  x\in I^N_i,\quad y\in I^N_{i+1},\quad i=1,\ldots,N-1,\\
      0,&\text{else}\quad \mu_X\text{-a.e.}\ y\in[0,1],
    \end{cases}
  \end{equation*}
It is obvious that $$\|W^N\|_{\mathsf{L}^p(X^2)}=\begin{cases}N^{\alpha-2/p}(N-1)^{1/p},& \text{if}\ 1\le p<\infty,\\ 
N^{\alpha},& \text{if}\ p=\infty.\end{cases}$$ Hence, for  $0<\alpha<1$,  $\{W^N\}_{N\in\mathbb{N}}$ converges to  $0\in \mathsf{L}^p(X^2)$ for $1\le p<\tfrac{1}{\alpha}$ while diverges in $L^{p}(X^2)$ for $\tfrac{1}{\alpha}\le p\le \infty$ (note that despite the norm of the sequence converges to $1$ for $p=\alpha^{-1}$, the sequence diverges in $\mathsf{L}^p(X^2)$). Hence in a space of $\mathsf{L}^p$-integrable graphons for $1\le p\le\infty$,  the sequence of graphs either converges to a trivial limit or diverges. {Nevertheless, by Proposition~\ref{prop-no-difference-representation}, the sequence of graphs represented in either way converges to $\eta(\alpha)$ under the metric $d_{\infty}$, which is not absolutely continuous w.r.t. the Lebesgue measure on $X^2$ for $\alpha=1$, which explains why $\{W^N\}_{N\in\mathbb{N}}$ diverges in $\mathsf{L}^p(X^2)$ in this case. This case demonstrates DHGM embraces graph limits which are not absolutely continuous.} Despite the \emph{cut metric} \cite{J13}\footnote{When $X=[0,1]$, the cut metric between two $\mathsf{L}^1$-graphons $W_1$ and $W_2$ equals the supremum $\int_{A\times B}(W_1(x)-W_2(x))dx$ over all Borel meaurable subsets $A,B\subseteq[0,1]$.  For the definition of cut metric for a general metric space $X$, we refer the interested reader to e.g. \cite{L12,J13}.} is different from the metric induced by the $L^1$-norm, the convergence of $\{W^N\}_{N\in\mathbb{N}}$ in $L^1$-norm coincides with the convergence in the cut metric. 
\end{example}

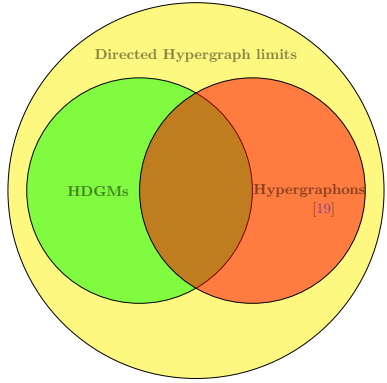
\begin{figure}
\begin{center}
\scalebox{.5}{\begin{tikzpicture}
    \begin{scope}[shift={(3cm,-5cm)}, fill opacity=0.5]
        \draw[fill=yellow, draw = black] (0,0) circle (5);
        \draw[fill=green, draw = black] (-1.5,0) circle (3);
    \draw[fill=red, draw = black] (1.5,0) circle (3);
    \node at (0,3.6) (A) {\large\textbf{Directed Hypergraph limits}};
    \node at (-2.6,0) (B) {\large\textbf{DHGMs}};
    \node at (3.02,0) (C) {\large\textbf{Hypergraphons}};
    \node at (3.4,-.5) (D) {\large\textbf{\cite{ES12}}};
    \end{scope}
\end{tikzpicture}}
\end{center}
\caption{Venn diagram for different hypergraph limits.}\label{fig-1}
\end{figure}

\begin{remark}
\begin{enumerate}
\item[$\bullet$] The topology induced by the total variation distance is stronger than the weak topology (note that for absolutely continuous measures of the same mass, the distance between two measures in total variation norm is equivalent to the distance of their densities in $\mathsf{L}^p$-norm), which results in a fact that the $\mathsf{L}^p$ graphons are a proper subset of the space of positive measures in general. This also explains why DHGMs may contain non-dense graph limits of certain classes.
\item[$\bullet$]
{Here we provide some explanation on why our definition of hypergraphon differs from the standard one defined in \cite{ES12}, pertaining the dimension of the vertex space. In \cite{ES12}, in order to define limit object of $k$-uniform hypergraphs w.r.t. the topology induced by \emph{homomorphism density}, one needs to regard the limit object as a measurable subset of $[0,1]^{2^k}$, where $2^k$ is the number of subsets of $[k]$. Since the empty set plays no role \cite{L12}, one can also barely consider measurable subsets of $[0,1]^{2^k-1}$, by excluding the coordinate of a point indexed by the empty set. The reason of considering the limit of hypergraphs as such sets rather than a function is due to the topology induced by some generalized \emph{homomorphism density} \cite{ES12}. For $k=2$, homomorphism density from graph $\mathcal{F}$ to graph $\mathcal{G}$ refers to as the probability of a random mapping from the vertex set of $\mathcal{F}=([N],\mathcal{E})$ to that of $\mathcal{G}$, whose definition can be extended as 
\begin{equation}\label{homomorphism-density}
t(\mathcal{F},\mathcal{W})=\int_{[0,1]^N}\prod_{(i,j)\in\mathcal{E}}\mathcal{W}(x_i,x_j)d x_1 \ldots d x_N,\end{equation}
representing the homomorphism density from a finite graph $\mathcal{F}$ to a limit object (graphon)\---a measurable function $\mathcal{W}\colon[0,1]^2\to[0,1]$. The convergence characterized in terms of the homomorphism density can also be characterized by the \emph{cut distance} \cite[Theorem~3.8]{BCLSV08}. However, in order to ensure the existence of the limit for a convergent sequence of hypergraphs w.r.t. the distance defined below analogous to \eqref{homomorphism-density} \cite{L12}:
\begin{equation}\label{homomorphism-density-hypergraph}
t(\mathcal{F},\mathcal{W})=\int_{[0,1]^N}\prod_{(i_1,\ldots,i_k)\in\mathcal{E}}\mathcal{W}(x_{i_1},\ldots,x_{i_k})d x_1 \ldots d x_N,\end{equation}
for any $k$-uniform test hypergraph $\mathcal{F}=([N],\mathcal{E})$, a \emph{na\"{i}ve} straightforward extension to limit of hypergraphs defined as measurable functions on $[0,1]^k$, is \emph{impossible}. The interested reader may refer to \cite{L12} for an example. Hence an appropriate generalized definition of homomorphism density as a probability when regarding hypergraph limit as measurable subsets of $[0,1]^{2^k}$, is proposed \cite{ES12}. The existence of limit theorem, confined to the family of such measurable sets as the space of hypergraph limits, holds w.r.t. this generalized definition of homomorphism density \cite{ES12}. In short, \emph{such a definition of hypergraph limit is for the completeness of the limit objects under homomorphism density}. In contrast, {with a different metric which is widely used in the context of dynamical networks,} we treat these limits from the measure-valued function viewpoint, and the completeness is naturally guaranteed due to very basic properties of space of continuous/bounded functions on a compact complete metric space. Hence, hypergraphons as HDMs defined in Definition~\ref{hyperdigraph}, may not coincide with that defined as hypergraphons in \cite{ES12}. For instance for $k=2$, despite the uniform weak topology induced by the uniform bounded Lipschitz metric is stronger than the topology induced by homomophism density (essentially due to the ``uniformity'' $\sup_{x\in X}$), the space of graph measures \cite{KX21} is \emph{not} a proper subset of any space of integral graphons as it contains non-absolutely continuous measure-valued functions, evidenced by Example~\ref{non-unique-representation} as well as examples below. {Hence these two definitions of graph limits may share a non-empty intersection of \emph{dense} graph limits (see Figure~\ref{fig-1}). Nevertheless, it remains to see how much they have in common. Analogously, the definition of hypergraphs in Definition~\ref{hyper-DGM} have a non-empty intersection with  hypergraphons defined in \cite{ES12}. It would be of great interest to figure out the difference set of these two classes of hypergraph limits.}} 
\item[$\bullet$] {It is noteworthy that graph limits have been defined as measures on the square $X^2$ \cite{KLS19} in the literature, which particularly are used to characterize sparse graph limits. Nevertheless, the DGMs \cite{KX21} as well as HDGMs defined here are measure-valued \emph{functions} on $X$, and the natural motivation of such a definition rather than product measures or kernels (graphons) on the product space owes to the \emph{fiberized characteristic equation} indexed by vertices of graph limits.} 
Nevertheless, though rare in the literature, it is noteworthy that similar topics have been dealt with under the cut distance in a very elegant way in \cite{JPS21}. The results in \cite{JPS21} only pertain sparse graph limits. It will be interesting to see if the approach in \cite{JPS21} can be further extended to non-sparse (hyper)graph limits.
\end{enumerate}
\end{remark}

{Since every HDGM is a finite sum of uniform HDGMs, 
in the following we provide several generic examples of $k$-uniform HDGMs which are not dense (and thus not hypergraphons). These examples showcase the difference between our definition of hypergraph limits and those in \cite{ES12}.}

{\begin{example}\label{ex-torical-graphop}
For every $N\in\mathbb{N}$, let $\mathcal{G}^N=(\mathcal{V}^N,\mathcal{E}^N)$ with $\mathcal{V}^N=[N]^2$ and $$\mathcal{E}^N=\{((i_1,i_2),(j_1,j_2))\in(\mathcal{V}^N)^2\colon i_1-j_1=\lfloor N/2\rfloor\mod N\}$$ with the weights 
\[a_{(i_1,i_2),(j_1,j_2)}=\begin{cases}1,& \text{if}\ j_1-i_1=\lfloor N/2\rfloor\mod N,\\ 
0,& \text{otherwise}.\end{cases}\]
Now we represent $\mathcal{G}^N$ as a \emph{graph measure} \cite{KX21} $\eta_N\in\cB(X,\cM_+(X))$ with $X=\mathbb{T}^2$ given as follows:
\begin{align*}
\eta_N^x=\lambda_{\mathbb{T}\big|_{\bigl\langle\frac{2i-1}{2N}+\tfrac{1}{2}\bigr\rangle}},\quad x\in \tilde{I}^N_i\times \tilde{I}^N_j,\quad \text{for}\ i,j=1,\ldots,N,
\end{align*}
where 
\begin{equation}\label{Eq-equipartition-circle}
\tilde{I}^N_i=[\tfrac{i-1}{N},\tfrac{i}{N}[\quad \text{for}\ i=1,\ldots,N,
\end{equation} and $\mathbb{T}_a\coloneqq\{a\}\times\mathbb{T}\subseteq X$ for every $a\in\mathbb{T}$. Then it is readily shown that $\eta_N$ converges to $\eta\in\cB(X,\mathcal{M}_+(X))$ with 
\[\eta^x=\lambda\bigl|_{\mathbb{T}_{\bigl\langle x_1+\tfrac{1}{2}\bigr\rangle}},\quad x=(x_1,x_2)\in X\]
which is an absolutely continuous measure supported on a subset of $X$ of a lower Hausdorff dimension.  Geometrically, on the generalized graph $\eta$, every vertex $x=(x_1,x_2)$ on the circle $\mathbb{T}^{x_2}\coloneqq\{y\in X\colon y_2=x_2\}$ connects to every vertex on the circle $\mathbb{T}_{\langle x_1+1/2\rangle}$ perpendicular to $\mathbb{T}^{x_2}$. We call {the graph limit} $\eta$ a \emph{torical graph measure}.
\end{example}}

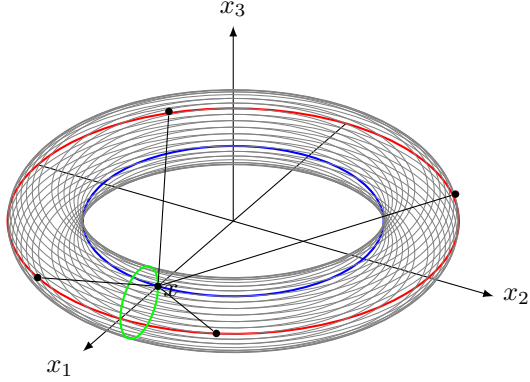
\begin{figure}
\begin{center}
\scalebox{1.0}{\tdplotsetmaincoords{60}{120}
\begin{tikzpicture}[tdplot_main_coords]
\draw [-latex] (-3,0,0) -- (4,0,0) node [below left] {$x_1$};
\draw [-latex] (0,-3,0) -- (0,4,0) node [right] {$x_2$};
\draw [-latex] (0,0,0)  -- (0,0,3) node [above] {$x_3$}; 
\begin{scope}[canvas is xy plane at z=0]
  \draw[thick, blue, decoration={markings}, postaction={decorate}] (0,0) circle (2);
  \draw[thick, red, decoration={markings}, postaction={decorate}] (0,0) circle (3);
\end{scope}
\foreach\z in {10,20,...,80}
{
  \pgfmathsetmacro\r{0.5*cos(\z)};
  \begin{scope}[canvas is xy plane at z=0.5*sin(\z)]
    \draw[very thin,gray] (0,0) circle (2.5+\r);
    \draw[very thin,gray] (0,0) circle (2.5-\r);
  \end{scope}
  \begin{scope}[canvas is xy plane at z=-0.5*sin(\z)]
    \draw[very thin,gray] (0,0) circle (2.5+\r);
    \draw[very thin,gray] (0,0) circle (2.5-\r);
  \end{scope}
}

\begin{scope}[canvas is xz plane at y=0]
  \draw[thick,green]  (2.5,0) circle (0.5);
\end{scope}
\node[circle,fill,inner sep=1pt] (A) at (2,0) {};
\node[circle,fill,inner sep=1pt] (B1) at (-2.1,-2.2) {};
\node[circle,fill,inner sep=1pt] (B2) at (2.6,-1.5) {};
\node[circle,fill,inner sep=1pt] (B3) at (2.7,1.3) {};
\node[circle,fill,inner sep=1pt] (B4) at (-2.1,2.2) {};

\node[] at (2,0.2) {$x$};
\draw (A) -- (B1);
\draw (A) -- (B2);
\draw (A) -- (B3);
\draw (A) -- (B4);
\end{tikzpicture}
}
\end{center}
\caption{Example~\ref{ex-torical-graphop}. Torical graph measure. Every point on the blue circle connects to every point on the red circle.}
\end{figure}

\begin{example}\label{ex-triangle}
Consider a sequence of hypergraphs $\{\mathcal{H}^N\}_N$ of rank $3$ with vertex sets being triangularization points of $X$: $$\mathcal{V}^N=\Bigl\{\bigl(\tfrac{i}{2N},\tfrac{\sqrt{3}}{2}\cdot\tfrac{j}{2N}\bigr)\colon j\le \min\{2i,4N-2i\},\quad i,j=0,\ldots,2N\Bigr\},$$ \begin{multline*}
\mathcal{E}^N=\Bigl\{\bigl(\bigl(\tfrac{i}{2N},\tfrac{\sqrt{3}}{2}\cdot\tfrac{j}{2N}\bigr),\bigl(\tfrac{i+N}{2N},\tfrac{\sqrt{3}}{2}\cdot\tfrac{j}{2N}\bigr),\bigl(\tfrac{i+N}{2N},\tfrac{\sqrt{3}}{2}\cdot\tfrac{j+N}{2N}\bigr)\bigr)\colon \\  j\le \min\{2i,2N-2i\},\quad i,j=0,\ldots,N\Bigr\}
\end{multline*}Let $a^N_{v_1,v_2,v_3}=\begin{cases}1& \text{if}\ (v_1,v_2,v_3)\in\mathcal{E}^N,\\ 0&\text{otherwise}.\end{cases}$ We represent $\mathcal{H}^N$ as a 3-uniform HGM $\eta_N\in\cB(X,\cM_+(X^2))$:
\[\eta_N^x=\begin{cases}
\delta_{\bigl(\tfrac{2i-1}{4N},\tfrac{\sqrt{3}}{2}\cdot\tfrac{2j-1}{4N}\bigr)},& \text{if}\quad x\in A^N_i\times B^N_j,\quad j\le \min\{2i,2N-2i\},\quad  i,j=0,\ldots,N,\\
0,& \text{otherwise}.
\end{cases}\]
 Let $X$ be the equilateral triangle with vertices $(0,0)$, $(0,\frac{\sqrt{3}}{2})$,  and $(\frac{1}{2},\frac{1}{2})$. Define $\eta\in\cB(X,\cM_+(X^2))$:
\[\eta^x=\begin{cases}\delta_{\bigl(x+\bigl(\tfrac{1}{2},0\bigr),x+\bigl(\tfrac{1}{4},\tfrac{\sqrt{3}}{4}\bigr)\bigr)},& \text{if}\quad 2x\in X,\\ 0,& \text{otherwise}.\end{cases}\] 
It can be readily shown that $d_{\infty}(\eta_N,\eta)\to0$ as $N\to\infty$. Note that $\eta$ is a $3$-uniform hypergraphing. We call $\eta$ a \emph{sparse triangle $3$-uniform hypergraphing}.
\begin{figure}
\begin{center}
\begin{tikzpicture}
   [
        dot/.style={circle,draw=black, fill,inner sep=1pt},scale=4
    ]
    \draw (1,.01) -- node[below,yshift=-1mm] {1} (1,0);
\node[below,xshift=-2mm,yshift=0mm] at (0,0) {0};
     \node[below,xshift=-3mm,yshift=-2mm] at (0,1){1};
\foreach \y in {.1,.2,...,1}
    \draw (.01,\y) -- node[below,xshift=-3mm,yshift=2mm] {} (0,\y);
\foreach \y in {0,.1,...,1.1}
     {\node[red,dot] at (\y,0){};
     }
    \foreach \y in {0,.1,...,1}
     {\node[red,dot] at (\y+0.05,.0866){};
     }
 \foreach \y in {0,.1,...,.9}
     {\node[red,dot] at (\y+0.1,2*.0866){};
     }
 \foreach \y in {0,.1,...,.8}
     {\node[red,dot] at (\y+0.15,3*.0866){};
     }
 \foreach \y in {0,.1,...,.7}
     {\node[red,dot] at (\y+0.2,4*.0866){};
     }
 \foreach \y in {0,.1,...,.6}
     {\node[red,dot] at (\y+0.25,5*.0866){};
     }
 \foreach \y in {0,.1,...,.5}
     {\node[red,dot] at (\y+0.3,6*.0866){};
     }
 \foreach \y in {0,.1,...,.4}
     {\node[red,dot] at (\y+0.35,7*.0866){};
     }
 \foreach \y in {0,.1,...,.3}
     {\node[red,dot] at (\y+0.4,8*.0866){};
     }
 \foreach \y in {0,.1,...,.2}
     {\node[red,dot] at (\y+0.45,9*.0866){};
     }
\node[red,dot] at (.5,10*.0866){};
\draw[->,very thick,-latex] (0,-.1) -- (0,1.0);
\draw[->,thick,-latex] (-.1,0) -- (1.1,0);
 \foreach \y in {0,.1,...,.6}
     {
\draw[thick,blue] (\y,0) -- (\y+.25,.433);
\draw[thick,blue] (\y,0) -- (\y+.5,0);
\draw[thick,blue] (\y+0.5,0) -- (\y+.25,.433);
}
 \foreach \y in {0,.1,...,.5}
     {
\draw[thick,blue] (\y+.05,.0866) -- (\y+.25+.05,.0866+.433);
\draw[thick,blue] (\y+.05,.0866) -- (\y+.5+.05,.0866+0);
\draw[thick,blue] (\y+0.5+.05,.0866) -- (\y+.25+.05,.0866+.433);
}
 \foreach \y in {0,.1,...,.4}
     {
\draw[thick,blue] (\y+2*.05,2*.0866) -- (\y+.25+2*.05,2*.0866+.433);
\draw[thick,blue] (\y+2*.05,2*.0866) -- (\y+.5+2*.05,2*.0866+0);
\draw[thick,blue] (\y+0.5+2*.05,2*.0866) -- (\y+.25+2*.05,2*.0866+.433);
}
 \foreach \y in {0,.1,...,.3}
     {
\draw[thick,blue] (\y+3*.05,3*.0866) -- (\y+.25+3*.05,3*.0866+.433);
\draw[thick,blue] (\y+3*.05,3*.0866) -- (\y+.5+3*.05,3*.0866+0);
\draw[thick,blue] (\y+0.5+3*.05,3*.0866) -- (\y+.25+3*.05,3*.0866+.433);
}
 \foreach \y in {0,.1,...,.2}
     {
\draw[thick,blue] (\y+4*.05,4*.0866) -- (\y+.25+4*.05,4*.0866+.433);
\draw[thick,blue] (\y+4*.05,4*.0866) -- (\y+.5+4*.05,4*.0866+0);
\draw[thick,blue] (\y+0.5+4*.05,4*.0866) -- (\y+.25+4*.05,4*.0866+.433);
}
 \foreach \y in {0}
     {
\draw[thick,blue] (\y+5*.05,5*.0866) -- (\y+.25+5*.05,5*.0866+.433);
\draw[thick,blue] (\y+5*.05,5*.0866) -- (\y+.5+5*.05,5*.0866+0);
\draw[thick,blue] (\y+0.5+5*.05,5*.0866) -- (\y+.25+5*.05,5*.0866+.433);
}
\draw[thick,green] (3*.05,3*.0866) -- (.25+3*.05,3*.0866+.433);
\draw[thick,green] (3*.05,3*.0866) -- (.5+3*.05,3*.0866+0);
\draw[thick,green] (0.5+3*.05,3*.0866) -- (.25+3*.05,3*.0866+.433);
  \fill[gray, opacity=.3] (0,0) -- (1,0) -- (0.5,0.866) -- cycle;
  \fill[gray, opacity=.5] (0,0) -- (.5,0) -- (0.25,0.866*.5) -- cycle;
  \end{tikzpicture}
  \caption{$3$-uniform triangular hypergraph $\mathcal{H}^5$ in Example~\ref{ex-triangle}. A typical hyper-edge consisting of three  vertices of an equilateral triangle of length 5 is colored in green (all others are colored in blue). The shaded area of a darker color consists of all vertices which belong to one ``triangle'' hyperedge located in the lower left corner.}\label{fig-ex-triangle}
\end{center}
\end{figure}
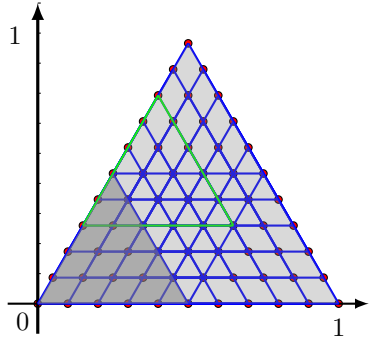
\end{example}

Below we provide more examples of $k$-uniform DHGM for general $k\in\mathbb{N}$.

\begin{example}
 Let $k\in\N\setminus\{1\}$ {and $f\in\cC(X,\R_+)$ and $\{A_x\subseteq X^{k-1}\colon x\in X\}$ such that $$\sup_{x\in X}\#A_x<\infty$$ 
 Define $\eta\colon X\to\cM_+(X^{k-1})$ by 
$$\eta^x=f(x)\lambda|_{A_x},\quad x\in X$$ 
Then $\eta$ is a $k$-uniform directed hypergraphing.} 
\end{example}

\begin{example}
 Let $k\in\N\setminus\{1\}$, $X=\mathbb{T}$, and $\mu_X=\lambda|_{\T}$. Define  the \emph{circular $k$-uniform DHGM} $\eta\in\cB(X,\cM_+(X^{k-1}))$: 
$$\eta^x=\otimes_{j=1}^{k-1}(\delta_{\langle x+1/4\rangle}+\delta_{\langle x+3/4\rangle}),\quad x\in\T.$$ {A geometric interpretation for $k=2$ is that every vertex $x$ on the graph limit connects to those two points on the circle with an angle of  $45^{\circ}$ clockwise and counter-clockwise, respectively.} It is readily verified that $\#\supp\eta^x=2^{k-1}$, {and $\eta$ is also a directed hypergraphing.} 
\end{example}

\begin{example}\label{ex-1}
 Let $X=[0,1]$, $\mu_X=\lambda|_X$, and $k\in\N\setminus\{1\}$. Let $\Delta\coloneqq\{(x_1,\ldots,x_k)\in\mathbb{R}_+^k\colon \sum_{j=1}^kx_j\le1\}$ and $\Delta_{x_1}\coloneqq\{y\in\Delta\colon y_1=x_1\}$
 be the slice of $\Delta$ at $x_1\in X$. Define $\eta\in\cB(X,\cM_+(X^{k-1}))$: $$\eta^{x_1}=k(1-x_1)^{k-1}\lambda|_{\Delta_{x_1}},\quad \text{for}\quad x_1\in X$$ Then it is straightforward to show that $\mu_X\otimes \eta^{x_1}=\lambda|_{\Delta}$,  which is a dense $k$-uniform DHGM. See Figure~\ref{fig-ex-1}.
  \begin{figure}[h]
\centering\scalebox{1.5}{
\begin{tikzpicture}[
        dot/.style={circle,draw=black, fill,inner sep=1pt},scale=1.5
    ]
\draw[->, >=latex, thick] (0,0,1.5) -- (0,0,2);
\draw[->, >=latex, thick] (1,0,0) -- (1.25,0,0);
\draw[->, >=latex, thick] (0,1,0) -- (0,1.25,0);
\draw[thick,dashed] (0,0,0) -- (0,0,2);
\draw[thick,dashed] (0,0,0) -- (1.25,0,0);
\draw[thick,dashed] (0,0,0) -- (0,1.25,0);
\draw[->, >=latex, thin] (1,1,0) -- (0.4,0.4,0.5);
\draw[dashed, thick] (0.3,0,0) -- (0.3,0.7,0);
\node[thick,black] at (0.4,-0.1,0) {\small$x_1$};
\node[thick,black] at (1,1.1,0) {\small$\Delta_{x_1}$};
\node[thick,black] at (-.1,0.05,0) {\small$0$};
\node[thick,black] at (1,-0.1,0) {\small$1$};
\node[thick,black] at (-.1,1,0) {\small$1$};
\node[thick,black] at (-.15,0,1.3) {\small$1$};
\draw[dashed, thick] (0.3,0,0) -- (0.3,0,1.1);
\draw[dashed, thick] (0.3,0,1.1) -- (0.3,0.7,0);
\draw[thick] (0,1,0) -- (1,0,0);
\draw[thick] (1,0,0) -- (0,0,1.5);
\draw[thick] (0,0,1.5) -- (0,1,0);
\fill[gray!, opacity=0.8] (0.3,0,0) -- (0.3,0.7,0) -- (0.3,0,1.1) -- cycle;
\end{tikzpicture}}
\caption{Example~\ref{ex-1} for $k=3$.}\label{fig-ex-1}
\end{figure}
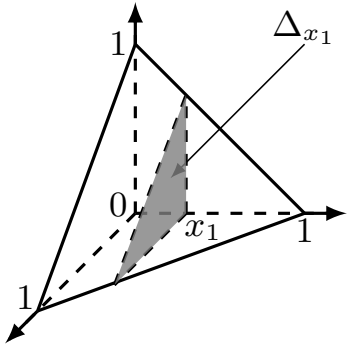
\end{example}

\begin{example}
Let $k\in\mathbb{N}\setminus\{1\}$. Consider the following sequence of dense \emph{directed $k$-uniform directed hypergraphs} $\{\mathcal{H}^N\}_{N\in\mathbb{N}}$ with $\mathcal{H}^N=(\mathcal{V}^N,\mathcal{E}^N)$ with $\mathcal{V}^N=[N]$ and 
\[\mathcal{E}^N=\Bigl\{(i,j_1,\ldots,j_{k-1})\colon N^{-1}j_{\ell}\in\bigl[N^{-1}\bigl(i+\bigl\lfloor \tfrac{N}{4}\bigr\rfloor\bigr),N^{-1}\bigl(i+\bigl\lfloor\tfrac{3N}{4}\bigr\rfloor\bigr)[\hspace{-.3cm}\mod 1,\ \ell=1,\ldots,k-1,\ 1\le i\le N\Bigr\}\] Let $X=\T$, $\mu_X=\lambda|_{\T}$. We can represent this hypergraph sequence as $\eta_N\in\cB(X,\cM_+(X^{k-1}))$:
\[\eta_N^x=\mathbbm{1}_{\prod_{\ell=1}^{k-1}\tilde{I}^N_{j_{\ell}}}(y_1,\ldots,y_{k-1}),\quad x\in \tilde{I}^N_i,\quad (i,j_1,\ldots,j_{k-1})\in\mathcal{E}^N,\]
where $\tilde{I}^N_{\ell}$ is defined in \eqref{Eq-equipartition-circle}. Then it is readily verified that $d_{\infty}(\eta^N,\eta)\to0$ for some 
 $k$-uniform DHGM $\eta\in\cB(X,\cM_+(X^{k-1}))$ defined by 
 $$\eta^x=\otimes_{j=1}^{k-1}\lambda\bigl|_{[x+1/4,x+3/4[\hspace{-.2cm}\mod1},\quad \text{for}\quad x\in X$$Note that $\eta$ is also a dense $k$-DHGM since $\inf_{x\in X}\dim_H(\supp \eta^x)=k-1=\dim_H X^{k-1}$.
\end{example}

\begin{example}
  Let $k\in\N\setminus\{1,2\}$ and $X=\T$. Define $\eta\in\cB(X,\cM_+(X^{k-1}))$: $$\eta^x=\otimes_{j=1}^{k-1}\eta_j^{x},\quad x\in X,$$ where for $j=1,\ldots,k-1$,
  \[\eta^{x}_j=\begin{cases}
    \delta_{\langle x+1/4\rangle}+\delta_{\langle x+3/4\rangle},& \text{if}\ j\ \text{is odd},\\
    \lambda\bigl|_{[x+1/4,x+3/4[\hspace{-.2cm}\mod1},&\text{if}\ j\ \text{is even}.
    \end{cases}\] Then by Definition~\ref{Def-dense-sparse}, $\eta$ is neither dense nor sparse.
\end{example}

\begin{example}\label{spherical-hypergraphop}
Let $d\in\N$, $2\le k\le d$ be an integer and $X=\mathbb{S}^d$ be the $d$-dimensional unit sphere. Let $\mathcal{E}=\{(x_1,\ldots,x_k)\in X^k\colon x_i\cdot x_j=0,\quad \text{for any}\quad i\neq j,\quad 1\le i,j\le k\}$ be the set of hyperedges. Then $(X,\mathcal{E})$ can be regarded as a generalized $k$-uniform hypergraph. Define $\eta\in\cB(X,\cM_+(X^{k-1}))$ to represent this generalized hypergraph: 
\[\eta^{x_1}=\lambda|_{S(x_1)},\quad x_1\in X,\]
where $$S(x_1)\coloneqq\{(x_2,\ldots,x_k)\in X^{k-1}\colon x_2\in x_1^{\perp},\ x_3\in x_1^{\perp}\cap x_2^{\perp},\ldots,\ x_{\ell}\in\cap_{j=1}^{\ell-1}x_j^{\perp},\ldots,\ x_k\in\cap_{j=1}^{k-1}x_j^{\perp}\},$$ where $a^{\perp}\coloneqq\{y\in X\colon a\cdot y=0\}$. It is easy to verify that $\eta$ is a HGM. Note that $$0<\dim_{H}(S(x_1))=\sum_{j=1}^{k-1}(d-j)=\tfrac{(2d-k)(k-1)}{2}<\dim_H(X^{k-1})=d(k-1)$$ Hence $\eta$ is neither dense nor sparse. We call this  $\eta$ \emph{$d$-dimensional spherical $k$-uniform HGM}.
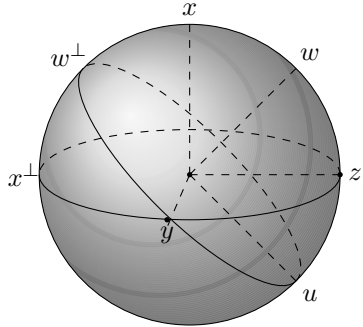
\begin{figure}[h]
  \begin{tikzpicture}
  \shade[ball color = gray!40, opacity = 0.4] (0,0) circle (2cm);
  \draw (0,0) circle (2cm);
  \draw[] (-2,0) arc (180:360:2 and 0.6);
  \draw[dashed] (2,0) arc (0:180:2 and 0.6);
  \draw[rotate=-45] (-2,0) arc (180:360:2 and 0.6);
  \draw[rotate=-45,dashed] (2,0) arc (0:180:2 and 0.6);
  \fill[fill=black] (0,0) circle (1pt);
  \fill[fill=black] (-0.3,-0.6) circle (1pt);
  \fill[fill=black] (2,0) circle (1pt);
  \draw[dashed] (0,0) -- node[above]{} (0,2);
  \draw[dashed] (0,0) -- node[above]{} (1.414,1.414);
  \draw[dashed] (0,0) -- node[above]{} (2,0);
\draw[dashed] (0,0) -- node[above]{} (-0.3,-0.7);
\draw[dashed] (0,0) -- node[above]{} (1.414,-1.414);
  \node[] at (0,2.2) {$x$};
  \node[] at (1.6,1.6) {$w$};
  \node[] at (-0.3,-0.8) {$y$};
  \node[] at (2.2,0) {$z$};
  \node[] at (-2.2,0) {$x^{\perp}$};
  \node[] at (-1.6,1.6) {$w^{\perp}$};
 \node[] at (1.6,-1.6) {$u$};
\end{tikzpicture}
\caption{$2$-dimensional spherical $3$-uniform HGM. $X=\mathbb{S}^2$. $\eta^x=\lambda|_{S(x)}$ with $S(x)=\{(y,z)\in X^2\colon y\in x^{\perp},\ z\in x^{\perp}\cap y^{\perp}\}$. Note that $\dim_H(S(x))=1$.}
\end{figure}
\end{example}

\section{Generalized interacting particle system}\label{sect-characteristics}
To state a generalized IPS on hypergraphs, we first collect some standing assumptions. 

\subsection{Assumptions}

\hfill\\

\noindent{$\mathbf{(A1)}$} $(X,\mathfrak{B}(X),\mu_X)$ is a compact Polish probability space equipped with the metric induced by the $\ell_1$-norm of $\R^{r_1}\supseteq X$.

\medskip

\noindent{$\mathbf{(A2)}$} Let $J\subseteq\mathbb{N}\setminus\{1\}$. For $\ell\in J$, $(t,u)\mapsto g_{\ell}(t,u)\in\R^{r_2(\ell-1)}$ is continuous in $t\in\cI$, and locally Lipschitz continuous in $u\in \R^{\ell r_2}$ uniformly in $t$, i.e., for every $u\in \R^{\ell r_2}$, there exists a neighbourhood $\mathcal{N}\subseteq\R^{\ell r_2}$ of $u$ such that
$$\sup_{t\in\cI}\sup_{\begin{array}{l}
  u_1\neq u_2,\\
u_1,u_2\in\mathcal{N}
\end{array}}\frac{|g_{\ell}(t,u_1)-g_{\ell}(t,u_2)|}{|u_1-u_2|}<\infty,$$
for $\ell\in J$.\\
\medskip
\noindent{$\mathbf{(A3)}$} $(t,x,\phi)\mapsto h(t,x,\phi)\in\R^{r_2}$ is continuous in $t\in\cI$, and locally Lipschitz continuous in $\phi\in \R^{r_2}$ uniformly in $(t,x)$, i.e., for every $\phi\in \R^{r_2}$ for some $r_2\in\N$, there exists a neighbourhood $\mathcal{N}\subseteq\R^{r_2}$ of $\phi$ such that $$\sup_{t\in\cI}\sup_{x\in X}\sup_{\tiny\begin{array}{l}
  \phi_1\neq \phi_2,\\
\phi_1,\phi_2\in\mathcal{N}
\end{array}}\frac{|h(t,x,\phi_1)-h(t,x,\phi_2)|}{|\phi_1-\phi_2|}<\infty.$$

\medskip
\noindent{$\mathbf{(A4)}$} $\eta_{\ell}\in \mathcal{B}(X,\cM_+(X^{\ell-1}))$, for $\ell\in J$.

\medskip
\noindent{$\mathbf{(A4)'}$} $\eta_{\ell}\in \mathcal{C}(X,\cM_+(X^{\ell-1}))$, for $\ell\in J$.

\medskip
\noindent{$\mathbf{(A5)}$} $\nu_{\cdot}\in \mathcal{C}(\cI,\mathcal{B}_{*}(X,\cM_+(\R^{r_2})))$ is \emph{uniformly compactly supported} in the sense that
there exists a compact set $E_{\nu_{\cdot}}\subseteq\R^{r_2}$ such that $\cup_{t\in\cI}\cup_{x\in X}\supp\nu_t^x\subseteq E_{\nu_{\cdot}}$.

\medskip
\noindent ($\mathbf{A6}$) There exists a convex compact set $Y\subseteq\R^{r_2}$
 such that for all $\nu_{\cdot}$ satisfying ($\mathbf{A5}$) uniformly supported within $Y$, the following inequality holds:
\[V[\eta,\nu_{\cdot},h](t,x,\phi)\cdot\upsilon(\phi)\le0,\quad \text{for all}\ t\in\mathcal{I},\ x\in X,\ \phi\in\partial Y,\]
where $\partial Y=Y\cap\overline{\R^{r_2}\setminus Y}$, $\upsilon(\phi)$ is the outer normal vector at $\phi$, and
\begin{multline}\label{Eq-VOperator}
V[\eta,\nu_{\cdot},h](t,x,\phi)
=\sum_{\ell\in J}\int_{X^{\ell-1}}\underbrace{\int_{\R^{r_2}}\ldots\int_{\R^{r_2}}}_{\ell-1}g_{\ell}
(t,\phi,\psi_1,\ldots,\psi_{\ell-1})
\rd\nu_t^{y_{\ell-1}}(\psi_{\ell-1})
\cdots\rd\nu_t^{y_{1}}(\psi_{1})\\
\cdot\rd\eta^x_{\ell}(y_1,\ldots,y_{\ell-1})+h(t,x,\phi),\quad t\in\cI, x\in X,\ \phi\in \R^{r_2}
\end{multline}
\medskip

\noindent{$(\mathbf{A7})$}  $(t,x,\phi)\mapsto h(t,x,\phi)\in\R^{r_2}$ is continuous in $x$ uniformly in $\phi$:
\[\lim_{|x-x'|\to0}\sup_{\phi\in Y}|h(t,x,\phi)-h(t,x',\phi)|=0,\quad t\in\cI,\]
where $Y$ is the compact set given in ($\mathbf{A6}$). Moreover, $h$ is integrable uniformly in $x$:
\[\int_0^T\int_Y\sup_{x\in X}|h(t,x,\phi)|\rd\phi\rd t<\infty.\]

Under {$\mathbf{(A1)}$}-{$\mathbf{(A5)}$}, the \emph{Vlasov operator} $V$ given in \eqref{Eq-VOperator} is well defined.

Now we provide some intuitive explanation for these assumptions. These assumptions can be regarded as analogues of those for digraph measures in \cite{KX21}, in the context of DHGMs. Assumption $(\mathbf{A1})$ means that the underlying generalized directed hyper-digraphs (DHGMs) have the same compact vertex space $X$. Such compactness is crucial in establishing discretization of DHGMs. Assumptions $(\mathbf{A2})$-$(\mathbf{A3})$ are the standard Lipschitz conditions for the well-posedness of (non-local) ODE models. Assumption $(\mathbf{A4})$ means that we interpret the hyper-graphs as measure-valued functions (see also \cite{KX21}); note that we can think of $\eta^x_{\ell}$ as the local hyper-edge density or connectivity near vertex $x$. Next, we need the assumption for the approximation of the VE (i.e., the mean field equation for the IPS) that the family of DHGMs $\eta^x$ are \emph{continuous} in the vertex variable $x$, which is encoded in assumption $(\mathbf{A4})'$ (essentially used in Lemma~\ref{le-graph}). As mentioned in \cite{KX21}, $\mathbf{(A4)'}$ is sufficient but not necessary for the approximation results. For instance, one can relax this assumption by allowing $x\mapsto\eta_{\ell}^x$ ($\ell\in J$) to have finitely many discontinuity points.

Next, we propose a generalized network on $\# J$ limits of sequences of directed hypergraphs of cardinality $\ell$, for $\ell\in J$.
\begin{equation}\label{general-characteristic}
  \begin{split}
    \frac{\rd\phi(t,x)}{\rd t}=&V[\eta,\nu_{\cdot},h](t,x,\phi(t,x)),\quad t\in\cI,\quad x\in X,\\
    \phi(0,x)=&\varphi(x),\hspace{3cm}  x\in X.
  \end{split}
\end{equation}

Indeed, a special case of the network \eqref{general-characteristic} is the following discrete set of ODEs:

\begin{equation}\label{Kuramoto}
  \begin{split}
    \dot{\phi}_i^N(t)    =&h_i^N(t,\phi^N_i(t))+\sum_{\ell\in J}\frac{1}{N^{\ell-1}}\sum_{j_1=1}^N\cdots\sum_{j_{\ell-1}=1}^N W^{\ell,N}_{i,j_1,\ldots,j_{\ell-1}}g_{\ell}(t,\phi_i^N(t),\phi_{j_1}^N(t),\ldots,\phi_{j_{\ell-1}}^N(t)).
  \end{split}
\end{equation}

\emph{We emphasize that $J$ is independent of $N$.} To verify that \eqref{Kuramoto} is a special case of \eqref{general-characteristic}, let $X=[0,1]$, and $\{I^N_i\}_{i=1}^N$ be an equipartition of $X$ with $I_i^N$ being defined in Example~\ref{Eq-equipartition}, $\mu_{X}\in\cP(X)$ be the reference measure,
$$\nu_t^{x}=\delta_{\phi^N_i(t)},\quad x\in I^N_i,\quad i=1,\ldots,N,$$ $$\frac{\rd\eta_{\ell,N}^{x}(y_1,\ldots,y_{\ell-1})}{\prod_{j=1}^{\ell-1}\rd \mu_{X}(y_j)}=W^{\ell,N}_{i_1,\ldots,i_{\ell-1}},\quad (y_1,\ldots,y_{\ell-1})\in I^N_{i_1}\times\ldots\times I^N_{i_{\ell-1}},\quad \ell\in J,$$ $$\phi(t,x)=\phi^N_i(t),\quad h(t,x,\phi(t,x))=h^N_i(t,\phi^N_i(t)),\quad x\in I^N_i.$$
Substituting the expressions above into   \eqref{general-characteristic} simply yields \eqref{Kuramoto}.

The following well-posedness of the network \eqref{Kuramoto} is a standard result of ODE theory \cite{T12}.
\begin{proposition}\label{prop-discrete-characteristic}
  Assume $h_i^N$ (for $i=1,\ldots,N$) and $g_{\ell}$ (for $\ell\in J$) are locally Lipschitz. Then there exists a local solution to the IVP of \eqref{Kuramoto}. In particular, if there exists a compact positively invariant set for \eqref{Kuramoto}, then the solution is global.
\end{proposition}

The IVP of \eqref{general-characteristic} confined to a finite time interval $\cI$ is the so-called \emph{fiberized equation of characteristics} (or \emph{fiberized characteristic equation}) \cite{KM18,KX21}. When the underlying space $X$ is finite, and the measures $\nu_t^x$ and $\eta^x_{\ell}$ for all $x\in X$ are finitely supported, \eqref{general-characteristic} becomes a system of ODEs as \eqref{Kuramoto} coupled on a finite set of directed graphs in terms of $\{\eta_{\ell}\}_{\ell\in J}$. Hence, the fiberized characteristic equation connects a finite-dimensional IPS and the VE, while maintaining the information about both systems. Analogous to Proposition~\ref{prop-discrete-characteristic}, the well-posedness of \eqref{general-characteristic} is also a standard result from ODE theory \cite{T12}.

\begin{theorem}\label{theo-well-posedness-characteristic}
Assume ($\mathbf{A1}$)-($\mathbf{A5}$). Let $\phi_0\in\mathcal{B}(X,\R^{r_2})$. Then for every $x\in X$ and $t_0\in\cI$, there exists a solution $\phi(t,x)$ to the IVP of \eqref{general-characteristic} with $\phi(t_0,x)=\phi_0(x)$ for all $t\in(T^{x,t_0}_{\min},T^{x,t_0}_{\max})\cap\cI$ with $(T^{x,t_0}_{\min},T^{x,t_0}_{\max})\subseteq\R$ being a neighbourhood of $t_0$ such that
\begin{enumerate}
\item[(i)] either (i-a) $T^{x,t_0}_{\max}>T$ or (i-b) $T^{x,t_0}_{\max}\le T$ and $\lim_{t\uparrow T^{x,t_0}_{\max}}|\phi(t,x)|=\infty$ holds, and
\item[(ii)] either (ii-a) $T^{x,t_0}_{\min}<0$ or (ii-b) $T^{x,t_0}_{\min}\ge0$ and $\lim_{t\downarrow T^{x,t_0}_{\min}}|\phi(t,x)|=\infty$ holds.
\end{enumerate}
In addition, assume ($\mathbf{A6}$) and $\nu_{\cdot}$ is uniformly supported within $Y$, then $(T^{x,t_0}_{\min},T^{x,t_0}_{\max})\cap\cI=\cI$ for all $x\in X$, and there exists a {fiberized flow of the vector field $V[\eta,\nu_{\cdot},h]$ such that
\begin{align*}
\frac{d}{dt}\Phi^x_{t,0}[\eta,\nu_{\cdot},h](\phi)=&V[\eta,\nu_{\cdot},h](t,x,\Phi^x_{t,0}[\eta,\nu_{\cdot},h](\phi)),\quad t\in\cI,\\
\Phi^x_{0,0}[\eta,\nu_{\cdot},h](\phi)=&\phi
\end{align*}}
\end{theorem}

\section{Well-posedness of the {Vlasov} equation}\label{sect-Vlasov}

From Theorem~\ref{theo-well-posedness-characteristic}, we have for all $x\in X$, $$\Phi^x_{t,0}[\eta,\nu_{\cdot},h])^{-1}=\Phi^x_{0,t}[\eta,\nu_{\cdot},h],\quad t\in\cI.$$ The pushforward under the flow $\Phi^x_{t,0}[\eta,\nu_{\cdot},h]$ of an initial measure $\nu_0^x\in\mathcal{B}_*(X,\cM_+(Y))$ defines another time-dependent measure in $\mathcal{B}_*(X,\cM_+(Y))$ via the following \emph{fixed point equation} \begin{equation}\label{Fixed}
  \nu_t^x=(\mathcal{A}[\eta,h]\nu)_t^x,\quad t\in\cI.
\end{equation} where {$(\mathcal{A}[\eta,h]\nu)_t^x\coloneqq\Phi^x_{0,t}[\eta,\nu_{\cdot},h]_{\sharp}\nu_0^x$ denotes the push-forwards of $\nu_0^x$ along $\Phi^x_{0,t}[\eta,\nu_{\cdot},h]$.} In particular, if $\nu_{\cdot}\in \mathcal{C}(\cI,\mathcal{B}_*(X,\cM_{+,\abs}(Y)))$, then $\cA[\eta,h]\nu_{\cdot}\in \mathcal{C}(\cI,\mathcal{B}_*(X,\cM_{+,\abs}(Y)))$ by the positive invariance of $Y$. Hence the \emph{Vlasov operator} can be represented in terms of the density $\rho(t,y,\phi)\coloneqq\frac{\rd\nu_t^y(\phi)}{\rd\mu_X(y)\rd\phi}$ for every $t\in\cI$: \begin{equation}\label{Vrho-0}
\begin{split}
\widehat{V}[\eta,\rho(\cdot),h](t,x,\phi)=&
\sum_{\ell\in J}\int_{X^{\ell-1}}\underbrace{\int_{Y}\cdots\int_{Y}}_{\ell-1}
g_{\ell}(t,\phi,\psi_1,\psi_2,\ldots,\psi_{\ell-1})\prod_{\ell=1}^{\ell-1}
\rho(t,y_{\ell},\psi_{\ell})\rd\psi_{1}\cdots\rd\psi_{\ell-1}\\
&\rd\eta^{x}_{\ell}(y_1,\ldots,y_{\ell-1})+h(t,x,\phi).
    \end{split}
\end{equation}

Let ${\sf L}^1(X\times Y;\mu_X\otimes\lambda)$ be the space of all integrable functions w.r.t. the reference measure $\mu_X\otimes\lambda$. Let \[\begin{split}{\sf L}^1_+(X\times Y;\mu_X\otimes\lambda)=&\Bigl\{f\in {\sf L}^1(X\times Y;\mu_X\otimes\lambda)\colon \int_{X\times Y}f\rd\mu_X\rd\lambda=1,\\ &\hspace{4.2cm} f\ge0,\ \mu_X\otimes\lambda\ \text{a.e. on}\ X\times Y\Bigr\},\end{split}\] be the space of densities of probabilities on $X\times Y$. Conversely, for every function $\rho\colon\mathcal{I}\to{\sf L}^1_+(X\times Y;\mu_X\otimes\lambda)$, for $(t,y)\in\mathcal{I}\times X$, $$\rd\nu_t^y(\phi)=\rho(t,y,\phi)\rd\phi$$ defines $\nu_{\cdot}\in\mathcal{B}(\mathcal{I},\mathcal{B}_{*}(X,\cM_+(Y)))$. Hence \eqref{Vrho-0} can be transformed to the Vlasov operator \eqref{Eq-VOperator} in terms of $\nu_{\cdot}$.

Let $\rho_0\colon X\times Y\to\R_+$ be continuous in $x$ for $\lambda$-a.e. $\phi\in Y$, and integrable in $\phi$ for every $x\in X$ such that $$\int_X\int_Y\rho_0(x,\phi)\rd\phi\rd\mu_X(x)=1.$$  Consider the VE
\begin{align}
 \label{Vlasov} &\frac{\partial\rho(t,x,\phi)}{\partial t}+\textrm{div}_{\phi}\left(\rho(t,x,\phi)\widehat{V}[\eta,\rho(\cdot),h](t,x,\phi)\right)=0, t\in(0,T],\ x\in X,\ \lambda\text{-a.e.}\ \phi\in Y,\\
\nonumber  &\rho(0,\cdot)=\rho_0(\cdot).
  \end{align}
  First, let us define the weak solution to  \eqref{Vlasov}. This definition is adapted from \cite[Definition~4.6]{KX21}.

\begin{definition}\label{def-weak-sol}
Let $Y$ be a compact positively invariant subset of \eqref{general-characteristic} given in Theorem~\ref{theo-well-posedness-characteristic}. We say $\rho\colon\cI\times X\times Y\to\R^{+} $ is a \emph{uniformly weak solution} to the IVP \eqref{Vlasov} if for every $x\in X$, the following three conditions are satisfied:
\medskip

\noindent(i) Normalization. $\int_X\int_Y\rho(t,x,\phi)\rd\phi\rd x=1$, for all $t\in\cI$.
\medskip

  \noindent(ii) Uniform weak continuity. The map $t\mapsto\int_Yf(\phi)\rho(t,x,\phi)\rd\phi$ is continuous uniformly in $x\in X$, for every $f\in \mathcal{C}(Y)$.
\medskip

  \noindent(iii) Integral identity: For all test functions $w\in \mathcal{C}^1(\cI\times Y)$ with $\supp w\subseteq [0,T[\times U$ and $U\subset\subset Y$, the equation below holds: \begin{multline}\label{eq-test}
    \int_0^T\int_Y\rho(t,x,\phi)\lt(\frac{\partial w(t,\phi)}{\partial t}+\widehat{V}[\eta,\rho(\cdot),h](t,x,\phi)\cdot\nabla_{\phi} w(t,\phi)\rt)\rd\phi\rd t\\
    +\int_Yw(0,\phi)\rho_0(x,\phi)\rd\phi=0,
  \end{multline}
  where $\supp w=\overline{\{(t,u)\in \cI\times Y\colon w(t,u)\neq0\}}$ is the support of $w$, and $\widehat{V}[\eta,\rho(\cdot),h]$ is given in \eqref{Vrho-0}.
\end{definition}

The well-posedness of the VE associated with the generalized IPS depends on continuity properties of the operator $\cA$. We comment that Definition~\ref{def-weak-sol} is well-posed, c.f. \cite[Remark~4.7]{KX21}.

We first provide the continuity property of $\cA$.

\begin{proposition}\label{prop-continuousdependence}
Assume $\mathbf{(A1)}$-$\mathbf{(A6)}$. Denote $\eta=\{\eta_{\ell}\}_{\ell\in J}$.
  \begin{enumerate}
    \item[(i)] Continuity in $t$. $$t\mapsto(\cA[\eta,h]\nu_{t}\in \mathcal{C}(\cI,\mathcal{B}_{*}(X,\cM_+(Y))).$$ In particular,
    if $\nu_{\cdot}\in\mathcal{C}(\cI,$ $\mathcal{C}_{*}(X,\cM_+(Y)))$, then $\cA[\eta,h]\nu_{\cdot}$ $\in \mathcal{C}(\cI,\mathcal{C}_{*}(X,$ $\cM_+(Y)))$.
    Moreover, the mass conservation law holds: $$\cA[\eta,h]\nu_t^x(Y)=\nu_0^x(Y),\quad \forall x\in X.$$
        \item[(ii)] Lipschitz continuity in $\nu_{\cdot}$. For $\nu^1,\nu^2\in{\mathcal{C}_{*}(X,\cM_+(Y)))}$, we have \[d_{\infty}(\cA[\eta,h]\nu^1_{t},\cA[\eta,h]\nu^2_{t})
            \le\textnormal{e}^{L_1t}d_{\infty}(\nu^1_0,\nu^2_0)+ L_2\textnormal{e}^{L_1t}\int_0^td_{\infty}(\nu^1_{\tau},\nu^2_{\tau})
  \textnormal{e}^{-L_1\tau}\rd\tau,\]where $L_1\coloneqq L_1(\nu^2_{\cdot})$ and $L_2\coloneqq L_2(\nu^1_{\cdot},\nu^2_{\cdot})$ are constants.
  \item[(iii)] Lipschitz continuity of $\cA[\eta,h]$ in $h$. For $h_1,\ h_2$ satisfying $\mathbf{(A3)}$ and $\mathbf{(A7)}$ with $h$ replaced by $h_i$ for $i=1,2$,
  \[d_{\infty}(\cA[\eta,h_1]\nu_{t},\cA[\eta,h_2]\nu_{t})\le
      L_3\|h_1-h_2\|_{\infty,\cI},\]
      where $L_3\coloneqq L_3(\nu_{\cdot})$ is a constant.
  \item[(iv)] Absolute continuity. If $\nu_0\in \mathcal{B}_{*}(X,\cM_{+,\abs}(Y))$, then $$\cA[\eta,h]\nu_t\in \mathcal{B}_{*}(X,\cM_{+,\abs}(Y)),\quad \forall t\in\cI.$$
  \end{enumerate}
\end{proposition}

The proof of Proposition~\ref{prop-continuousdependence} is provided in Appendix~\ref{appendix-prop-continuousdependence}.

\begin{proposition}\label{prop-sol-fixedpoint}
Assume $\mathbf{(A1)}$-$\mathbf{(A4)}$ and $\mathbf{(A6)}$-$\mathbf{(A7)}$. Let $\nu_0\in \mathcal{B}_{*}(X,\cM_+(Y))$, and $L_1$, $L_2$, and $L_3$ be defined as in Proposition~\ref{prop-continuousdependence}. Then there exists a unique solution $\nu_{\cdot}\in\mathcal{C}(\cI,\mathcal{B}_{*}(X,\cM_+(Y)))$ to the fixed point equation
\eqref{Fixed}. In addition, if  $\mathbf{(A4)'}$ holds and $\nu_0\in \mathcal{C}_{*}(X,\cM_+(Y))$, then $\nu_{\cdot}\in\mathcal{C}(\cI,\mathcal{C}_{*}(X,$ $\cM_+(Y)))$; if $\nu_0\in \mathcal{B}_{*}(X,\cM_{+,\abs}(Y))$, then $$\nu_{t}\in \mathcal{B}_{*}(X,\cM_{+,\abs}(Y)),\quad \forall t\in\cI.$$
 Moreover, the solutions have continuous dependence on
\begin{enumerate}
\item[\textnormal{(i)}] the initial conditions:
\[d_{\infty}(\nu_t^1,\nu_t^2)\le \textnormal{e}^{(L_1(\nu^2_{\cdot})+L_2\|\nu^1_{\cdot}\|)t}d_{\infty}(\nu_0^1,\nu_0^2),\quad t\in\cI,\]
where $\nu^i_{\cdot}$ is the solution to \eqref{Fixed} with initial condition $\nu^i_0$ for $i=1,2$.
\item[\textnormal{(ii)}] $h$: {Assume $\nu_0^1=\nu_0^2$. Then} $$d_{\infty}(\nu_t^1,\nu_t^2)\le \frac{1}{L_3(\nu^2_{\cdot})}\|\nu^1_{\cdot}\|
    \textnormal{e}^{\BL(h_2)\textnormal{e}^{L_1(\nu^2_{\cdot})T}T}
    \textnormal{e}^{(L_1(\nu^2_{\cdot})+L_2\|\nu^1_{\cdot}\|)t}\|h_1-h_2\|_{\infty},$$
where $\nu^i_{\cdot}$ is the solution to \eqref{Fixed} with  $h$ replaced by $h_i$ for $i=1,2$.
\item[\textnormal{(iii)}] $\eta$: Let $\{\eta^\ell\}_{K\in\N}\subseteq \mathcal{B}(X,\cM_+(X^{\ell-1}))$ such that $\lim_{K\to\infty}d_{\infty}(\eta_{\ell},\eta^{K}_{\ell})=0$, for $\ell\in J$. Assume $\nu_0,\ \nu^K_0\in \mathcal{C}_{*}(X,\cM_+(Y))$ with $$\lim_{K\to\infty}d_{\infty}(\nu_0,\nu^K_0)=0.$$ Then
    $$\lim_{K\to\infty}\sup_{t\in\cI}d_{\infty}(\nu_t,\nu^K_t)=0,$$
where $\nu^K_{\cdot}$ is the solution to \eqref{Fixed} with $\eta_{\ell}$ replaced by $\eta_{\ell}^K$ for $\ell\in J$ and $K\in\N$.
\end{enumerate}
\end{proposition}
The proof of Proposition~\ref{prop-sol-fixedpoint} is provided in Appendix~\ref{appendix-prop-sol-fixedpoint}.

With the above assumptions and under appropriate metrics, one can show that the operator defined in \eqref{Fixed} $$\mathcal{A}=(\mathcal{A}^x)_{x\in X}\colon \mathcal{C}(\cI,\mathcal{B}_{*}(X,\mathcal{M}_+(Y)))\to  \mathcal{C}(\cI,\mathcal{B}_{*}(X,\mathcal{M}_+(Y)))$$ is a contraction.
Now we obtain the well-posedness of the VE \eqref{Vlasov}.

\begin{theorem}\label{thm-well-posedness}
Assume ($\mathbf{A1}$)-($\mathbf{A4}$) and ($\mathbf{A6}$). Let $\rho_0\in {\sf L}^1_+(X\times Y;\mu_X\otimes\lambda)$. Assume additionally that $\rho_0(x,\phi)$ is continuous in $x\in X$ for $\lambda$-a.e. $\phi\in Y$. Then there exists a unique uniform weak solution to the IVP of \eqref{Vlasov} with initial condition $\rho(0,x,\phi)=\rho_0(x,\phi)$, $x\in X$, $\phi\in Y$.
\end{theorem}
\begin{proof}
  The proof is the same as that of \cite[Theorem~4.8]{KX21}, which is independent of the specific form of $V$ and $\widehat{V}$, but based on the continuous dependence properties given in Proposition~\ref{prop-continuousdependence} and Proposition~\ref{prop-sol-fixedpoint}.
\end{proof}

\section{Approximation of time-dependent solutions to VE}\label{sect-approximation}
In this section, we study approximation of the solution to the VE \eqref{eq-test}.

Based on the continuous dependence of solutions to the fixed point equation on the underlying DHGMs $\{\eta_{\ell}\}_{\ell\in J}$, on the initial measure $\nu_0$, as well as on function $h$ established in Proposition~\ref{prop-sol-fixedpoint}, together with the recently established results on \emph{deterministic empirical approximation of positive measures} \cite{XB19,C18,BJ22} (see Propositions~\ref{le-ini-2} and \ref{le-graph} below), we will establish the discretization of solutions of VE over finite time interval $\cI$ by a sequence of discrete ODE systems coupled on finite directed hypergraphs converging \emph{weakly} to the DHGMs $\{\eta_{\ell}\}_{\ell\in J}$ (Theorem~\ref{th-approx} below).

Beforehand, let us recall some approximation results from \cite{KX21}.

\begin{proposition}
  [Partition of $X$]\cite[Lemma~5.4]{KX21}\label{le-partition}
Assume  $(\mathbf{A1})$. Then there exists a sequence of pairwise disjoint partitions $\{A^m_i\colon i=1,\ldots,m\}_{m\in\N}$ of $X$ such that
  $X=\cup_{i=1}^m A_i^m$ for every $m\in\N$ and
  \[\lim_{m\to\infty}\max_{1\le i\le m}\Diam A^m_i=0.\]
\end{proposition}

\begin{proposition}[Approximation of the initial distribution]\cite[Lemma~5.5]{KX21}\label{le-ini-2}
Assume $(\mathbf{A1})$ and $\nu_0\in\mathcal{B}_*(X,\cM_+(Y))$. Let $\{A^m_i\}_{1\le i\le m}$ be a partition of $X$ for $m\in \N$ satisfying \[\lim_{m\to\infty}\max_{1\le i\le m}\Diam A^m_i=0.\] Let $x_i^m\in A_i^m$, for $i=1,\ldots,m$, $m\in\N$. Then there exists a sequence $\{\varphi^{m,n}_{(i-1)n+j}\colon i=1,\ldots,m,j=1,\ldots,n\}_{n,m\in\N}\subseteq Y$ such that $$\lim_{m\to\infty}\lim_{n\to\infty}d_{\infty}(\nu_0^{m,n},\nu_0)=0,$$
where $\nu_0^{m,n}\in \mathcal{B}_{*}(X,\cM_+(Y))$ with $$\nu_0^{m,n,x}\coloneqq\sum_{i=1}^m\mathbbm{1}_{A^m_i}(x)\frac{a_{m,i}}{n}\sum_{j=1}^n
\delta_{\varphi^{m,n}_{(i-1)n+j}},\quad x\in X,$$
$$a_{m,i}=\begin{cases}
  \frac{\int_{A_i^m}\nu_0^x(Y)\rd\mu_X(x)}{\mu_X(A_i^m)},& \textnormal{if}\quad \mu_X(A_i^m)>0,\\
  \nu_0^{x^m_i}(Y),& \textnormal{if}\quad \mu_X(A_i^m)=0.
\end{cases}$$
\end{proposition}

\begin{proposition}[Approximation of the DHGM]\label{le-graph}
Assume $(\mathbf{A1})$ and $(\mathbf{A4})'$. For every $m\in\N$, let $A^m_i$ and $x_i^m$ be defined in Proposition~\ref{le-ini-2} for $i=1,\ldots,m$, $m\in\N$. Then for every $\ell\in J$, there exists a sequence $\bigl\{y^{\ell,m,n}_{(i-1)n+j}\colon i=1,\ldots,m,j=1,\ldots,n\bigr\}_{m,n\in\N}\subseteq X^{\ell-1}$ such that $$\lim_{m\to\infty}\lim_{n\to\infty}d_{\infty}(\eta_{\ell}^{m,n},\eta_{\ell})=0,$$
where $\eta_{\ell}^{m,n}\in \mathcal{B}(X,\cM_+(X^{\ell-1}))$ with $$\eta_{\ell}^{m,n,x}\coloneqq\sum_{i=1}^m\mathbbm{1}_{A^m_i}(x)\frac{b_{\ell,m,i}}{n}
\sum_{j=1}^n\delta_{y^{\ell,m,n}_{(i-1)n+j}},\quad x\in X,$$ $$b_{\ell,m,i}=\begin{cases}
  \frac{\int_{A_i^m}\eta^{x}_{\ell}(X)\rd\mu_X(x)}{\mu_X(A_i^m)},& \textnormal{if}\quad \mu_X(A_i^m)>0,\\
  \eta_{\ell}^{x_i^m}(X),& \textnormal{if}\quad \mu_X(A_i^m)=0.
\end{cases}$$
\end{proposition}
The proof of Proposition~\ref{le-graph} is analogous to that of \cite[Lemma~5.6]{KX21} and thus is omitted.
\begin{proposition}[Approximation of $h$]\cite[Lemma~5.9]{KX21}\label{le-h}
Assume $(\mathbf{A3})$ and $(\mathbf{A7})$.

For every $m\in\N$, let $x_i^m$ be defined in Proposition~\ref{le-ini-2} and $$h^m(t,z,\phi)=\sum_{i=1}^{{m}}\mathbbm{1}_{A^m_i}(z)h(t,x_i^m,\phi),\quad t\in\cI,\ z\in X,\ \phi\in Y.$$ Then
  $$\lim_{m\to\infty}\int_0^T\int_Y\sup_{x\in X}\lt|h^m(t,x,\phi)-h(t,x,\phi)\rt|\rd\phi\rd t=0.$$
\end{proposition}

Now we are ready to provide a discretization of the VE on the DHGM by a sequence of ODEs.
To summarize, there exists
 \begin{enumerate}
 \item[$\bullet$] a partition $\{A^m_i\}_{1\le i\le m}$ of $X$ and points $x^m_i\in A^m_i$ for $i=1,\ldots,m$, for every $m\in\N$,
 \item[$\bullet$]  a sequence $\{\varphi^{m,n}_{(i-1)n+j}\colon i=1,\ldots,m,j=1,\ldots,n\}_{n,m\in\N}\subseteq Y^{\ell-1}$ and $\{a_{m,i}\colon i=1,\ldots,$ $m\}_{m\in\N}\subseteq\R_+$, for $\ell\in J$, and
  \item[$\bullet$] a sequence $\{y^{\ell,m,n}_{(i-1)n+j}\colon i=1,\ldots,m,j=1,\ldots,n\}_{m,n\in\N}\subseteq X^{\ell-1}$ and $\{b_{\ell,m,i}\colon i=1,\ldots,$ $m\}_{m\in\N}\subseteq\R_+$, for $\ell\in J$,
 \end{enumerate}
 such that
  $$\lim_{m\to\infty}\lim_{n\to\infty}d_{\infty}(\nu_0^{m,n},\nu_0)=0,$$
   $$\lim_{m\to\infty}\lim_{n\to\infty}d_{\infty}(\eta_{\ell}^{m,n},\eta_{\ell})=0,\quad \ell\in J,$$
     $$\lim_{m\to\infty}\int_0^T\int_Y\sup_{x\in X}\lt|h^m(t,x,\phi)-h(t,x,\phi)\rt|\rd\phi\rd t=0,$$
 where  \begin{equation}
   \label{discretization}
 \begin{split}
\nu_0^{m,n,x}\coloneqq\sum_{i=1}^m\mathbbm{1}_{A^m_i}(x)\frac{a_{m,i}}{n}\sum_{j=1}^n
\delta_{\varphi^{m,n}_{(i-1)n+j}},\quad x\in X,\\ \eta_{\ell}^{m,n,x}\coloneqq\sum_{i=1}^m\mathbbm{1}_{A^m_i}(x)\frac{b_{\ell,m,i}}{n}
\sum_{j=1}^n\delta_{y^{\ell,m,n}_{(i-1)n+j}},\quad x\in X,\\
 h^m(t,z,\phi)\coloneqq\sum_{i=1}^m\mathbbm{1}_{A^m_i}(z)h(t,x_i^m,\phi),\quad t\in\cI,\ z\in X,\ \phi\in Y.
 \end{split}
 \end{equation}
 Consider the following IVP of a coupled ODE system:
\begin{multline}
  \label{lattice}
  \dot{\phi}_{(i-1)n+j}=F^{m,n}_{i}(t,\phi_{(i-1)n+j},\Phi),\quad 0<t\le T,\quad \phi_{(i-1)n+j}(0)=\varphi^{m,n}_{(i-1)n+j},\\ i=1,\ldots,m,\ j=1,\ldots,n,
\end{multline}
where $\Phi=(\phi_{(i-1)n+j})_{1\le i\le m,1\le j\le n}$ and \begin{align*}
F^{m,n}_{i}(t,\psi,\Phi)=&  \sum_{\ell\in J}\frac{b_{\ell,m,i}}{n}\sum_{j=1}^n\sum_{p_1=1}^m\frac{a_{m,p_1}}{n}\mathbbm{1}_{A^m_{p_1}}
(y^{\ell,m,n}_{(i-1)n+j,1})
\cdots\sum_{p_{\ell-1}=1}^m\frac{a_{m,p_{\ell-1}}}{n}\\&\mathbbm{1}_{A^m_{p_{\ell-1}}}
(y^{\ell,m,n}_{(i-1)n+j,\ell-1})\sum_{q_1=1}^n\cdots\sum_{q_{\ell-1}=1}^n\\
&g_{\ell}(t,\psi,\phi^{m,n}_{(p_{1}-1)n+q_{1}},\ldots,\phi^{m,n}_{(p_{\ell-1}-1)n+q_{\ell-1}})
+h^m(t,x_i^m,\psi)
\end{align*}

The following well-posedness result is akin to Proposition~\ref{prop-discrete-characteristic} and hence the proof is omitted.

\begin{proposition}\label{prop-well-posed-lattice}
Then there exists a unique solution $\phi^{m,n}(t)=(\phi^{m,n}_{(i-1)n+j}(t))$ to \eqref{lattice},
for $m,n\in\N$.
\end{proposition}

Based on Proposition~\ref{prop-well-posed-lattice}, let \begin{equation}\label{Eq-approx}
  \nu_t^{m,n,x}\coloneqq\sum_{i=1}^m\mathbbm{1}_{A^m_i}(x)\frac{a_{m,i}}{n}\sum_{j=1}^n
\delta_{\phi^{m,n}_{(i-1)n+j}(t)},\quad x\in X.
\end{equation}

Now we present the approximation of solutions to the VE \eqref{Vlasov}.
\begin{theorem}\label{th-approx}
Assume ($\mathbf{A1}$)-($\mathbf{A3}$), $\mathbf{(A4)'}$, ($\mathbf{A6}$)-($\mathbf{A7}$). Assume $\rho_0(x,\phi)$ is continuous in $x\in X$ for $\lambda$-a.e. $\phi\in Y$ such that $\rho_0\in {\sf L}^1_+(X\times Y;\mu_X\otimes\lambda)$ and $$\sup_{x\in X}\|\rho_0(x,\cdot)\|_{{\sf L}^1(Y;\lambda)}<\infty.$$  Let  $\rho(t,x,\phi)$ be the uniformly weak solution to the VE \eqref{Vlasov} with initial condition $\rho_0$. Let $\nu_{\cdot}\in\mathcal{C}(\cI,\mathcal{B}_{*}(X,\cM_{+,\abs}(Y)))$ be the measure-valued function defined in terms of the uniform weak solution to \eqref{Vlasov}:
$$\rd\nu_t^x(\phi)=\rho(t,x,\phi)\rd\phi,\quad \text{for every}\quad t\in\cI\quad \text{and}\quad x\in X,\quad \lambda\ \text{a.e.}\ \phi\in Y.$$ Then $\nu_t\in \mathcal{C}_{*}(X,\cM_+(Y))$, for all $t\in\cI$, provided $\nu_0\in\mathcal{C}_{*}(X,\cM_+(Y))$. Moreover, let
$\nu^{m,n}_0\in\mathcal{B}_{*}(X,\cM_+(Y))$, $\eta^{\ell,m,n}\in \mathcal{B}(X,\cM_+(X^{\ell-1}))$, and $h^m\in \mathcal{C}(\cI\times X\times Y,\R^{r_2})$ be defined in \eqref{discretization},
and $\nu^{m,n}_{t}$ be defined in \eqref{Eq-approx}. Then $$\lim_{n\to\infty}d_{0}(\nu_{\cdot}^{m,n},\nu_{\cdot})=0.$$
\end{theorem}

The proof of Theorem~\ref{th-approx} is provided in Section~\ref{sect-proof}.

\section{Applications}

In this section, we apply our main results to investigate the MFL of three networks with higher order interactions emerging from physics, epidemiology, and ecology.

\subsection{A Kuramoto-Sakaguchi model with heterogeneous higher order interactions}\label{sect-applications}

Consider the following Kuramoto-Sakaguchi phase reduction network proposed in \cite{BAR16}:
\begin{equation*}
\begin{split}
  \dot{\phi}_i^N=&h_i^N+\frac{1}{N}\sum_{j=1}^NW^{2,N}_{i,j}g_2(\phi_i^N,\phi_j^N)
  +\frac{1}{N^2}\sum_{j=1}^N\sum_{k=1}^NW^{3,N}_{i,j,k}g_3(\phi_i^N,\phi_j^N,\phi_k^N)\\
  &+\frac{1}{N^3}\sum_{j=1}^N\sum_{k=1}^N\sum_{p=1}^NW^{4,N}_{i,j,k,p}
  g_4(\phi_i^N,\phi_j^N,\phi_k^N,\phi_{p}^N),
\end{split}\end{equation*}
where $\phi^N_i\in\mathbb{T}$ is the phase and $h_i$ the natural frequency of the $i$-th oscillator, and $g_{\ell}$ ($\ell=2,3,4$) are the coupling functions of different higher-order interactions.

Let $(I^N_j)_{1\le j\le N}$ be an equi-partition of $X=[0,1]$ defined as in\eqref{Eq-equipartition}, and let $$h^N\coloneqq\sum_{j=1}^N\mathbbm{1}_{I_j^N}h^N_j$$

Assume
\smallskip

\noindent($\mathbf{H1.1}$) $g_{\ell}$ ($\ell=2,3,4$) are Lipschitz continuous.

\noindent($\mathbf{H1.2}$) $h_j^N$ ($j=1,\ldots,N$) fulfill that there exists $h\in\mathcal{C}(X)$ such that
$$\lim_{N\to\infty}\sup_{x\in X}|h(x)-h^N(x)|=0.$$
\medskip

\noindent($\mathbf{H1.3}$) For $\ell=2,3,4$, $W^{\ell,N}$ converges in the uniform bounded Lipschitz metric to $\eta_{\ell}\in\mathcal{C}(X,\mathcal{M}_+(X^{\ell}))$.

\medskip

Note that ($\mathbf{H1.2}$) implies that $$\lim_{N\to\infty}\|h-h^N\|_1=0.$$

Now we consider the VE
\begin{align}
 \label{Vlasov-model-1} &\frac{\partial\rho(t,x,\phi)}{\partial t}+\textrm{div}_{\phi}\left(\rho(t,x,\phi)\widehat{V}[\eta,\rho,h](t,x,\phi)\right)=0,\ t\in(0,T],\ x\in X,\ \lambda\text{-a.e.}\ \phi\in \T,\\
\nonumber  &\rho(0,\cdot)=\rho_0(\cdot),
  \end{align}
where \begin{align*}
  \widehat{V}[\eta,\rho,h](t,x,\phi)=&h(x)+\int_X\int_{\T}g_2(\phi(t,x),\psi)\rho(t,y_1,\psi)\rd\psi\rd\eta^x_{2}(y_1)\\
  &+\int_{X^2}\int_{\T}\int_{\T}g_3(\phi(t,x),\psi_1,\psi_2)\rho(t,y_1,\psi_1)\rho(t,y_2,\psi_2)
\rd\psi_1\rd\psi_2\rd\eta^x_{3}(y_1,y_2)\\
  &+\int_{X^3}\int_{\T}\int_{\T}\int_{\T}g_4(\phi(t,x),\psi_1,\psi_2,\psi_3)\rho(t,y_1,\psi_1)
  \rho(t,y_2,\psi_2)\rho(t,y_3,\psi_3)\\
  &\rd\psi_1\rd\psi_2\rd\psi_3\rd\eta^{x}_{4}(y_1,y_2,y_3).
\end{align*}

\begin{theorem}
Assume $\mathbf{(H1.1)}$-$\mathbf{(H1.3)}$. Let $T>0$. Assume $\rho_0(x,\phi)$ is continuous in $x\in X$ for $\lambda$-a.e. $\phi\in Y$ such that $\rho_0\in {\sf L}^1_+(X\times Y;\mu_X\otimes\lambda)$ and $$\sup_{x\in X}\|\rho_0(x,\cdot)\|_{{\sf L}^1(Y;\lambda)}<\infty.$$ Define $\nu_0\in\cB_*(X,\cM_+(Y))$ by $$\rho_0(x,\phi)=\frac{\rd\nu_0(x,\phi)}{\rd\phi}\quad \text{for}\quad x\in X\quad \text{and}\quad \lambda\text{-a.e.}\quad \phi\in\sT.$$ Then there exists a unique uniformly weak solution $\rho(t,\cdot)$ to the VE \eqref{Vlasov-model-1}. Moreover, if $\nu_0\in\cC_*(X,\cM_+(Y))$ and  $\lim_{N\to\infty}d_{\infty}(\nu_{N,0},\nu_0)=0$, then
\begin{equation*}
  \lim_{N\to\infty}d_0(\nu_{N,\cdot},\nu_{\cdot})=0.
\end{equation*}
In particular,
\begin{equation*}
\lim_{N\to\infty}d_{\sf BL}\Bigl(\frac{1}{N}\sum_{i=1}^N\delta_{\phi^N_i(t)},
\int_X\nu_t^x(\sbullet)\rd\mu_X(x)\Bigr)=0,\quad \text{for}\quad t\in\cI.
\end{equation*}
\end{theorem}
\begin{proof}
  The proof is similar to those of Theorem~\ref{thm-well-posedness} and Theorem~\ref{th-approx}. Here we still have the invariance under the flow of fiberized equation of characteristics of a compact set which is $\mathbb{T}$ (in lieu of a compact set $Y$ in a Euclidean space).
\end{proof}
\subsection{An epidemic model with higher-order interactions}

Assume

\medskip
\noindent ($\mathbf{H2.1}$)  $(X,\mathfrak{B}(X),\mu_X)$ is a compact probability space.

\medskip
\noindent ($\mathbf{H2.2}$) For $(u_1,u_2)\in\R_+^2$, let $\beta(t,u_1,u_2)\ge0$ be the disease transmission function satisfying $\beta(t,u_1,u_2)=0$ provided $u_1u_2=0$. Moreover, $\beta$ is continuous in $t$, and locally Lipschitz continuous in $u_1,\ u_2$  uniformly in $t$.

\medskip

\noindent ($\mathbf{H2.3}$)  For $u\in\R_+$, let $\gamma(t,x,u)\ge0$ be the recovery rate function, and for every $x\in X$, $\gamma(t,x,0)=0$. Moreover, $\gamma$ is continuous in $t$,  and continuous in $x$ uniformly in $u$, and Lipschitz continuous in $u\in \R_+$ uniformly in $t$.

For any fixed $N\in\N$, let \begin{equation}
  \label{Eq-Y-SIS}Y=\{u\in\R^2_+\colon u_1+u_2=N\}.
\end{equation}

\medskip
\noindent{$\mathbf{(H2.4)}$}   $\eta\in \mathcal{C}(X,\mathcal{M}_{+}(X^2))$.
\medskip

\noindent{$\mathbf{(H2.5)}$} $\nu_{\cdot}\in \mathcal{C}(\cI,\mathcal{B}_{*}(X,\cM_+(\R^2)))$  is uniformly compactly supported within $Y\subseteq\R_+^{2}$.

\medskip

Under ($\mathbf{H2.1}$)-($\mathbf{H2.5}$), motivated by \cite{BKS16}, we propose the following generalized non-local multi-group SIS epidemic model on a DHGM $\eta$ incorporating the higher-order interactions due to the nonlinear dependence of both the infection pressure and the community structure (home and workplace):

\begin{align*}
  \frac{\partial S_x}{\partial t} = & -\int_{X^2}\int_{\R_+^{2}}\int_{\R_+^{2}}\left(\beta(t,\psi_{1,2},S_x)+\beta(t,\psi_{2,2},S_x)
  \right)\rd\nu_t^{y_2}(\psi_2)\rd\nu_t^{y_1}(\psi_1)\rd\eta^x(y_1,y_2)\\&+\gamma(t,x,I_x),\\
  \frac{\partial I_x}{\partial t} = & \int_{X^2}\int_{\R_+^{2}}\int_{\R_+^{2}}\left(\beta(t,\psi_{1,2},S_x)+\beta(t,\psi_{2,2},S_x)
  \right)\rd\nu_t^{y_2}(\psi_2)\rd\nu_t^{y_1}(\psi_1)\rd\eta^x(y_1,y_2)\\&-\gamma(t,x,I_x),
\end{align*}
which further generalizes the epidemic network on a digraph measure proposed in \cite{KX21}. Here $S_x$ and $I_x$ stand for the number of susceptible and infected individuals at location $x\in X$ (or interpreted as in the group with label $x$), $\beta(t,\psi_2^1,S_x)$ stands for the infection caused by  family members of $S_x$ at home while $\beta(t,\psi_2^2,S_x)/S_x$ the infection rate caused by  colleagues of $S_x$ in the workplace, and $\eta\colon X\to\cM_+(X^2)$ is the generalized hypergraph.

By $\mathbf{(H2.3)}$, let
$$g(t,\phi,\psi^1,\psi^2)=\beta(t,\psi_{1,2},\phi_1)+\beta(t,\psi_{2,2},\phi_1)\begin{pmatrix}
  -1\\
  1
\end{pmatrix},\quad h(t,x,\phi)=\gamma(t,x,\phi_1)\begin{pmatrix}
1\\
-1
\end{pmatrix},$$ $$V[\eta,\nu_{\cdot},h](t,x,\phi)=\int_{X^2}\int_{Y}\int_{Y}g(t,\phi,\psi_1,\psi_2)
\rd\nu_t^{y_2}(\psi_2)\rd\nu_t^{y_1}(\psi_1)\rd\eta^x(y_1,y_2)+h(t,x,\phi),$$and

\begin{equation*}
\widehat{V}[\eta,\rho_{\cdot},h](t,x,\phi)=\int_{X^2}\int_{Y}\int_{Y}g(t,\phi,\psi_1,\psi_2)
\rho(t,y_1,\psi_1)\rho(t,y_2,\psi_2)\rd\psi_2\rd\psi_1\rd\eta^x(y_1,y_2)+h(t,x,\phi).
\end{equation*}
Consider the VE
\begin{equation}\label{Vlasov-SIS}
\begin{split}
&\frac{\partial\rho(t,x,\phi)}{\partial t}+\textrm{div}_{\phi}\left(\rho(t,x,\phi)\widehat{V}[\eta,\rho(\cdot),h](t,x,\phi)\right)=0,\quad t\in(0,T],\ x\in X,\ \lambda\text{-a.e.}\ \phi\in Y,\\
 &\rho(0,\cdot)=\rho_0(\cdot).
\end{split}
  \end{equation}

According to Propositions~\ref{le-partition}-\ref{le-h}, there exists
 \begin{enumerate}
 \item[$\bullet$] a partition $\{A^m_i\}_{1\le i\le m}$ of $X$ and points $x^m_i\in A^m_i$ for $i=1,\ldots,m$, for every $m\in\N$,
 \item[$\bullet$]  a sequence $\{\varphi^{m,n}_{(i-1)n+j}\coloneqq i=1,\ldots,m,j=1,\ldots,n\}_{n,m\in\N}\subseteq Y$ and $\{a_{m,i}\colon i=1,\ldots,$ $m\}_{m\in\N}\subseteq\R_+$, and
  \item[$\bullet$] a sequence $\{y^{m,n}_{(i-1)n+j}\colon i=1,\ldots,m,j=1,\ldots,n\}_{m,n\in\N}\subseteq X^2$ and $\{b_{m,i}\colon i=1,\ldots,$ $m\}_{m\in\N}\subseteq\R_+$,
 \end{enumerate}
 such that
  $$\lim_{m\to\infty}\lim_{n\to\infty}d_{\infty}(\nu_0^{m,n},\nu_0)=0,$$
   $$\lim_{m\to\infty}\lim_{n\to\infty}d_{\infty}(\eta^{m,n},\eta)=0,$$
     $$\lim_{m\to\infty}\int_0^T\int_Y\sup_{x\in X}\lt|h^m(t,x,\phi)-h(t,x,\phi)\rt|\rd\phi\rd t=0,$$
 where
 \begin{equation}
   \label{discretization-SIS}
 \begin{split}
\nu_0^{m,n,x}\coloneqq\sum_{i=1}^m\mathbbm{1}_{A^m_i}(x)\frac{a_{m,i}}{n}\sum_{j=1}^n
\delta_{\varphi^{m,n}_{(i-1)n+j}},\quad x\in X,\\ \eta^{m,n,x}\coloneqq\sum_{i=1}^m\mathbbm{1}_{A^m_i}(x)\frac{b_{m,i}}{n}
\sum_{j=1}^n\delta_{y^{m,n}_{(i-1)n+j}},\quad x\in X,\\
 h^m(t,z,\phi)\coloneqq\sum_{i=1}^m\mathbbm{1}_{A^m_i}(z)h(t,x_i^m,\phi),\quad t\in\cI,\ z\in X,\ \phi\in Y.
 \end{split}
 \end{equation}
 Consider the following IVP of a coupled ODE system:
\begin{multline}
  \label{lattice-SIS}
  \dot{\phi}_{(i-1)n+j}=F^{m,n}_{i}(t,\phi_{(i-1)n+j},\Phi),\quad 0<t\le T,\quad \phi_{(i-1)n+j}(0)=\varphi^{m,n}_{(i-1)n+j},\\ i=1,\ldots,m,\ j=1,\ldots,n,
\end{multline}
where $\Phi=(\phi_{(i-1)n+j})_{1\le i\le m,1\le j\le n}$ and \begin{align*}
  F^{m,n}_{i}(t,\psi,\Phi)=&\frac{b_{m,i}}{n}\sum_{j=1}^n\sum_{p_1=1}^m\frac{a_{m,p_1}}{n}
\mathbbm{1}_{A^m_{p_1}}(y^{m,n}_{(i-1)n+j,1})\sum_{p_2=1}^m\frac{a_{m,p_2}}{n}
\mathbbm{1}_{A_{p_2}^m}(y^{m,n}_{(i-1)n+j,2})\\
&\sum_{q_1=1}^{n}\sum_{q_2=1}^ng(t,\psi,\phi_{(p_1-1)n+q_1},\phi_{(p_2-1)n+q_2})+h^m(t,x^m_i,\psi).
\end{align*}

Then by Proposition~\ref{prop-well-posed-lattice}, there exists a unique solution $\phi^{m,n}(t)=(\phi^{m,n}_{(i-1)n+j}(t))_{1\le i\le m, 1\le j\le n}$ to \eqref{lattice-SIS},
for $m,n\in\N$.

For $t\in\cI$, define \begin{equation}\label{Eq-approx-SIS}
\nu_t^{m,n,x}\coloneqq\sum_{i=1}^m\mathbbm{1}_{A^m_i}(x)\frac{a_{m,i}}{n}\sum_{j=1}^{n}\delta_{\phi^{m,n}_{(i-1)n+j}(t)},\quad x\in X.
\end{equation}
\begin{theorem}
  Assume ($\mathbf{H3.1}$)-($\mathbf{H3.4}$). Then there exists a unique uniformly weak solution $\rho(t,x,\phi)$ to \eqref{Vlasov-SIS}. Assume additionally $\rho_0(x,\phi)$ is continuous in $x\in X$ for $\lambda$-a.e. $\phi\in Y$ such that $\rho_0\in {\sf L}^1_+(X\times Y;\mu_X\otimes\lambda)$ and $$\sup_{x\in X}\|\rho_0(x,\cdot)\|_{{\sf L}^1(Y;\lambda)}<\infty.$$ Let $\nu_{\cdot}\in\mathcal{C}(\cI,\mathcal{B}_{*}(X,\cM_{+,\abs}(Y)))$ be the measure-valued function defined in terms of the uniformly weak solution to \eqref{Vlasov-SIS}:
$$\rd\nu_t^x=\rho(t,x,\phi)\rd\phi,\quad \text{for every}\quad t\in\cI,\quad \text{and}\quad x\in X.$$  Then $\nu_t\in \mathcal{C}_{*}(X,\cM_+(Y))$, for all $t\in\cI$, provided $\nu_0\in\mathcal{C}_{*}(X,\cM_+(Y))$. Moreover, let
$\nu^{m,n}_0\in\mathcal{B}_{*}(X,\cM_+(Y))$, $\eta^{m,n}\in \mathcal{B}(X,$ $\cM_+(Y))$, and $h^m\in \mathcal{C}(\cI\times X\times Y,\R^2)$ be defined in \eqref{discretization-SIS},
and $\nu^{m,n}_{t}$ be defined in \eqref{Eq-approx-SIS}. Then $$\lim_{n\to\infty}d_{0}(\nu_{\cdot}^{m,n},\nu_{\cdot})=0.$$
\end{theorem}
\begin{proof}
It is straightforward to verify that ($\mathbf{H3.1}$) implies ($\mathbf{A1}$),  ($\mathbf{H3.2}$) implies ($\mathbf{A2}$),  and ($\mathbf{H3.3}$) implies ($\mathbf{A3}$) and ($\mathbf{A7}$).
It remains to show ($\mathbf{A6}$) is fulfilled with $Y$ defined in \eqref{Eq-Y-SIS}. This is a simple consequence of the fact that  this SIS model is conservative:
$$\frac{\partial }{\partial t}(S_x(t)+I_x(t))=0.$$
\end{proof}
\subsection{Lotka-Volterra model with dispersal on a hypergraph}

Assume that
\medskip

\noindent ($\mathbf{H3.1}$)  $(X,\mathfrak{B}(X),\mu_X)$ be a compact probability space.

\medskip
\noindent ($\mathbf{H3.2}$)  $W_{i,j}$ are odd functions and locally Lipschitz continuous satisfying $0\le W_{i,j}(u)$ $\le u$ for all $u\in \R^+$ for all $i,j=1,2$.

\medskip

\noindent ($\mathbf{H3.3}$) $\eta_1,\ \eta_2\in \mathcal{B}(X,\cM_+(X^2))$.

\medskip

Let $\Lambda_1,\ \Lambda_2>0$ satisfy
\begin{equation}
  \label{Lambdacondition}
\Lambda_1\ge\frac{\alpha}{\beta},\quad \Lambda_2\ge-\frac{\iota}{\theta}+\frac{\sigma}{\theta}\Lambda_1.
\end{equation}
Let $Y=\{\phi\in\R^2_+\colon \phi_1\le\Lambda_1,\ \phi_2\le\Lambda_2\}$ be the rectangle in the positive cone, which is a convex compact set.

\medskip

\noindent{$\mathbf{(H3.4)}$} $\nu_{\cdot}\in \mathcal{C}(\cI,\mathcal{C}_{*}(X,\cM_+(\R^2)))$  is uniformly compactly supported within $Y\subseteq\R_+^2$.

\medskip

Under ($\mathbf{H3.1}$)-($\mathbf{H3.4}$), consider the general Lotka-Volterra with the species of two types moving on $3$-uniform DHGMs $\eta_1$ and $\eta_2$:
\begin{equation*}
\begin{split}
  \frac{\partial \phi_1(t,x)}{\partial t}=&\phi_1(t,x)(\alpha-\beta\phi_1(t,x)-\gamma\phi_2(t,x))+\int_X\int_X\int_Y\int_Y
  \\
  &\left(W_{1,1}(\psi_1-\phi_1(t,x))+W_{1,2}(\psi_2-\phi_1(t,x))\right)\rd\nu_t^{y_1}(\psi_1)
  \rd\nu_t^{y_2}(\psi_2)\rd\eta^x_1(y_1,y_2)\\
  \frac{\partial \phi_2(t,x)}{\partial t}=&\phi_2(t,x)(-\iota+\sigma\phi_1(t,x)-\theta\phi_2(t,x))
  +\int_X\int_X\int_Y\int_Y\\
  &\left(W_{2,1}(\psi_1-\phi_2(t,x))+W_{2,2}(\psi_2-\phi_2(t,x))\right)\rd\nu_t^{y_1}(\psi_1)
  \rd\nu_t^{y_2}(\psi_2)\rd\eta^x_2(y_1,y_2),
  \end{split}
\end{equation*}
where $\phi_1(t)$ and $\phi_2(t)$ stand for population densities of two competing species at time $t$, respectively, and all given functions and parameters are non-negative. The model can be regarded as a generalization of the Lotka-Volterra model on the graph proposed in \cite{S20}.

Let \[g_2(t,\phi,\psi_1,\psi_2)=\begin{pmatrix}
  W_{1,1}(\psi_1-\phi_1)+W_{1,2}(\psi_2-\phi_1)\\
  0
\end{pmatrix},\]
\[ g_3(t,\phi,\psi_1,\psi_2)=\begin{pmatrix}
  0\\
  W_{2,1}(\psi_1-\phi_2)+W_{2,2}(\psi_2-\phi_2)
\end{pmatrix},\] \[h(\phi)=\begin{pmatrix}
  \phi_1(\alpha-\beta\phi_1-\gamma\phi_2)\\
  \phi_2(-\iota+\sigma\phi_1-\theta\phi_2)
\end{pmatrix},\] and \begin{align*}
\widehat{V}[\eta,\rho_{\cdot},h](t,x,\phi)=&\sum_{\ell=1}^2
\int_X\int_X\int_Y\int_Yg_{\ell}(t,\phi,\psi_1,\psi_2)\rho(t,y_1,\psi_1)\rho(t,y_2,\psi_2)\rd\psi_1\rd\psi_2
\rd\eta_{\ell}^x(y_1,y_2)\\&+h(\phi).
\end{align*}
\begin{remark}
{Despite that there are two DHGMs of the same rank, the main results obtained in previous sections do hold for $J$ with multiplicity. In other words, the main results hold \emph{mutatis mutandis} if uniform HDGMs in the vector field of the VE have the same rank while are associated with different coupling functions $g_{\ell}$.}
\end{remark}

Consider the VE
\begin{equation}\label{Vlasov-LS}
\begin{split}
&\frac{\partial\rho(t,x,\phi)}{\partial t}+\textrm{div}_{\phi}\left(\rho(t,x,\phi)\widehat{V}[\eta,\rho(\cdot),h](\phi)\right)=0,\quad t\in(0,T],\ x\in X,\ \lambda\text{-a.e.}\ \phi\in Y,\\
 &\rho(0,\cdot)=\rho_0(\cdot).
\end{split}
  \end{equation}
Note that there exists
 \begin{enumerate}
 \item[$\bullet$] a partition $\{A^m_i\}_{1\le i\le m}$ of $X$ and points $x^m_i\in A^m_i$ for $i=1,\ldots,m$, for every $m\in\N$,
 \item[$\bullet$]  a sequence $\{\varphi^{m,n}_{(i-1)n+j}\colon i=1,\ldots,m,j=1,\ldots,n\}_{n,m\in\N}\subseteq Y$ and $\{a_{m,i}\colon i=1,\ldots,$ $m\}_{m\in\N}\subseteq\R_+$, and
  \item[$\bullet$] a sequence $\{y^{\ell,m,n}_{(i-1)n+j}\colon i=1,\ldots,m,j=1,\ldots,n\}_{m,n\in\N}\subseteq X^2$ and $\{b_{\ell,m,i}\colon i=1,\ldots,$ $m\}_{m\in\N}\subseteq\R_+$, for $\ell=1,2$,
 \end{enumerate}
 such that
  $$\lim_{m\to\infty}\lim_{n\to\infty}d_{\infty}(\nu_0^{m,n},\nu_0)=0,$$
   $$\lim_{m\to\infty}\lim_{n\to\infty}d_{\infty}(\eta^{m,n}_{\ell},\eta_{\ell})=0,\quad \ell=1,2,$$
     $$\lim_{m\to\infty}\int_0^T\int_Y\sup_{x\in X}\lt|h^m(t,x,\phi)-h(t,x,\phi)\rt|\rd\phi\rd t=0,$$
 where
 \begin{equation}
   \label{discretization-LS}
 \begin{split}
\nu_0^{m,n,x}\coloneqq\sum_{i=1}^m\mathbbm{1}_{A^m_i}(x)\frac{a_{m,i}}{n}\sum_{j=1}^n
\delta_{\varphi^{m,n}_{(i-1)n+j}},\quad x\in X,\\ \eta^{m,n,x}_{\ell}\coloneqq\sum_{i=1}^m\mathbbm{1}_{A^m_i}(x)\frac{b_{\ell,m,i}}{n}
\sum_{j=1}^n\delta_{y^{\ell,m,n}_{(i-1)n+j}},\quad x\in X,\quad \ell=1,2,\\
 h^m(t,z,\phi)\coloneqq\sum_{i=1}^m\mathbbm{1}_{A^m_i}(z)h(t,x_i^m,\phi),\quad t\in\cI,\ z\in X,\ \phi\in Y.
 \end{split}
\end{equation}
 Consider the following IVP of a coupled ODE system:
\begin{multline}
  \label{lattice-LS}
  \dot{\phi}_{(i-1)n+j}=F^{m,n}_{i}(t,\phi_{(i-1)n+j},\Phi),\quad 0<t\le T,\quad \phi_{(i-1)n+j}(0)=\varphi^{m,n}_{(i-1)n+j},\\ i=1,\ldots,m,\ j=1,\ldots,n,
\end{multline}
where $\Phi=(\phi_{(i-1)n+j})_{1\le i\le m,1\le j\le n}$ and $\Phi=(\phi_{(i-1)n+j})_{1\le i\le m,1\le j\le n}$ and \begin{align*}
  F^{m,n}_{i}(t,\psi,\Phi)=&\sum_{\ell=1}^2\frac{b_{\ell,m,i}}{n}\sum_{j=1}^n\sum_{p_1=1}^m\frac{a_{m,p_1}}{n}
\mathbbm{1}_{A^m_{p_1}}(y^{\ell,m,n}_{(i-1)n+j,1})\sum_{p_2=1}^m\frac{a_{m,p_2}}{n}
\mathbbm{1}_{A_{p_2}^m}(y^{\ell,m,n}_{(i-1)n+j,2})\\
&\sum_{q_1=1}^{n}\sum_{q_2=1}^ng_{\ell}(t,\psi,\phi_{(p_1-1)n+q_1},\phi_{(p_2-1)n+q_2})+h^m(t,x^m_i,\psi).
\end{align*}
Then by Proposition~\ref{prop-well-posed-lattice}, there exists a unique solution $\phi^{m,n}(t)=(\phi^{m,n}_{(i-1)n+j}(t))_{1\le i\le n,1\le j\le m}$ to \eqref{lattice-LS},
for $m,n\in\N$.

For $t\in\cI$, define \begin{equation}\label{Eq-approx-LS}
\nu_t^{m,n,x}\coloneqq\sum_{i=1}^m\mathbbm{1}_{A^m_i}(x)\frac{a_{m,i}}{n}
\sum_{j=1}^{n}\delta_{\phi^{m,n}_{(i-1)n+j}(t)},\quad x\in X.
\end{equation}

\begin{theorem}
  Assume ($\mathbf{H3.1}$)-($\mathbf{H3.4}$). Additionally assume $\Lambda_1, \Lambda_2$ satisfy \eqref{Lambdacondition}. Then there exists a unique uniformly weak solution $\rho(t,x,\phi)$ to \eqref{Vlasov-LS}. Assume additionally $\rho_0(x,\phi)$ is continuous in $x\in X$ for $\lambda$-a.e. $\phi\in Y$ such that $\rho_0\in {\sf L}^1_+(X\times Y;\mu_X\otimes\lambda)$ and $$\sup_{x\in X}\|\rho_0(x,\cdot)\|_{{\sf L}^1(Y;\lambda)}<\infty.$$ Let $\nu_{\cdot}\in\mathcal{C}(\cI,\mathcal{B}_{*}(X,\cM_{+,\abs}(Y)))$ be the measure-valued function defined in terms of the uniformly weak solution to \eqref{Vlasov-LS}:
$$\rd\nu_t^x=\rho(t,x,\phi)\rd\phi,\quad \text{for every}\quad t\in\cI,\quad \text{and}\quad x\in X.$$ Then $\nu_t\in \mathcal{C}_{*}(X,\cM_+(Y))$, for all $t\in\cI$, provided $\nu_0\in\mathcal{C}_{*}(X,\cM_+(Y))$. Moreover, let
$\nu^{m,n}_0$, $\eta^{\ell,m,n}$ for $\ell=1,2$, and $h^m$ be defined in \eqref{discretization-LS},
and $\nu^{m,n}_{t}$ be defined in \eqref{Eq-approx-LS}. Then $$\lim_{n\to\infty}d_{0}(\nu_{\cdot}^{m,n},\nu_{\cdot})=0.$$
\end{theorem}

\begin{proof}
 First note that ($\mathbf{H3.1}$) implies ($\mathbf{A1}$) and ($\mathbf{H3.2}$) implies ($\mathbf{A2}$). It is readily verified that  ($\mathbf{A3}$) and  ($\mathbf{A7}$) are fulfilled since $Y$ is compact. In addition, $\mathbf{(A4)'}$ follows from ($\mathbf{H3.3}$). Hence it suffices to show that ($\mathbf{A6}$) holds with $Y$ for some $c,\Lambda>0$.

  Note that $\partial Y=\{\phi_1=0\}\cup\{\phi_2=0\}\cup\{\phi_1=\Lambda_1\}\cup\{\phi_2=\Lambda_2\}$. In the following, we will show that \[V[\eta,\nu_{\cdot},h](t,x,\phi)\cdot\upsilon(\phi)\le0,\quad \text{for all}\ t\in\mathcal{T},\ x\in X,\quad \phi\in\partial Y,\]
  where $\upsilon(\phi)$ is the outer normal vector at $\phi$. We prove it case by case.

\begin{enumerate}
\item[(i)] For $\phi\in\{\varphi\in Y\colon \varphi_1=0\}$, $\upsilon(\phi)=(-1,0)$, and
\begin{align*}
  &V[\eta,\nu_{\cdot},h](t,x,\phi)\cdot\upsilon(\phi)\\
  =&-\int_X\int_X\int_Y\int_Y\left(W_{1,1}(\psi_1-\phi_1)+W_{1,2}(\psi_2-\phi_1)\right)\rd\nu_t^{y_1}(\phi_1)
  \rd\nu_t^{y_2}(\phi_2)\rd\eta_1^x(y_1,y_2)\le0.
\end{align*}
\item[(ii)] For $\phi\in\{\varphi\colon \varphi_2=0\}$, $\upsilon(\phi)=(0,-1)$, and
\begin{align*}
  &V[\eta,\nu_{\cdot},h](t,x,\phi)\cdot\upsilon(\phi)\\
  =&-\int_X\int_X\int_Y\int_Y\left(W_{2,1}(\psi_1-\phi_2)+W_{2,2}(\psi_2-\phi_2)\right)\rd\nu_t^{y_1}(\phi_1)
  \rd\nu_t^{y_2}(\phi_2)\rd\eta^x_2(y_1,y_2)\le0.
\end{align*}
\item[(iii)] For $\phi\in\{\varphi\colon \varphi_1=\Lambda_1\}$, $\upsilon(\phi)=(1,0)$. By \eqref{Lambdacondition},
\begin{align*}
  &V[\eta,\nu_{\cdot},h](t,x,\phi)\cdot\upsilon(\phi)\\
  =&\Lambda_1(\alpha-\beta\Lambda_1-\gamma\phi_2)\\
  &+\int_X\int_X\int_Y\int_Y\left(W_{1,1}(\psi_1-\Lambda_1)+W_{1,2}(\psi_2-\Lambda_1)\right)\rd\nu_t^{y_1}
  (\phi_1)\rd\nu_t^{y_2}(\phi_2)\rd\eta^x_1(y_1,y_2)\\
  \le&\Lambda_1(\alpha-\beta\Lambda_1)\le0,
\end{align*}
since $\Lambda_1\ge\frac{\alpha}{\beta}$, due to \eqref{Lambdacondition}.
\item[(iv)] For $\phi\in\{\varphi\colon \varphi_2=\Lambda_2\}$, $\upsilon(\phi)=(1,0)$. By \eqref{Lambdacondition}, and $\psi_2\le\Lambda_2$ for $\psi\in Y$, we have
\begin{align*}
  &V[\eta,\nu_{\cdot},h](t,x,\phi)\cdot\upsilon(\phi)\\
  =&\Lambda_2(-\iota+\sigma\phi_1-\theta\Lambda_2)\\
  &+\int_X\int_X\int_Y\int_Y\left(W_{2,1}(\psi_1-\Lambda_2)+W_{2,2}
  (\psi_2-\Lambda_2)\right)\rd\nu_t^{y_1}(\psi_1)\nu_t^{y_2}(\psi_2)
  \rd\eta^x_2(y_1,y_2)\\
  \le&\Lambda_2(-\iota+\sigma\Lambda_1-\theta\Lambda_2)\le0,
\end{align*}
since by \eqref{Lambdacondition}, we have $\Lambda_2\ge-\frac{\iota}{\theta}+\frac{\sigma}{\theta}\Lambda_1$.
\end{enumerate}
\end{proof}

\section{Discussion}
In this paper we regard directed hypergraph limits as {linear combinations of elements in $\cB(X,\cM_+(X^{\ell-1}))$ for  $\ell\in J\subseteq\N\setminus\{1\}$.} The motivation of doing such a work comes from the emergent demanding applications from networks models of higher-order interactions \cite{BGHS22}. We extend the idea proposed in \cite{KX21}  from \emph{directed graph measures} (DGM) to \emph{directed hyper-graph measures} (HDGM). We apply our main results to {investigate mean field limits of}  Kuramoto networks of higher-order interactions in physics as well as models in epidemiology and ecology. {Nevertheless, our results cannot directly apply to models of higher-order interactions which do \emph{not} admit a \emph{compact} positively invariant set and hence the trajectories are not ultimately bounded \cite{B20}. Hence it would be desirable to consider extension of our results to those models. In addition, our results thus far only apply to the case where the hypergraph measure is a limit of the sequence of hypergraphs of \emph{uniformly bounded} ranks. It will be challenging while exciting to study the case e.g., where the hypergraph measure is a limit of the sequence of hypergraphs of \emph{unbounded} ranks, e.g., a sequence of simpicial complexes with growing ranks.} 

{From graph theoretic perspective, there is room to improve our definition of hypergraph limits to fit in well with the cut metric naturally measuring the distance between hypergraphs, as well as identifying the difference set of DHMGs and the hypergraph limit as a sequence of hypergraphs convergent in cut metric \cite{ES12,L12}.} 

\section{Proof of Theorem~\ref{th-approx}}\label{sect-proof}
\begin{proof}
First, the proof of $\nu_{\cdot}\in \mathcal{C}(\cI,\mathcal{C}_{*}(X,\cM_+(Y)))$ follows directly from that of \cite[Theorem~5.15]{KX21}. The strategy based on four steps to prove the approximation result is identical to that given in the proof of \cite[Theorem~5.15]{KX21}. Steps II-IV among the four steps are independent of $\widehat{V}$. We essentially need to verify that $\nu^{m,n}_{\cdot}$ is the unique solution to  the fixed point equation associated with $\eta^{m,n}$ and $h^m$: \begin{equation*}\nu^{m,n}_{\cdot}
=\cA[\eta^{m,n},h^m]\nu^{m,n}_{\cdot}.\end{equation*} In the light of Step I in the proof of \cite[Theorem~5.15]{KX21}, it suffices to express the Vlasov operator in the discrete context and show it is consistent with the definition of $F_i^{m,n}$.

Denote $\eta^{m,n}=(\eta^{m,n}_{\ell})_{\ell\in J}$. For $x\in I^N_i$, $t\in\cI$, by \eqref{discretization} and \eqref{Eq-approx}, we have
\begin{align*}
  &V^{m,n}[\eta^{m,n},\nu^{m,n}_{\cdot},h^m](t,x,\phi)\\
=&\sum_{\ell\in J}\int_{X^{\ell-1}}\underbrace{\int_{\R^{r_2}}\ldots\int_{\R^{r_2}}}_{\ell-1}
g_{\ell}(t,\phi,\psi_1,\ldots,\psi_{\ell-1})
\rd\nu_t^{m,n,y_{\ell-1}}(\psi_{\ell-1})\\
&\cdots\rd\nu_t^{m,n,y_{1}}(\psi_{1})\cdot\rd\eta_{\ell}^{m,n,x}(y_1,\ldots,y_{\ell-1})+h^m(t,x,\phi),\\
=&\sum_{\ell\in J}\frac{b_{\ell,m,i}}{n}\sum_{j=1}^n\underbrace{\int_{\R^{r_2}}\ldots\int_{\R^{r_2}}}_{\ell-1}
g_{\ell}(t,\phi,\psi_1,\ldots,\psi_{\ell-1})
\\
&\cdot\rd\nu_t^{m,n,y^{\ell,m,n}_{(i-1)n+j,\ell-1}}(\psi_{\ell-1})
\cdots\rd\nu_t^{m,n,y^{\ell,m,n}_{(i-1)n+j,1}}(\psi_{1})+h^m(t,x_i^m,\phi)\\
=&\sum_{\ell\in J}\frac{b_{\ell,m,i}}{n}\sum_{j=1}^n\sum_{p_1=1}^m\frac{a_{m,p_1}}{n}\mathbbm{1}_{A^m_{p_1}}
(y^{\ell,m,n}_{(i-1)n+j,1})
\cdots\sum_{p_{\ell-1}=1}^m\frac{a_{m,p_{\ell-1}}}{n}\\&\mathbbm{1}_{A^m_{p_{\ell-1}}}
(y^{\ell,m,n}_{(i-1)n+j,\ell-1})\sum_{q_1=1}^n\cdots\sum_{q_{\ell-1}=1}^n\\
&g_{\ell}(t,\phi,\phi^{m,n}_{(p_{1}-1)n+q_{1}},\ldots,\phi^{m,n}_{(p_{\ell-1}-1)n+q_{\ell-1}})
+h^m(t,x_i^m,\phi)\\
=&F^{m,n}_i(t,\phi,\Phi^{m,n}(t)).
\end{align*}
\end{proof}

\section*{Acknowledgements}
Both authors thank (1) Giulio Zucal for bringing \cite{Bom07} to their attention; (2) Christian Bick, Tobias B\"{o}hle, Raffaella Mulas, Davide Sclosa, and Giulio Zucal (i.e., the other members of the ``H-Team'') for helpful discussions on hypergraph limits and networks of higher-order interactions; (3) an anonymous referee for pointing out that there should be a $k!$ in the expression of the degree of a directed hypergraph owing to the distinction from an undirected hypergraph. CK acknowledges support via a Lichtenberg Professorship of the VolkswagenStiftung. CX acknowledges the support from TUM University Foundation, the Alexander von Humboldt Foundation, and the Simons Foundation. 

\subsection*{Data availibility statement} Data sharing not applicable to this article as no datasets were generated or analysed
during the current study.

\section*{Declarations}
\subsection*{Conflict of interest} Not applicable.

\bibliographystyle{plain}
\bibliography{references}
\appendix
\section{Proof of Proposition~\ref{prop-continuousdependence}}\label{appendix-prop-continuousdependence}
\begin{proof}
We will suppress the variables in $V[\eta,\nu_{\cdot},h](t,x,\psi)$ and $\Phi^x_{s,t}[\eta,\nu_{\cdot},h]$ whenever they are clear and not the emphasis from the context. The proof of the absolute continuity property is analogous to that of \cite[Proposition~4.4(iv)]{KX21}.

In the following, we prove the rest three continuous dependence properties item by item.

The properties of $\cA$ follows from that of $\Phi^x_{0,t}[\eta,\nu_{\cdot},h]$. Hence in the following, we will first establish corresponding continuity and Lipschitz continuity for $\Phi^x_{0,t}[\eta,\nu_{\cdot},h]$ and then apply the results to derive respective properties for $\cA$.

\begin{enumerate}
  \item[(i)]
  Continuity in $t$. Indeed,
\begin{align*}
  &d_{\infty}(\cA[\eta,h]\nu_{t},\cA[\eta,h]\nu_{s})\\
  =&\sup_{x\in X}d_{\sf BL}(\Phi^x_{t,0}[\eta,\nu_{\cdot},h]\#\nu_0^x,\Phi^x_{s,0}[\eta,\nu_{\cdot},h]\#\nu_0^x)\\
  =&\sup_{x\in X}\sup_{f\in\mathcal{BL}_1(Y)}\lt|\int_Y\lt(f\circ \Phi^x_{t,0}[\eta,\nu_{\cdot},h]\phi-f\circ \Phi^x_{s,0}[\eta,\nu_{\cdot},h]\phi\rt)\rd\nu_0^x(\phi)\rt|\\
  \le&\sup_{x\in X}\int_Y\lt|\Phi^x_{t,0}[\eta,\nu_{\cdot},h]\phi-\Phi^x_{s,0}[\eta,\nu_{\cdot},h]\phi\rt|
  \rd\nu_0^x(\phi)\\
  =&\sup_{x\in X}\int_Y
  \Bigl|\int_s^t\Bigl(h(\tau,x,\Phi^x_{\tau,0}[\eta,\nu_{\cdot},h]\phi)\\
  &+\sum_{\ell\in J}\int_{X^{\ell-1}}\underbrace{\int_{Y}\cdots\int_{Y}}_{\ell-1}g_{\ell}(\tau,
  \Phi^x_{\tau,0}[\eta,\nu_{\cdot},h]\phi,\psi_1,\psi_2,\ldots,\psi_{\ell-1})
    \rd\nu^{y_{\ell-1}}_{\tau}(\psi_{\ell-1})\cdots\rd\nu^{y_1}_t(\psi_1)\\
    &\rd\eta^{x}_{\ell}(y_1,\ldots,y_{\ell-1})\Bigr)\rd\tau\Bigr|\rd\nu_0^x(\phi)\\
  \le&\sup_{x\in X}\int_Y\int_s^t\Bigl(\sum_{\ell\in J}\int_{X^{\ell-1}}\int_{Y}\cdots\int_Y\|g_{\ell}\|_{\infty}
  \rd\nu_{\tau}^{y_{\ell-1}}
  \cdots\rd\nu_{\tau}^{y_1}\rd\eta^x_{\ell}(y_1,\ldots,y_{\ell-1})+\|h\|_{\infty,\cI}\Bigr)\rd\tau
  \rd\nu_0^x(\phi)\\
  \le&\|\nu_{\cdot}\|\Bigl(\sum_{\ell\in J}\|\nu_{\cdot}\|^{\ell-1}\|g_{\ell}\|_{\infty}\|\eta_{\ell}\|
  +\|h\|_{\infty,\cI}\Bigr)|t-s|\\
  \le&L_1\|\nu_{\cdot}\||t-s|\to0,\quad \text{as}\quad |s-t|\to0,
\end{align*}
where \begin{equation*}
L_1\coloneqq L_1(\nu_{\cdot})=\BL(h)+\sum_{\ell\in J}\BL(g_{\ell})\|\eta_{\ell}\|\|\nu_{\cdot}\|^{\ell-1},
\end{equation*}
$\BL(h)=\sup_{t\in\cI}\sup_{x\in X}\BL(h(t,x,\cdot))$ and $\BL(g_{\ell})=\sup_{t\in\cI}\sup_{x\in X}\BL(g_{\ell}(t,x,\cdot))$ for $\ell\in J$.

This shows that $t\mapsto\cA[\eta,h]\nu_{t}\in \mathcal{C}(\cI,\mathcal{B}_{*}(X,\cM_+(Y)))$.

The proof of the mass conservation law is the same as that of \cite[Proposition~4.4]{KX21}.
\item[(ii)] Next, we show $\cA^x[\eta,h]\nu_{\cdot}$ is Lipschitz continuous in $\nu_{\cdot}$. We first prove Lipschitz continuity of $\Phi_{t,0}^x\phi$ in the initial condition $\phi$. Note that
   \begin{align}
    \nonumber &|V[\eta,\nu_{\cdot},h](t,x,\phi_1)-V[\eta,\nu_{\cdot},h](t,x,\phi_2)|\\
   \nonumber  \le&\sum_{\ell\in J}\int_{X^{\ell-1}}\underbrace{\int_{Y}\ldots\int_{Y}}_{\ell-1}
     |g_{\ell}(t,\phi_1,\psi_1,\ldots,\psi_{\ell-1})-g_{\ell}(t,\phi_2,\psi_1,\ldots,\psi_{\ell-1})|
     \rd\nu_t^{y_{\ell-1}}(\psi_{\ell-1})\\
\nonumber&\cdots\rd\nu_t^{y_{1}}(\psi_{1})\cdot\rd\eta^x_{\ell}(y_1,\ldots,y_{\ell-1})+|h(t,x,\phi_1)-
h(t,x,\phi_2)|\\
\label{V-Lip}\le&L_1(\nu_{\cdot})|\phi_1-\phi_2|,\quad t\in\cI,
   \end{align}
   This yields
\begin{align*}
    &|\Phi_{t,0}^x\phi_1(x)-\Phi_{t,0}^x\phi_2(x)|\\
    \le&|\phi_1(x)-\phi_2(x)|+\int_0^t|V[\eta,\nu_{\cdot},h](\tau,x,\Phi_{\tau,0}^x\phi_1(x))-
    V[\eta,\nu_{\cdot},h](\tau,x,\Phi_{\tau,0}^x\phi_2(x))|\rd\tau\\
    \le&|\phi_1(x)-\phi_2(x)|+L_1(\nu_{\cdot})
    \int_0^t|\Phi_{\tau,0}^x\phi_1(x)-\Phi_{\tau,0}^x\phi_2(x)|\rd\tau.
    \end{align*}
    By Gronwall's inequality,
    \begin{equation}\label{Lip-phi}|\Phi_{t,0}^x\phi_1(x)-\Phi_{t,0}^x\phi_2(x)|\le \textnormal{e}^{L_1(\nu_{\cdot})t}|\phi_1(x)-\phi_2(x)|,\end{equation}

Similarly, one can also show that
\begin{equation*}|\Phi_{0,t}^x\phi_1(x)-\Phi_{0,t}^x\phi_2(x)|\le \textnormal{e}^{L_1(\nu_{\cdot})t}|\phi_1(x)-\phi_2(x)|.\end{equation*}

Now we are ready to show $\cA^x[\eta,h]\nu_{\cdot}$ is Lipschitz continuous in $\nu_{\cdot}$. Observe that
\begin{align}
\nonumber&d_{\sf BL}(\Phi_{t,0}^x[\nu^{1}_{\cdot}]_{\#}\nu_0^{1,x},\Phi_{t,0}^x[\nu^{2}_{\cdot}]_{\#}\nu_0^{2,x})\\
\label{sum}\le&d_{\sf BL}(\Phi_{t,0}^x[\nu^{1}_{\cdot}]_{\#}\nu_0^{1,x},\Phi_{t,0}^x[\nu^{2}_{\cdot}]_{\#}\nu_0^{1,x})+d_{\sf BL}(\Phi_{t,0}^x[\nu^{2}_{\cdot}]_{\#}\nu_0^{1,x},\Phi_{t,0}^x[\nu^{2}_{\cdot}]_{\#}\nu_0^{2,x}).
\end{align}
Note that
\begin{align}
 \nonumber&|V[\nu^1_{\cdot}](t,x,\phi)-V[\nu^2_{\cdot}](t,x,\phi)|\\
\nonumber\le&\sum_{\ell\in J}\int_{X^{\ell-1}}\Bigl|\int_{Y^{\ell-1}}
     g_{\ell}(t,\phi,\psi_1,\ldots,\psi_{\ell-1})\rd(\otimes_{j=1}^{\ell-1}\nu_t^{1,y_j}(\psi_j)
     -\otimes_{j=1}^{\ell-1}\nu_t^{2,y_j}(\psi_j))\Bigr|\\
\nonumber&\cdot\rd\eta^x_{\ell}(y_1,\ldots,y_{\ell-1})\\
\nonumber=&\sum_{\ell\in J}\int_{X^{\ell-1}}\Bigl|\int_{Y^{\ell-1}}
     g_{\ell}(t,\phi,\psi_1,\ldots,\psi_{\ell-1})\rd(\nu_t^{1,y_1}(\psi_1)\otimes\nu_t^{1,y_2}(\psi_2)
     \cdots\otimes\nu_t^{1,y_{\ell-1}}(\psi_{\ell-1})\\
\nonumber&-\nu_t^{2,y_1}(\psi_1)\otimes\nu_t^{1,y_2}(\psi_2)
     \cdots\otimes\nu_t^{1,y_{\ell-1}}(\psi_{\ell-1})+\nu_t^{2,y_1}(\psi_1)\otimes\nu_t^{1,y_2}(\psi_2)
     \cdots\otimes\nu_t^{1,y_{\ell-1}}(\psi_{\ell-1})\\
\nonumber&-\nu_t^{2,y_1}(\psi_1)\otimes\nu_t^{2,y_2}(\psi_2)\otimes\nu_t^{1,y_3}(\psi_3)
     \cdots\otimes\nu_t^{1,y_{\ell-1}}(\psi_{\ell-1})+\cdots\\
\nonumber&+\nu_t^{2,y_1}(\psi_1)\otimes\nu_t^{2,y_2}(\psi_2)
     \cdots\otimes\nu_t^{2,y_{\ell-2}}(\psi_{\ell-2})\otimes\nu_t^{1,y_{\ell-1}}(\psi_{\ell-1})
     -\otimes_{j=1}^{\ell-1}\nu_t^{2,y_j}(\psi_j))\Bigr|\\
\nonumber&\cdot\rd\eta^x_{\ell}(y_1,\ldots,y_{\ell-1})\\
\nonumber\le&L_2d_{\infty}(\nu^1_{t},\nu^2_{t}),\quad t\in\cI,
\end{align}
where $L_2=L_2(\eta,\nu_{\cdot}^1,\nu_{\cdot}^2)\coloneqq\begin{cases}\sum_{\ell\in J}\|\eta_{\ell}\|\BL(g_{\ell})
\sum_{i=1}^{\ell-2}\|\nu_{\cdot}^1\|^i\|\nu_{\cdot}^2\|^{\ell-2-i},& \text{if}\quad \ell>2,\\ \sum_{\ell\in J}\|\eta_{\ell}\|\BL(g_{\ell}),& \text{if}\quad \ell=2.\end{cases}$ We now estimate the first term.
\begin{align}
\nonumber&d_{\sf BL}(\Phi^x_{t,0}[\nu^1_{\cdot}]_{\#}\nu_0^{1,x},\Phi^x_{t,0}[\nu^2_{\cdot}]_{\#}\nu_0^{1,x})\\
\nonumber=&\sup_{f\in\mathcal{BL}_1(Y)}\int_{Y} f(\phi)\mathrm{d}(\Phi^x_{t,0}[\nu^1_{\cdot}]_{\#}\nu_0^{1,x}-\Phi^x_{t,0}[\nu^1_{\cdot}]_{\#}\nu_0^{2,x})\\
\nonumber=&\sup_{f\in\mathcal{BL}_1(Y)}\int_{Y} ((f\circ \Phi^x_{t,0}[\nu^1_{\cdot}])(\phi)-(f\circ \Phi^x_{t,0}[\nu^2_{\cdot}])(\phi))\mathrm{d}\nu_0^{1,x}(\phi)\\
\nonumber\le&\int_{Y}\lt|\Phi^x_{t,0}[\nu^1_{\cdot}](\phi)- \Phi^x_{t,0}[\nu^2_{\cdot}](\phi)\rt|\mathrm{d}\nu_0^{1,x}(\phi)\eqqcolon\lambda_x(t)\\
\nonumber\le&\int_{Y}\int_0^t\lt|V[\nu^1_{\cdot}](\tau,x,\phi)-V[\nu^2_{\cdot}](\tau,x,\phi)\rt|\rd\tau
\mathrm{d}\nu_0^{1,x}(\phi)\\
\label{term-1}\le&L_2(\nu_{\cdot}^1,\nu_{\cdot}^2)\nu_0^{1,x}(Y)\int_0^td_{\infty}
(\nu^1_{\tau},\nu^2_{\tau})\rd\tau.
\end{align}

Next, we estimate the second term. For $f\in\mathcal{BL}_1(Y)$, from \eqref{Lip-phi} it follows that
\[\mathcal{L}(f\circ \Phi^x_{t,0}[\nu^2_{\cdot}])\le \mathcal{L}(f)\mathcal{L}(\Phi^x_{t,0}[\nu^2_{\cdot}])
\le\mathcal{L}(f)\textnormal{e}^{L_1(\nu^2_{\cdot})t},\quad \|f\circ \Phi^x_{t,0}[\nu^2_{\cdot}]\|_{\infty}\le\|f\|_{\infty}.\]
Hence \begin{equation}
  \label{Eq-A}
\BL(f\circ \Phi^x_{t,0})\le \textnormal{e}^{L_1(\nu^2_{\cdot})t}.
\end{equation} For every $x\in X$,
\begin{alignat}{2}
 \nonumber &d_{\sf BL}(\Phi^x_{t,0}[\nu^1_{\cdot}]_{\#}\nu_0^{1,x},\Phi^x_{t,0}[\nu^2_{\cdot}]_{\#}\nu_0^{1,x})\\
  \nonumber=&\sup_{f\in\mathcal{BL}_1(Y)}\int_Y(f\circ \Phi^x_{t,0}[\nu^2_{\cdot}])(\phi)\rd(\nu_0^{1,x}(\phi)-\nu_0^{2,x}(\phi))\\
 \label{term-2} \le&\textnormal{e}^{L_1(\nu^2_{\cdot})t}d_{\sf BL}(\nu_0^{1,x},\nu^{2,x}_0)\le \textnormal{e}^{L_1(\nu^2_{\cdot})t}d_{\infty}(\nu^1_0,\nu^2_0).
\end{alignat}
Combining \eqref{term-1} and \eqref{term-2}, it follows from \eqref{sum} that
\begin{align}
\nonumber d_{\infty}(\cA[\eta,h]\nu^1_{t},\cA[\eta,h]\nu^2_{t})
=&\sup_{x\in X}d_{\sf BL}(\Phi^x_{t,0}[\nu^1_{\cdot}]_{\#}\nu_0^{1,x},\Phi^x_{t,0}[\nu^2_{\cdot}]_{\#}\nu_0^{1,x})\\
 \nonumber\le&\textnormal{e}^{L_1(\nu^2_{\cdot})t}d_{\infty}(\nu^1_0,\nu^2_0)+ \|\nu^1_{\cdot}\|L_2(\eta,\nu_{\cdot}^1,\nu_{\cdot}^2)\textnormal{e}^{L_1(\nu^2_{\cdot})t}
 \int_0^td_{\infty}(\nu^1_{\tau},\nu^2_{\tau})
  \textnormal{e}^{-L_1(\nu^2_{\cdot})\tau}\rd\tau.
\end{align}

\item[(iii)] Lipschitz continuity of $\cA[\eta,h]$ in $h$.

We first need to establish the Lipschitz continuity for $\Phi^x_{s,t}[h]$ on $h$. Note that
\begin{align*}
  &|\Phi^x_{t,0}[h_1]\phi-\Phi^x_{t,0}[h_2]\phi|\\
  \le&\int_0^t|h_1(\tau,x,\Phi^x_{\tau,0}[h_1]\phi)-h_2(\tau,x,\Phi^x_{\tau,0}[h_2]\phi)|
  \rd\tau\\
  &+\int_0^t\sum_{\ell\in J}\int_{X^{\ell-1}}\int_{Y}\cdots\int_{Y}
  \Bigl|g_{\ell}(\tau,\Phi^x_{0,\tau}[h_1]\phi,\psi_1,\psi_2,\ldots,\psi_{\ell-1})\\
  &-g_{\ell}(\tau,\Phi^x_{0,\tau}[h_2]\phi,\psi_1,\psi_2,\ldots,\psi_{\ell-1})\Bigr|
  \rd\nu^{y_{\ell-1}}_{\tau}(\psi_{\ell-1})\cdots\rd\nu^{y_1}_{\tau}(\psi_1)\rd\eta^{x}_{\ell}
  (y_1,\ldots,y_{\ell-1})\rd\tau\\
  \le&\int_0^t|h_1(\tau,x,\Phi^x_{\tau,0}[h_1]\phi)-h_2(\tau,x,\Phi^x_{\tau,0}[h_1]\phi)|
  \rd\tau
  +\int_0^t|h_2(\tau,x,\Phi^x_{\tau,0}[h_1]\phi)-h_2(\tau,x,\Phi^x_{\tau,0}[h_2]\phi)|
  \rd\tau\\
  &+\int_0^t\sum_{\ell\in J}\int_{X^{\ell-1}}\int_{Y}\cdots\int_{Y}
  \Bigl|g_{\ell}(\tau,\Phi^x_{\tau,0}[h_1]\phi,\psi_1,\psi_2,\ldots,\psi_{\ell-1})
  -g_{\ell}(\tau,\Phi^x_{\tau,0}[h_2]\phi,\psi_1,\psi_2,\ldots,\psi_{\ell-1})\Bigr|\\
  &\rd\nu^{y_{\ell-1}}_{\tau}(\psi_{\ell-1})\cdots\rd\nu^{y_1}_{\tau}(\psi_1)\rd\eta^{x}_{\ell}
  (y_1,\ldots,y_{\ell-1})\rd\tau\\
  \le&T\|h_1-h_2\|_{\infty,\cI}+(\Lip(h_2)+\sum_{\ell\in J}\Lip(g_{\ell})\|\eta_{\ell}\|\|\nu_{\cdot}\|^
  {\ell-1})\int_0^t
  |\Phi^x_{\tau,0}[h_1]\phi-\Phi^x_{\tau,0}[h_2]\phi|\rd\tau\\
  \le&T\|h_1-h_2\|_{\infty,\cI}+L_1(\nu_{\cdot})\int_0^t
  |\Phi^x_{\tau,0}[h_1]\phi-\Phi^x_{\tau,0}[h_2]\phi|\rd\tau.
\end{align*}
By Gronwall's inequality,
\begin{align*}
  |\Phi^x_{t,0}[h_1]\phi-\Phi^x_{t,0}[h_2]\phi|\le&T\e^{L_1(\nu_{\cdot})t}\|h_1-h_2\|_{\infty,\cI}.
\end{align*}
This shows that
\begin{align*}
  &d_{\infty}(\cA[\eta,h_1]\nu_{t},\cA[\eta,h_2]\nu_{t})\\
  =&\sup_{x\in X}d_{\sf BL}(\Phi^x_{t,0}[\eta,\nu_{\cdot},h_1]\#\nu_0^{x},\Phi^x_{t,0}[\eta,\nu_{\cdot},h_2]\#\nu_0^{x})\\
  =&\sup_{x\in X}\sup_{f\in\mathcal{BL}_1(Y)}\lt|\int_Y\lt(f\circ \Phi^x_{t,0}[\eta,\nu_{\cdot},h_1]\phi-f\circ \Phi^x_{t,0}[\eta,\nu_{\cdot},h_2]\phi\rt)\rd\nu_0^x(\phi)\rt|\\
  \le&\int_Y|\Phi^x_{t,0}[\eta,\nu_{\cdot},h_1]\phi-\Phi^x_{t,0}[\eta,\nu_{\cdot},h_2]\phi|
  \rd\nu_0^x(\phi)\\
  \le&L_3\|h_1-h_2\|_{\infty,\cI},
\end{align*}
where \begin{equation}
  \label{Eq-L3}
L_3=L_3(\nu_{\cdot})\coloneqq T\e^{L_1(\nu_{\cdot})T}\|\nu_{\cdot}\|.\end{equation}
\end{enumerate}
\end{proof}

\section{Proof of Proposition~\ref{prop-sol-fixedpoint}}\label{appendix-prop-sol-fixedpoint}
\begin{proof}
The unique existence of solutions to the fixed point equation \eqref{Fixed} is proved by the Banach fixed point theorem, which is analogous to that of \cite[Proposition~{4.5}]{KX21}.
\begin{enumerate}
\item[$\bullet$]
Continuity in $t$. It follows directly from Proposition~\ref{prop-continuousdependence}(i).
\item[$\bullet$] Continuity in $x$. Assume $\nu_{\cdot}\in \mathcal{C}(\cI,\mathcal{C}_{*}(X,\cM_+(Y)))$. We will show $\cA[\eta,h]\nu_{\cdot}\in \mathcal{C}(\cI,\mathcal{C}_{*}(X,\cM_+(Y)))$. It suffices to show that
the continuity of measures in $x$ is preserved: $x\mapsto\nu_0^x\circ \Phi_{t,0}^x[\eta,\nu_{\cdot},h]$ is continuous.  Indeed,
\begin{align*}
  &d_{\sf BL}(\Phi_{t,0}^x[\eta,\nu_{\cdot},h]\#\nu_0^x,\Phi_{t,0}^{x'}[\eta,\nu_{\cdot},h]\#\nu_0^{x'})\\
  =&\sup_{f\in\BL_1(Y)}\lt|\int_Yf\circ \Phi_{t,0}^x[\eta,\nu_{\cdot},h]\phi\rd\nu_0^x(\phi)-f\circ \Phi_{t,0}^{x'}[\eta,\nu_{\cdot},h]\phi\rd\nu_0^{x'}(\phi)\rt|\\
  \le&\int_Y\lt|\Phi_{t,0}^x[\eta,\nu_{\cdot},h]\phi-\Phi^{x'}_{t,0}[\eta,\nu_{\cdot},h]\phi\rt|
  \rd\nu_0^x(\phi)\\&+\sup_{f\in\BL_1(Y)}\lt|\int_Yf\circ\cS_{t,0}^{x'}[\eta,\nu_{\cdot},h]\phi
  \rd(\nu_0^x(\phi)-\nu_0^{x'}(\phi))\rt|.
\end{align*}
It follows from \eqref{Lip-phi} that $\Phi_{t,0}^{x'}[\eta,\nu_{\cdot},h]\phi$ is Lipschitz continuous in $\phi$ with constant $\e^{L_1(\nu_{\cdot})T}$, and from \eqref{Eq-A} it follows that \begin{equation*}
  \frac{f\circ\cS_{t,0}^{x'}[\eta,\nu_{\cdot},h]}{\textnormal{e}^{L_1(\nu_{\cdot})T}}\in\BL_1(Y).
  \end{equation*}
In addition, from \eqref{V-Lip}, we have
\begin{align*}
 &|V[\eta,\nu_{\cdot},h](\tau,x,\Phi_{\tau,0}^x[\eta,\nu_{\cdot},h]\phi)
  -V[\eta,\nu_{\cdot},h](\tau,x',\Phi_{\tau,0}^{x'}[\eta,\nu_{\cdot},h]\phi)|\\
  \le&|V[\eta,\nu_{\cdot},h](\tau,x,\Phi_{\tau,0}^x[\eta,\nu_{\cdot},h]\phi)
  -V[\eta,\nu_{\cdot},h](\tau,x,\Phi_{\tau,0}^{x'}[\eta,\nu_{\cdot},h]\phi)|\\
  &+|V[\eta,\nu_{\cdot},h](\tau,x,\Phi_{\tau,0}^{x'}[\eta,\nu_{\cdot},h]\phi)
  -V[\eta,\nu_{\cdot},h](\tau,x',\Phi_{\tau,0}^{x'}[\eta,\nu_{\cdot},h]\phi)|\\
 \le&L_1(\nu_{\cdot})
 |\Phi_{\tau,0}^x[\eta,\nu_{\cdot},h]\phi-\Phi_{\tau,0}^{x'}[\eta,\nu_{\cdot},h]\phi|
 +\sum_{\ell\in J}\Bigl|\int_{X^{\ell-1}}\int_{Y^{k_{\ell-1}}}g_{\ell}(\tau,\Phi_{\tau,0}^{x'}
 [\eta,\nu_{\cdot},h]\phi,\psi_1,\ldots,
 \psi_{\ell-1}))\\
 &\otimes_{j=1}^{\ell-1}\rd\nu_{\tau}^{y_j}(\psi_j)
\rd(\eta_{\ell}^x(y_1,\ldots,y_{\ell-1})-\eta_{\ell}^{x'}(y_1,\ldots,y_{\ell-1}))\Bigr|\\
 &+|h(\tau,x,\Phi_{\tau,0}^{x'}[\eta,\nu_{\cdot},h]\phi)-h(\tau,x',\Phi_{\tau,0}^{x'}[\eta,\nu_{\cdot},h]\phi)|\\
 \le&L_1(\nu_{\cdot})
 |\Phi_{\tau,0}^x[\eta,\nu_{\cdot},h]\phi-\Phi_{\tau,0}^{x'}[\eta,\nu_{\cdot},h]\phi|+\sup_{\varphi\in Y}|h(\tau,x,\varphi)-h(\tau,x',\varphi)|\\
 &+\sum_{\ell\in J}\Bigl|\int_{X^{\ell-1}}\int_{Y^{\ell-1}}g_{\ell}(\tau,\Phi_{\tau,0}^{x'}
 [\eta,\nu_{\cdot},h]\phi,
 \psi_1,\ldots,\psi_{\ell-1})
 \otimes_{j=1}^{\ell-1}\rd\nu_{\tau}^{y_j}(\psi_j)\\
 &\rd(\eta_{\ell}^x(y_1,\ldots,y_{\ell-1})-\eta^{x'}_{\ell}(y_1,\ldots,y_{\ell-1}))\Bigr|\\
 \le& L_1(\nu_{\cdot})
 |\Phi_{\tau,0}^x[\eta,\nu_{\cdot},h]\phi-\Phi_{\tau,0}^{x'}[\eta,\nu_{\cdot},h]\phi|+\sup_{\varphi\in Y}|h(\tau,x,\varphi)-h(\tau,x',\varphi)|\\
 &+\sum_{\ell\in J}\Bigl|\int_{X^{\ell-1}}\Bigl(\int_{Y^{\ell-1}}g_{\ell}
 (\tau,\Phi_{\tau,0}^{x'}[\eta,\nu_{\cdot},h]\phi,
 \psi_1,\ldots,\psi_{\ell-1})
 \otimes_{j=1}^{\ell-1}\rd\nu_{\tau}^{y_j}(\psi_j)\Bigr)\\
 &\rd(\eta_{\ell}^x(y_1,\ldots,y_{\ell-1})-\eta^{x'}_{\ell}(y_1,\ldots,y_{\ell-1}))\Bigr|.
\end{align*}
This implies that
\begin{align*}
  &|\Phi_{t,0}^x[\eta,\nu_{\cdot},h]\phi-\Phi_{t,0}^{x'}[\eta,\nu_{\cdot},h]\phi|\\
  =&\int_0^t|V[\eta,\nu_{\cdot},h](\tau,x,\Phi_{\tau,0}^x[\eta,\nu_{\cdot},h]\phi)
  -V[\eta,\nu_{\cdot},h](\tau,x',\Phi_{\tau,0}^{x'}[\eta,\nu_{\cdot},h]\phi)|\rd\tau\\
  \le&L_1(\nu_{\cdot})
 \int_0^t|\Phi_{\tau,0}^x[\eta,\nu_{\cdot},h]\phi-\Phi_{\tau,0}^{x'}[\eta,\nu_{\cdot},h]\phi|\rd\tau
 +\int_0^t\sup_{\varphi\in Y}|h(\tau,x,\varphi)-h(\tau,x',\varphi)|\rd\tau\\
 &+\sum_{\ell\in J}\int_0^t\Bigl|\int_{X^{\ell-1}}\int_{Y^{\ell-1}}g_{\ell}
 (\tau,\Phi_{\tau,0}^{x'}[\eta,\nu_{\cdot},h]\phi,\psi_1,\cdots,\psi_{\ell-1})
 \otimes_{j=1}^{\ell-1}\rd\nu_{\tau}^{y_j}(\psi_j)\\
 &\rd(\eta_{\ell}^x(y_1,\ldots,y_{\ell-1})-\eta_{\ell}^{x'}(y_1,\ldots,y_{\ell-1}))\Bigr|\rd\tau.
\end{align*}

By Gronwall's inequality,
\begin{align*}
  &|\Phi_{t,0}^x[\eta,\nu_{\cdot},h]\phi-\Phi_{t,0}^{x'}[\eta,\nu_{\cdot},h]\phi|\\
  \le&\Bigl(\sum_{\ell\in J}\int_0^t\Bigl|\int_{X^{\ell-1}}\int_{Y^{\ell-1}}g_{\ell}(\tau,
  \Phi_{\tau,0}^{x'}[\eta,\nu_{\cdot},h]\phi,\psi_1,\cdots,\psi_{\ell_k-1})
 \otimes_{j=1}^{\ell-1}\rd\nu_{\tau}^{y_j}(\psi_j)\rd(\eta_{\ell}^x(y_1,\ldots,y_{\ell-1})\\
 &-\eta_{\ell}^{x'}(y_1,\ldots,y_{\ell-1}))\Bigr|\rd\tau+\int_0^t\sup_{\varphi\in Y}|h(\tau,x,\varphi)-h(\tau,x',\varphi)|\rd\tau\Bigr)
 \textnormal{e}^{L_1(\nu_{\cdot})t}.
\end{align*}
By \eqref{Eq-A}, this further shows that
\begin{align}
 \nonumber &d_{\sf BL}(\Phi_{t,0}^x[\eta,\nu_{\cdot},h]\#\nu_0^x,\Phi_{t,0}^{x'}[\eta,\nu_{\cdot},h]\#\nu_0^{x'})\\
 \nonumber \le&\int_Y|\Phi_{t,0}^x[\eta,\nu_{\cdot},h]\phi
  -\Phi_{t,0}^{x'}[\eta,\nu_{\cdot},h]\phi|\rd\nu_0^x(\phi)+ \textnormal{e}^{L_1(\nu_{\cdot})T}d_{\sf BL}(\nu_0^x,\nu_0^{x'})\\
  \label{Eq-star-star} \le&\e^{L_1(\nu_{\cdot})T}\Bigl(\int_0^t\int_Y\sum_{\ell\in J}\Bigl|\int_{X^{\ell-1}}\int_{Y^{\ell-1}}g_{\ell}(\tau,
   \Phi_{\tau,0}^{x'}[\eta,\nu_{\cdot},h]\phi,\psi_1,\ldots,\psi_{\ell-1})
   \rd\otimes_{j=1}^{\ell-1}\nu_{\tau}^{y_j}(\psi_j)\\
   \nonumber&\rd(\eta^x_{\ell}(y_1,\ldots,y_{\ell-1})-\eta^{x'}_{\ell}(y_1,\ldots,y_{\ell-1}))\Bigr|
   \rd\nu_0^x(\phi)\rd\tau\\
\nonumber&+\|\nu_0\|\int_0^t
 \sup_{\varphi\in Y}|h(\tau,x,\varphi)-h(\tau,x',\varphi)|\rd\tau+d_{\sf BL}(\nu_0^x,\nu_0^{x'})\Bigr).
\end{align}
Since $\nu_{\cdot}\in \mathcal{C}(\cI,\mathcal{C}_{*}(X,\cM_+(Y))$, by Proposition~\ref{prop-nu}(iv), $\otimes_{j=1}^{\ell-1}\nu_t^{y_j}$ is weakly continuous in $(y_1,\ldots,y_{\ell-1})$.
By $(\mathbf{A2})$, $g_{\ell}$ is bounded Lipschitz continuous in $\phi,\psi_1,\ldots,\psi_{\ell-1}$; and $\Phi_{\tau,0}^{x'}[\eta,\nu_{\cdot},h]\phi$ is Lipschitz continuous in $\phi$ by \eqref{Lip-phi}, we have $g_{\ell}(\tau,\Phi_{\tau,0}^{x'}[\eta,\nu_{\cdot},h]\phi,\psi_1,\ldots,$ $\psi_{\ell-1})$ is bounded continuous in $\phi$ and $\psi_1\ldots,\psi_{\ell-1}$. Hence for $\ell\in J$,\\ \noindent $\int_{Y^{\ell-1}}g_{\ell}(\tau,\Phi_{\tau,0}^{x'}[\eta,\nu_{\cdot},h]\phi,
\psi_1,\ldots,\psi_{\ell-1})
 \rd\otimes_{j=1}^{\ell-1}\nu_{\tau}^{y_j}(\psi_j)$ is continuous in $(y_1,\ldots,y_{\ell-1})$. Moreover, $$\Bigl|\int_{Y^{\ell-1}}g_{\ell}(\tau,
 \Phi_{\tau,0}^{x'}[\eta,\nu_{\cdot},h]\phi,\psi_1,\ldots,\psi_{\ell-1})
 \rd\otimes_{j=1}^{\ell-1}\nu_{\tau}^{y_j}(\psi_j)\Bigr|\le\BL(g_{\ell})
 \|\nu_{\cdot}\|^{\ell-1}<\infty$$ is also uniformly bounded for $x\in X$. Hence by Proposition~\ref{prop-nu}(iii), we know
\begin{equation}\label{Eq-B}
\begin{split}
  &\lim_{|x-x'|\to0}\Bigl|\int_{X^{\ell-1}}\int_{Y^{\ell-1}}g_{\ell}(\tau,\Phi_{\tau,0}^{x'}
  [\eta,\nu_{\cdot},h]\phi,\psi_1,\ldots,\psi_{\ell-1})
 \rd\otimes_{j=1}^{\ell-1}\nu_{\tau}^{y_j}(\psi_j)\\
 &\rd(\eta_{\ell}^x(y_1,\ldots,y_{\ell-1})-\eta_{\ell}^{x'}(y_1,\ldots,y_{\ell-1}))\Bigr|=0,
\end{split} \end{equation} since $\eta\in \mathcal{C}(X,\cM_+(X^{\ell-1}))$.
Notice that \begin{align*}
   &\int_0^t\Bigl|\int_{X^{\ell-1}}\int_{Y^{\ell-1}}g_{\ell}(\tau,
   \Phi_{\tau,0}^y[\eta,\nu_{\cdot},h]\phi,
   \psi_1,\ldots,\psi_{\ell-1})
 \rd\otimes_{j=1}^{\ell-1}\nu_{\tau}^{y_j}(\psi_j)\rd\eta_{\ell}^x(y_1,\ldots,y_{\ell-1})
 \Bigr|\rd\tau\\
 \le&\BL(g_{\ell})\|\nu_{\cdot}\|^{\ell-1}\|\eta_{\ell}\|T<\infty,
  \end{align*}
 by the Dominated Convergence Theorem, it follows from \eqref{Eq-B} that
 \begin{align*}
 &\lim_{|x-x'|\to0}\int_0^t\Bigl|\int_{X^{\ell-1}}\int_{Y}\ldots\int_Y
 g_{\ell}(\tau,\Phi_{\tau,0}^{x'}[\eta,\nu_{\cdot},h]\phi,\psi_1,\ldots,\psi_{\ell-1})
 \rd\otimes_{j=1}^{\ell-1}\nu_{\tau}^{y_j}(\psi_j)\\
 &\rd(\eta^x_{\ell}(y_1,\ldots,y_{\ell-1})-\eta^{x'}_{\ell}(y_1,\ldots,y_{\ell-1}))\Bigr|\rd\tau=0.
 \end{align*}
Moreover, by
 $(\mathbf{A7})$ as well as the Dominated Convergence Theorem again,
 \[\lim_{|x-x'|\to0}\int_0^t\sup_{\varphi\in Y}|h(\tau,x,\varphi)-h(\tau,x',\varphi)|\rd\tau=0.\] Since
$\nu_0\in\mathcal{C}_{*}(X,\cM_+(Y))$, we have $$\lim_{|x-x'|\to0}d_{\sf BL}(\nu_0^x,\nu_0^{x'})=0.$$ Hence from \eqref{Eq-star-star} it follows that
\[\lim_{|x-x'|\to0}d_{\sf BL}(\Phi_{t,0}^x[\eta,\nu_{\cdot},h]\#\nu_0^x,\Phi_{t,0}^{x'}[\eta,\nu_{\cdot},h]\#\nu_0^{x'})=0.\]
\end{enumerate}
{The absolute continuity of solutions follow from  Proposition~\ref{prop-continuousdependence}(iv).}
In the following, we prove properties (i)-(iii) item by item.
\begin{enumerate}
\item[(i)] Lipschitz continuity in $\nu_{0}$. It follows from Proposition~\ref{prop-continuousdependence}(ii) via Gronwall inequality.
\item[(ii)] Lipschitz continuity of $\nu_{\cdot}$ in $h$. Assume $\nu^1_0=\nu^2_0$.

We first need to establish the Lipschitz continuity for $\Phi^x_{s,t}[h]$. Note that
\begin{align*}
  &|\Phi^x_{0,t}[\nu^1,h_1]\phi-\Phi^x_{0,t}[\nu^2,h_2]\phi|\\
\le&|\Phi^x_{0,t}[\nu^1,h_1]\phi-\Phi^x_{0,t}[\nu^1,h_2]\phi|+|\Phi^x_{0,t}[\nu^1,h_2]\phi-\Phi^x_{0,t}[\nu^2,h_2]\phi|.
\end{align*}
The second term follows from \eqref{term-2}. The estimate for the first term follows from Proposition~\ref{prop-continuousdependence}~(iii).  It follows from \eqref{V-Lip} that
\begin{align*}
  \le&\lt|\int_0^t\lt(V[\nu^1,h_1](\tau,x,\Phi^x_{0,\tau}[h_1]\phi)
  -V[\nu^1,h_2](\tau,x,\Phi^x_{0,\tau}[h_2]\phi)\rt)\rd\tau\rt|\\
  \le&\int_0^t\lt|V[\nu^1,h_1](\tau,x,\Phi^x_{0,\tau}[h_1]\phi)
  -V[\nu^1,h_2](\tau,x,\Phi^x_{0,\tau}[h_1]\phi(x))\rt|\rd\tau\\
  &+\int_0^t\lt|V[\nu^1,h_2](\tau,x,\Phi^x_{0,\tau}[h_1]\phi)
  -V[\nu^1,h_2](\tau,x,\Phi^x_{0,\tau}[h_2]\phi)\rt|\rd\tau\\
  \le&\int_0^t\lt|h_1(\tau,x,\Phi^x_{0,\tau}[h_1]\phi(x))
  -h_2(\tau,x,\Phi^x_{0,\tau}[h_1]\phi(x))\rt|\rd\tau\\
  &+L_1(\nu_{\cdot})\int_0^t\lt|\Phi^x_{0,\tau}[h_1]\phi-\Phi^x_{0,\tau}[h_2]\phi\rt|\rd\tau.
\end{align*}
By Gronwall's inequality, we have
\begin{align*}
\lt|\Phi^x_{0,t}[\nu^1,h_1]\phi-\Phi^x_{0,t}[\nu^2,h_2]\phi\rt|
\le&\textnormal{e}^{L_1(\nu_{\cdot})t}\int_0^t|h_1(\tau,x,\Phi^x_{0,\tau}[h_1]\phi)
-h_2(\tau,x,\Phi^x_{0,\tau}[h_1]\phi)|\rd\tau.
\end{align*}
Hence \begin{align*}
  &d_{\sf BL}(\Phi^x_{0,t}[h_1]\#\nu_0^x,\Phi^x_{0,t}[h_2]\#\nu_0^x)\\
  =&\sup_{f\in\mathcal{BL}_1(Y)}\int_Yf(\phi)\rd(\Phi^x_{0,t}[h_1]\#\nu_0^x-\Phi^x_{0,t}[h_2]\#\nu_0^x)\\
  =&\sup_{f\in\mathcal{BL}_1(Y)}\int_Y\lt((f\circ \Phi^x_{t,0}[h_1])(\phi)
  -(f\circ \Phi^x_{t,0}[h_2])(\phi)\rt)\rd\nu_0^x(\phi)\\
  \le&\int_Y\lt|\Phi^x_{t,0}[h_1]\phi
  -\Phi^x_{t,0}[h_2]\phi)\rt|\rd\nu_0^x(\phi)\\
  \le& \textnormal{e}^{L_1(\nu_{\cdot})t}\int_0^t\int_Y|h_1(\tau,x,\Phi^x_{0,\tau}[h_1]\phi)
  -h_2(\tau,x,\Phi^x_{0,\tau}[h_1]\phi)|\rd\nu_0^x(\phi)\rd\tau\\
  \le&\textnormal{e}^{L_1(\nu_{\cdot})t}\int_0^t\int_Y|h_1(\tau,x,\phi)
  -h_2(\tau,x,\phi)|\rd\nu_{\tau}^x(\phi)\rd\tau\\
  \le&L_3\|h_1-h_2\|_{\infty,\cI},
\end{align*}
where $L_3$ is defined in \eqref{Eq-L3}, which further implies that
\begin{align*}
&d_{\infty}(\mathcal{A}[\eta,h_1](\nu_{t}),\mathcal{A}[\eta,h_2](\nu_{t}))\\
=&\sup_{x\in X}d_{\sf BL}(\Phi^x_{0,t}[h_1]\#\nu_0^x,\Phi^x_{0,t}[h_2]\#\nu_0^x)\le L_3(\nu_{\cdot})\|h_1-h_2\|_{\infty,\cI}.
\end{align*}
\item[(iii)] Continuous dependence on $\eta$. Since $\nu_0\in \mathcal{C}_{*}(X,\cM_+(Y))$, we have $\nu_{\cdot}\in\mathcal{C}(\cI,\mathcal{C}_{*}$ $(X,\cM_+(Y)))$. {Based on the continuous dependence on the initial distributions proved in (i), as well as a triangle inequality, it suffices to prove the case assuming $$\nu^K_0=\nu_0$$} Let $\nu^K_{\cdot}\in \mathcal{C}(\mathcal{I},\mathcal{B}_{*}(X,\cM_+(Y)))$ with $\nu^K_0=\nu_0$ be the solutions to the fixed point equations
    \[\nu_t=\cA[\eta,h]\nu_t,\quad \nu^K_t=\cA[\eta^K,h]\nu^K_t,\quad t\in\cI,\]
    where $\eta^K=(\eta^\ell)_{\ell=1}^r$.
    Assume $$\lim_{K\to\infty}d_{\infty}(\eta^{\ell},\eta^{K,\ell})=0,\quad \ell\in J.$$ In the following, we show
    \[\lim_{K\to\infty}d_{\infty}(\cA[\eta,h]\nu_t,\cA[\eta^K,h]\nu^K_t)=0,\quad t\in\cI.\]
By the triangle inequality,    \begin{alignat}{2}
\nonumber      d_{\sf BL}(\nu_t^x,\nu^{K,x}_t)=&d_{\sf BL}(\Phi^x_{t,0}[\eta,\nu_{\cdot}]_{\#}\nu_0^x, \Phi^x_{t,0}[\eta^K,\nu^K_{\cdot}]_{\#}\nu_0^x)\\
\label{Eq-10}      \le&d_{\sf BL}(\Phi^x_{t,0}[\eta^K,\nu_{\cdot}]_{\#}\nu_0^x, \Phi^x_{t,0}[\eta^K,\nu^K_{\cdot}]_{\#}\nu_0^x)\\
\nonumber      &+d_{\sf BL}(\Phi^x_{t,0}[\eta,\nu_{\cdot}]_{\#}\nu_0^x, \Phi^x_{t,0}[\eta^K,\nu_{\cdot}]_{\#}\nu_0^x).
    \end{alignat}
    From \eqref{term-1} it follows that
   \begin{align}
   \nonumber&d_{\sf BL}(\Phi^x_{t,0}[\eta^K,\nu_{\cdot}]_{\#}\nu_0^x, \Phi^x_{t,0}[\eta^K,\nu^K_{\cdot}]_{\#}\nu_0^x)\\
   \nonumber\le&\int_Y|\Phi_{t,0}^x[\eta^K,\nu_{\cdot}]\phi
   -\Phi^x_{t,0}[\eta^K,\nu_{\cdot}^K]\phi|\rd\nu_0^x(\phi)\eqqcolon\beta^K_x(t),\\
   \label{Eq-11}\le& L_{2,K}(\eta^K)\|\nu_{\cdot}\|\textnormal{e}^{L_{1,K}(\nu_{\cdot}^K)t}
   \int_0^td_{\infty}(\nu_{\tau},\nu^K_{\tau})\textnormal{e}^{-L_{1,K}(\nu_{\cdot}^K)\tau}\rd\tau,
   \end{align}
    where the index $K$ in the constants indicates the dependence on $K$.

   We now estimate the second term. By \eqref{V-Lip},
\begin{align*}
  &d_{\sf BL}(\Phi^x_{t,0}[\eta,\nu_{\cdot}]_{\#}\nu_0^x, \Phi^x_{t,0}[\eta^K,\nu_{\cdot}]_{\#}\nu_0^x)\\
  =&\sup_{f\in\mathcal{BL}_1(Y)}\int_Yf\rd(\Phi^x_{t,0}[\eta,\nu_{\cdot}]_{\#}\nu_0^x-\Phi^x_{t,0}[\eta^K,\nu_{\cdot}]_{\#}\nu_0^x)\\
  =&\sup_{f\in\mathcal{BL}_1(Y)}\int_Y\lt((f\circ \Phi^x_{t,0}[\eta,\nu_{\cdot}])(\phi)-(f\circ \Phi^x_{t,0}[\eta^K,\nu_{\cdot}])(\phi)\rt)\rd\nu_0^x(\phi)\\
  \le&\int_Y|\Phi^x_{t,0}[\eta,\nu_{\cdot}]\phi-\Phi^x_{t,0}[\eta^K,\nu_{\cdot}]\phi)|\rd\nu_0^x(\phi)
  =\colon\gamma^K_x(t)\\
  =&\int_Y\lt|\int_0^t(V[\eta,\nu_{\cdot}](\tau,x,\Phi^x_{\tau,0}[\eta,\nu_{\cdot}]\phi)
  -V[\eta^K,\nu_{\cdot}](\tau,x,\Phi^x_{\tau,0}[\eta^K,\nu_{\cdot}]\phi))
  \mathrm{d}\tau\rt|
  \mathrm{d}\nu_0^x(\phi)\\
  \le&\int_Y\lt|\int_0^t(V[\eta,\nu_{\cdot}](\tau,x,\Phi^x_{\tau,0}[\eta,\nu_{\cdot}]\phi)
  -V[\eta,\nu_{\cdot}](\tau,x,\Phi^x_{\tau,0}[\eta^K,\nu_{\cdot}]\phi))
  \mathrm{d}\tau\rt|
  \mathrm{d}\nu_0^x(\phi)\\
  &+\int_Y\lt|\int_0^t(V[\eta,\nu_{\cdot}](\tau,x,\Phi^x_{\tau,0}[\eta^K,\nu_{\cdot}]\phi)
  -V[\eta^K,\nu_{\cdot}](\tau,x,\Phi^x_{\tau,0}[\eta^K,\nu_{\cdot}]\phi))
  \mathrm{d}\tau\rt|
  \mathrm{d}\nu_0^x(\phi)\\
  \le&\int_Y\int_0^t\lt|V[\eta,\nu_{\cdot}](\tau,x,\Phi^x_{\tau,0}[\eta,\nu_{\cdot}]\phi)
  -V[\eta,\nu_{\cdot}](\tau,x,\Phi^x_{\tau,0}[\eta^K,\nu_{\cdot}]\phi)\rt|
  \mathrm{d}\tau\mathrm{d}\nu_0^x(\phi)\\
  &+\int_Y\int_0^t\lt|V[\eta,\nu_{\cdot}](\tau,x,\Phi^x_{\tau,0}[\eta^K,\nu_{\cdot}]\phi)
  -V[\eta^K,\nu_{\cdot}](\tau,x,\Phi^x_{\tau,0}[\eta^K,\nu_{\cdot}]\phi)\rt|
  \mathrm{d}\tau\mathrm{d}\nu_0^x(\phi)\\
  \le&L_1(\nu_{\cdot})\int_0^t\int_Y\lt|\Phi^x_{\tau,0}[\eta,\nu_{\cdot}]\phi
  -\Phi^x_{\tau,0}[\eta^K,\nu_{\cdot}]\phi\rt|
\mathrm{d}\nu_0^x(\phi)\mathrm{d}\tau\\
  &+\int_0^t\int_Y\lt|V[\eta,\nu_{\cdot}](\tau,x,\Phi^x_{\tau,0}[\eta^K,\nu_{\cdot}]\phi)
  -V[\eta,\nu_{\cdot}](\tau,x,\Phi^x_{\tau,0}[\eta^K,\nu^K_{\cdot}]\phi)\rt|
\mathrm{d}\nu_0^x(\phi)\mathrm{d}\tau\\
  +&\int_0^t\int_Y\lt|V[\eta,\nu_{\cdot}](\tau,x,\Phi^x_{\tau,0}[\eta^K,\nu^K_{\cdot}]\phi)
  -V[\eta^K,\nu_{\cdot}](\tau,x,\Phi^x_{\tau,0}[\eta^K,\nu^K_{\cdot}]\phi)\rt|
  \mathrm{d}\nu_0^x(\phi)\mathrm{d}\tau\\
  +&\int_0^t\int_Y\lt|V[\eta^K,\nu_{\cdot}](\tau,x,\Phi^x_{\tau,0}[\eta^K,\nu^K_{\cdot}]\phi)
  -V[\eta^K,\nu_{\cdot}](\tau,x,\Phi^x_{\tau,0}[\eta^K,\nu_{\cdot}]\phi)\rt|
\mathrm{d}\nu_0^x(\phi)\mathrm{d}\tau\\
  \le&L_1(\nu_{\cdot})\int_0^t\gamma^K_x(\tau)\rd\tau
  +L_1(\nu_{\cdot})\int_0^t\int_Y\lt|\Phi^x_{\tau,0}[\eta^K,\nu_{\cdot}]\phi
  -\Phi^x_{\tau,0}[\eta^K,\nu^K_{\cdot}]\phi\rt|
\mathrm{d}\nu_0^x(\phi)\mathrm{d}\tau\\
&+\int_0^t\int_Y\lt|V[\eta,\nu_{\cdot}](\tau,x,\phi)
  -V[\eta^K,\nu_{\cdot}](\tau,x,\phi)\rt|
\mathrm{d}\nu^{K,x}_{\tau}(\phi)\mathrm{d}\tau\\
&+L_{1,K}(\nu_{\cdot})\int_0^t\int_Y\lt|\Phi^x_{\tau,0}[\eta^K,\nu^K_{\cdot}]\phi
-\Phi^x_{\tau,0}[\eta^K,\nu_{\cdot}]\phi\rt|
\mathrm{d}\nu_0^x(\psi)\mathrm{d}\tau \\
=&L_1(\nu_{\cdot})\int_0^t\gamma^K_x(\tau)\rd\tau+(
L_1(\nu_{\cdot})+L_{1,K}(\nu_{\cdot}))\int_0^t\beta^K_x(\tau)\rd\tau\\
&+\int_0^t\int_Y\lt|V[\eta,\nu_{\cdot}](\tau,x,\phi)
  -V[\eta^K,\nu_{\cdot}](\tau,x,\phi)\rt|
\mathrm{d}\nu^{K,x}_{\tau}(\phi)\mathrm{d}\tau.
\end{align*}
To obtain further estimates, let $$\zeta^K_x(\tau)\coloneqq\int_Y\lt|V[\eta,\nu_{\cdot}](\tau,x,\phi)
  -V[\eta^K,\nu_{\cdot}](\tau,x,\phi)\rt|
\mathrm{d}\nu_{\tau}^x(\phi).$$ By the triangle inequality,
\begin{align*}
&\int_0^t\int_Y\lt|V[\eta,\nu_{\cdot}](\tau,x,\phi)
  -V[\eta^K,\nu_{\cdot}](\tau,x,\phi)\rt|
\mathrm{d}\nu^{K,x}_{\tau}(\phi)\mathrm{d}\tau\\
\le& \int_0^t\zeta^K_x(\tau)\mathrm{d}\tau+\int_0^t\lt|\int_Y\lt|V[\eta,\nu_{\cdot}](\tau,x,\phi)
  -V[\eta^K,\nu_{\cdot}](\tau,x,\phi)\rt|
\mathrm{d}(\nu^{K,x}_{\tau}(\phi)-\nu_{\tau}^x(\phi))\rt|\mathrm{d}\tau.
\end{align*}
Recall that $$\sup_{x\in X}\eta^{K,x}_{\ell}(Y)\le \sup_{x\in X}\Bigl(\eta^{x}_{\ell}(Y)+d_{\infty}(\eta^x_{\ell},\eta^{K,x}_{\ell})\Bigr),\quad \ell\in J.$$ Moreover, since $$\lim_{k\to\infty}d_{\infty}(\nu_0^K,\nu_0)=0$$ and $$\sum_{\ell\in J}\left|\|\eta_{\ell}\|-\|\eta^{K}_{\ell}\|\right|\le\sum_{\ell\in J}d_{\infty}(\eta_{\ell},
\eta^{K}_{\ell})\to0,\quad \text{as}\quad K\to\infty,$$ we have there exists some $b>0$ independent of $K$ such that \begin{equation}\label{Eq-Star-1}\sup_{K\in\mathbb{N}}(L_1(\nu_{\cdot})+L_{1,K}(\nu_{\cdot}))\le b,\quad \sup_{K\in\mathbb{N}}(L_2(\eta,\nu_{\cdot},\nu_{\cdot}^K)+L_2(\eta^K,\nu_{\cdot},\nu_{\cdot}^K))\le b.\end{equation} Let
$$f_K(\tau,x,\varphi)\coloneqq\lt|V[\eta,\nu_{\cdot}](\tau,x,\varphi)
  -V[\eta^K,\nu_{\cdot}](\tau,x,\varphi)\rt|.$$
Using analogous arguments as in the proof for the limit \cite[(A.2)]{KX21}, we have
 \[\lim_{K\to\infty}\sup_{x\in X}|f_K(\tau,x,\varphi)|=0,\]
 which further implies that $f_K$ is bounded.  Moreover, it follows from \eqref{V-Lip} again that \begin{align*}
  &|f_K(\tau,x,\varphi)-f_K(\tau,x,\phi)|\\
  \le&\lt|V[\eta,\nu_{\cdot}](\tau,x,\varphi)
  -V[\eta,\nu_{\cdot}](\tau,x,\phi)\rt|+\lt|V[\eta^K,\nu_{\cdot}](\tau,x,\varphi)
  -V[\eta^K,\nu_{\cdot}](\tau,x,\phi)\rt|\\
  \le&(L_1+L_{1,K})|\varphi-\phi|\le b|\varphi-\phi|.
\end{align*}
Further, by \eqref{V-Lip}, one can show that $f_K(\tau,x,\varphi)$ is bounded Lipschitz continuous in $\varphi$ with some constant $\widehat{b}>0$ such that $$\sup_{K\in\N}\sup_{\tau\in\cI}\sup_{x\in X}\BL(f_K(\tau,x,\cdot))\le \widehat{b}.$$ Hence
  \begin{align*}
  &\Bigl|\int_0^t\int_Y\lt|V[\eta,\nu_{\cdot}](\tau,x,\phi)
  -V[\eta^K,\nu_{\cdot}](\tau,x,\phi)\rt|
\mathrm{d}(\nu^{K,x}_{\tau}(\phi)-\nu_{\tau}^x(\phi))\mathrm{d}\tau\Bigr|
\le \widehat{b}\int_0^td_{\sf BL}(\nu^{K,x}_{\tau},\nu_{\tau}^x)\rd\tau.\end{align*}

This further implies that
\begin{align*}
\gamma^K_x(t)\le& L_1\int_0^t\gamma^K_x(\tau)\rd\tau+b\int_0^t\beta^K_x(\tau)\rd\tau
+\widehat{b}\int_0^td_{\sf BL}(\nu^{K,x}_{\tau},\nu_{\tau}^x)\rd\tau+\int_0^t\zeta^K_x(\tau)\rd\tau.
\end{align*}
By Gronwall's inequality, we have
\begin{align*}
  \gamma^K_x(t)\le \textnormal{e}^{L_1t}\lt(b\int_0^t\beta_x(\tau)\rd\tau
  +\widehat{b}\int_0^td_{\sf BL}(\nu^{K,x}_{\tau},\nu_{\tau}^x)\rd\tau
  +\int_0^t\zeta^K_x(\tau)\rd\tau\rt)
  \end{align*}
Hence by \eqref{term-1}, \eqref{Eq-10}, \eqref{Eq-11}, and \eqref{Eq-Star-1}, we have for $t\in\cI$,
\begin{align*}
&d_{\sf BL}(\nu_t^x,\nu^{K,x}_t)\le\beta^K_x(t)+\gamma^K_x(t)\\
\le& \beta^K_x(t)+\textnormal{e}^{L_1t}\lt(b\int_0^t\beta_x(\tau)\rd\tau
+\widehat{b}\int_0^td_{\sf BL}(\nu^{K,x}_{\tau},\nu_{\tau}^x)\rd\tau
+\int_0^t\zeta^K_x(\tau)\rd\tau\rt)\\
\le&
L_{2,K}\|\nu_{\cdot}\|\int_0^t\textnormal{e}^{L_{1,K}(t-\tau)}d_{\infty}(\nu^K_{\tau},\nu_{\tau})\rd\tau\\
&+\textnormal{e}^{L_1t}b\int_0^t\beta_x(\tau)\rd\tau+\widehat{b}\textnormal{e}^{L_1t}\int_0^td_{\sf BL}(\nu^{K,x}_{\tau},\nu_{\tau}^x)\rd\tau
+\textnormal{e}^{L_1t}\int_0^t\zeta^K_x(\tau)\rd\tau\\
\le&(b-L_2)\|\nu_{\cdot}\|\int_0^t\textnormal{e}^{(b-L_1)(t-\tau)}d_{\infty}(\nu^K_{\tau},\nu_{\tau})\rd\tau\\
&+\textnormal{e}^{L_1t}b\int_0^t(b-L_2)\|\nu_{\cdot}\|\int_0^{\tau}\textnormal{e}^{(b-L_1)(\tau-s)}
d_{\infty}(\nu^K_{s},\nu_{s})\rd s\rd\tau\\
&+\widehat{b}\textnormal{e}^{L_1t}\int_0^td_{\sf BL}(\nu^{K,x}_{\tau},\nu_{\tau}^x)\rd\tau
+\textnormal{e}^{L_1t}\int_0^t\zeta^K_x(\tau)\rd\tau\\
\le&(b-L_2)\|\nu_{\cdot}\|(1+\textnormal{e}^{L_1t}bt)\int_0^t\textnormal{e}^{(b-L_1)(t-\tau)}d_{\infty}
(\nu^K_{\tau},\nu_{\tau})\rd\tau\\
&+\widehat{b}\textnormal{e}^{L_1t}\int_0^td_{\sf BL}(\nu^{K,x}_{\tau},\nu_{\tau}^x)\rd\tau
+\textnormal{e}^{L_1t}\int_0^t\zeta^K_x(\tau)\rd\tau\\
   \le&(b-L_2)\|\nu_{\cdot}\|(1+bT)\textnormal{e}^{L_1t}\int_0^t\textnormal{e}^{(b-L_1)(t-\tau)}d_{\infty}
   (\nu^K_{\tau},\nu_{\tau})\rd\tau\\
&+\widehat{b}\textnormal{e}^{L_1t}\int_0^td_{\sf BL}(\nu^{K,x}_{\tau},\nu_{\tau}^x)\rd\tau
+\textnormal{e}^{L_1t}\int_0^t\zeta^K_x(\tau)\rd\tau\\
\le& L_4\int_0^td_{\infty}(\nu_{\tau}^K,\nu_{\tau})\rd\tau+\textnormal{e}^{L_1T}\int_0^t\zeta_x^K(\tau)\rd\tau,
\end{align*}
where $L_4\coloneqq \textnormal{e}^{bT}(b-L_2)\|\nu_{\cdot}\|(1+bT)+\textnormal{e}^{L_1T}\widehat{b}$. By Gronwall's inequality,
\[d_{\infty}(\nu_t^K,\nu_t)\le\textnormal{e}^{L_4+L_1T}\sup_{x\in X}\int_0^t\zeta^{K}_{x}(\tau)\rd\tau.\]

To show $\lim_{K\to\infty}d_0(\nu^K_{\cdot},\nu_{\cdot})=0$, it suffices to show $$\lim_{K\to\infty}\sup_{x\in X}\int_0^T\zeta^{K}_{x}(\tau)\rd\tau=0.$$
For every $t\in\cI$, define $\widehat{\nu}_t\equiv\sup_{x\in X}\nu_t^x$:
\[\widehat{\nu}_t(E)=\sup_{x\in X}\nu_t^x(E),\quad \forall E\in\mathcal{B}(Y).\] Since $\nu_t\in \mathcal{B}(X,\cM_+(Y))$, it is easy to show that $\widehat{\nu}_t\in\cM_+(Y)$.
By Fatou's lemma, \begin{align*}
\sup_{x\in X}\zeta_x^K(\tau)=&\sup_{x\in X}\int_Y\lt|V[\eta,\nu_{\cdot}](\tau,x,\phi)
  -V[\eta^K,\nu_{\cdot}](\tau,x,\phi)\rt|
\mathrm{d}\nu_{\tau}^x(\phi)\\
\le&\sup_{x\in X}\int_Y\lt|V[\eta,\nu_{\cdot}](\tau,x,\phi)
  -V[\eta^K,\nu_{\cdot}](\tau,x,\phi)\rt|
\mathrm{d}\widehat{\nu}_{\tau}(\phi)\\
\le&\int_Y\sup_{x\in X}\lt|V[\eta,\nu_{\cdot}](\tau,x,\phi)
  -V[\eta^K,\nu_{\cdot}](\tau,x,\phi)\rt|
\mathrm{d}\widehat{\nu}_{\tau}(\phi).\end{align*}
Since $\nu_{\cdot}\in\cB(\cI,\cM_+(Y))$, we have $\widehat{\nu}_{\cdot}\in\cB(\cI,\cM_+(Y))$. Using analogous arguments as those for proving \cite[Proposition~3.2]{KX21}, we have \begin{align*}
&\sup_{x\in X}\int_0^T\zeta_x^K(\tau)\rd\tau\le\int_0^T\sup_{x\in X}\zeta_x^K(\tau)\rd\tau\\
\le&\int_0^T\int_Y\sup_{x\in X}\lt|V[\eta,\nu_{\cdot}](\tau,x,\phi)
  -V[\eta^K,\nu_{\cdot}](\tau,x,\phi)\rt|
\mathrm{d}\widehat{\nu}_{\tau}(\phi)\rd\tau\to0,\quad \text{as}\ K\to\infty.
\end{align*}
\end{enumerate}
\end{proof}

\end{document}